\documentclass{article}

\usepackage[latin1]{inputenc}
\usepackage{amssymb}
\usepackage{amsfonts}
\usepackage{amscd}
\usepackage{mathrsfs}
\usepackage[dvips]{graphicx}
\usepackage[T1]{fontenc}
\usepackage[english]{babel}
\usepackage{marvosym}
\usepackage{amsthm}
\usepackage{multicol}
\usepackage[all]{xy}

\usepackage{todonotes}

\usepackage{graphicx,tikz}

\usepackage{pb-diagram}

\usepackage{bm}

\usepackage{wasysym}

\usepackage{color}

\usetikzlibrary{decorations.pathreplacing,backgrounds,decorations.markings,arrows}
\tikzset{wei/.style={draw=red,double=red!40!white,double distance=1.5pt,thin}}
\tikzset{bdot/.style={fill,circle,color=blue,inner sep=3pt,outer sep=0}}
\tikzset{dir/.style={postaction={decorate,decoration={markings,mark=at position .8 with {\arrow[scale=1.3]{>}}}}}}

\theoremstyle{plain}
\newtheorem{thm}{Theorem}[section]
\newtheorem{lem}[thm]{Lemma}

\newtheorem{prop}[thm]{Proposition}

\newtheorem{coro}[thm]{Corollary}

\theoremstyle{definition}
\theoremstyle{remark}
\newtheorem{rk}[thm]{Remark}
\newtheorem{df}[thm]{Definition}
\newtheorem{ex}[thm]{Example}

\def\Pol{\mathcal{P}ol}

\def\ui{{\mathbf{i}}}
\def\uj{{\mathbf{j}}}

\def\bbC{\mathbb{C}}

\def\bbN{\mathbb{N}}

\def\bbQ{\mathbb{Q}}

\def\bbZ{\mathbb{Z}}

\def\scrF{\mathscr{F}}

\def\frakS{\mathfrak{S}}

\def\calC{\mathcal{C}}

\def\calF{\mathcal{F}}

\def\calO{\mathcal{O}}

\def\frakg{\mathfrak{g}}

\def\bfe{\mathbf{e}}

\def\bfk{\mathbf{k}}

\def\bfB{\mathbf{B}}

\def\leqs{\leqslant}

\newcommand\restr[2]{{% we make the whole thing an ordinary symbol
  \left.\kern-\nulldelimiterspace % automatically resize the bar with \right
  #1 % the function
  \vphantom{\big|} % pretend it's a little taller at normal size
  \right|_{#2} % this is the delimiter
  }}

\def\homo{\operatorname{\it \mathscr{H}\kern-.25em om}}
\def\ext{\operatorname{\it \mathscr{E}\kern-.25em xt}}
\def\edo{\operatorname{\it \mathscr{E}\kern-.25em nd}}
\def\der{\operatorname{\it \mathscr{D}\kern-.25em er}}

\def\mod{\mathrm{mod}}

\def\ad{\mathrm{ad}}

\def\Hom{\mathrm{Hom}}

\def\End{\mathrm{End}}

\def\Rep{\mathrm{Rep}}

\title{\Huge Categorical representations and KLR algebras}
\author{Ruslan Maksimau \\\\
Institut Montpelli\'erain Alexander Grothendieck,\\
Universit\'e de Montpellier,\\
CC051, Place Eug\`ene Bataillon, 34095 Montpellier, France,\\
            ruslmax@gmail.com, ruslan.maksimau@umontpellier.fr
}

\date{}

\renewcommand\thesubsection{\thesection\Alph{subsection}}

\begin{document}

\maketitle
\setcounter{tocdepth}{2}

\begin{abstract}

We prove that the KLR algebra associated with the cyclic quiver of length $e$ is a subquotient of the KLR algebra associated with the cyclic quiver of length $e+1$. We also give a geometric interpretation of this fact. This result has an important application in the theory of categorical representations. We prove that a category with an action of $\widetilde{\mathfrak{sl}}_{e+1}$ contains a subcategory with an action of $\widetilde{\mathfrak{sl}}_{e}$.
We also give generalizations of these results to more general quivers and Lie types.
\end{abstract}

\tableofcontents

\section{Introduction}

Consider the complex affine Lie algebra $\widetilde{\mathfrak{sl}}_e={\mathfrak{sl}}_e[t,t^{-1}]\oplus \bbC\mathbf{1}$. In this paper, we study categorical representations of $\widetilde{\mathfrak{sl}}_e$. Our goal is to relate the notion of a categorical representation of $\widetilde{\mathfrak{sl}}_e$ with the notion of a categorical representation of $\widetilde{\mathfrak{sl}}_{e+1}$. 

The Lie algebra $\widetilde{\mathfrak{sl}}_e$ has generators $e_i$, $f_i$ for $i\in[0,e-1]$. Let $\alpha_0,\cdots,\alpha_{e-1}$ be the simple roots of $\widetilde{\mathfrak{sl}}_e$. Fix $k\in[0,e-1]$. Consider the following inclusion of Lie algebras $\widetilde{\mathfrak{sl}}_e\subset\widetilde{\mathfrak{sl}}_{e+1}$

\begin{equation}
\label{ch3:eq_inclusion-affine-sl}
e_r\mapsto
\left\{\begin{array}{rl}
e_r &\mbox{ if }r\in[0,k-1],\\
{[e_k,e_{k+1}]} &\mbox{ if }r=k,\\
e_{r+1} &\mbox{ if }r\in[k+1,e-1],
\end{array}\right.
\end{equation}
$$
f_r\mapsto
\left\{\begin{array}{rl}
f_r &\mbox{ if }r\in[0,k-1],\\
{[f_{k+1},f_k]} &\mbox{ if }r=k,\\
f_{r+1} &\mbox{ if }r\in[k+1,e-1].
\end{array}\right.
$$

It is clear that each $\widetilde{\mathfrak{sl}}_{e+1}$-module can be restricted to the subalgebra $\widetilde{\mathfrak{sl}}_{e}$ of $\widetilde{\mathfrak{sl}}_{e+1}$. So it is natural to ask if we can do the same with categorical representations.

First, we recall the notion of a categorical representation. Let $\bfk$ be a field. Let $\calC$ be an abelian $\Hom$-finite $\bfk$-linear category that admits a direct sum decomposition $\calC=\bigoplus_{\mu\in\bbZ^{e}}\calC_\mu$. A categorical representation of $\widetilde{\mathfrak{sl}}_e$ in $\calC$ is a pair of biadjoint functors $E_i,F_i\colon \calC\to\calC$ for $i\in[0,e-1]$ satisfying a list of axioms. The main axiom is that for each positive integer $d$ there is an algebra homomorphism $R_d(A_{e-1}^{(1)})\to \End(F^d)^{\rm op}$, where $F=\bigoplus_{i=0}^{e-1}F_i$ and $R_d(A_{e-1}^{(1)})$ is the KLR algebra of rank $d$ associated with the quiver $A_{e-1}^{(1)}$ (i.e., with the cyclic quiver of length $e$).

%Now we explain our main result about categorical representations.
Let $\overline\calC$ be an abelian $\Hom$-finite $\bfk$-linear category. Assume that $\overline\calC=\bigoplus_{\mu\in\bbZ^{e+1}}\overline\calC_\mu$ has a structure of a categorical representation of $\widetilde{\mathfrak{sl}}_{e+1}$ with respect to functors $\overline E_i,\overline F_i$ for $i\in[0,e]$. We want to restrict the action of $\widetilde{\mathfrak{sl}}_{e+1}$ on $\overline\calC$ to $\widetilde{\mathfrak{sl}}_{e}$. The most obvious way to do this is to define new functors $E_i, F_i\colon \overline \calC\to\overline\calC$, $i\in [0,e-1]$ from the functors $\overline E_i, \overline F_i\colon \overline \calC\to\overline\calC$, $i\in [0,e]$ by the same formulas as in (\ref{ch3:eq_inclusion-affine-sl}). Of course, this makes no sense because the notion of a commutator of two functors does not exist. However, we are able to get a structure of a categorical representation on a subcategory $\calC\subset \overline\calC$ (and not on the category $\overline\calC$ itself). We do this in the following way.

Assume additionally that the category $\overline\calC_\mu$ is zero whenever $\mu$ has a negative entry.
For each $e$-tuple $\mu=(\mu_1,\cdots,\mu_e)\in\bbZ^e$ we consider the $(e+1)$-tuple $\overline\mu=(\mu_1,\cdots,\mu_k,0,\mu_{k+1},\cdots,\mu_e)$ and we set
$\calC_\mu=\overline\calC_{\overline\mu}$,
$$
\calC=\bigoplus_{\mu\in \bbZ^e}\calC_{\mu}.
$$
Next, consider the endofunctors of $\calC$ given by
$$
E_i=
\left\{
\begin{array}{lll}
\restr{\overline E_i}{\calC} &\mbox{ if } 0\leqslant i<k,\\
\restr{\overline E_{k}\overline E_{k+1}}{\calC} &\mbox{ if } i=k,\\
\restr{\overline E_{i+1}}{\calC} &\mbox{ if } k<i<e,
\end{array}
\right.
$$
$$
F_i=
\left\{
\begin{array}{lll}
\restr{\overline F_i}{\calC} &\mbox{ if } 0\leqslant i<k,\\
\restr{\overline F_{k+1}\overline F_k}{\calC} &\mbox{ if } i=k,\\
\restr{\overline F_{i+1}}{\calC} &\mbox{ if } k<i<e.
\end{array}
\right.
$$
The following theorem holds.

\smallskip
\begin{thm}
\label{ch3:thm_main-thm-categ-rep-int-art}
The category $\calC$ has the structure of a categorical representation of $\widetilde{\mathfrak{sl}}_{e}$ with respect to the functors
$E_0,\cdots, E_{e-1}$, $F_0,\cdots, F_{e-1}$.
\qed
\end{thm}

\smallskip
Let us explain our motivation for proving Theorem \ref{ch3:thm_main-thm-categ-rep-int-art} (see \cite{Mak-Zuck} for more details). Let $O^\nu_{-e}$ be the parabolic category $\calO$ for $\widehat{\mathfrak{gl}}_N={\mathfrak{gl}}_N[t,t^{-1}]\oplus \bbC\mathbf{1}\oplus\bbC\partial$ with parabolic type $\nu$ at level $-e-N$. By \cite{RSVV}, there is a categorical representation of $\widetilde{\mathfrak{sl}}_{e}$ in $O^\nu_{-e}$. Now we apply Theorem \ref{ch3:thm_main-thm-categ-rep-int-art} to $\overline\calC=O^\nu_{-(e+1)}$. It happens that in this case the subcategory $\calC\subset \overline\calC$ defined as above is equivalent to $O^\nu_{-e}$. This allows us to compare the categorical representations in the category $\calO$ for $\widehat{\mathfrak{gl}}_N$ for two different (negative) levels. 

A result similar to Theorem \ref{ch3:thm_main-thm-categ-rep-int-art} has also recently appeared in \cite{RW}. It is applied in \cite{RW} in the following way. It is known from \cite{ChRou} that there is a categorical representation of $\widetilde{\mathfrak{sl}}_p$ in the category $\Rep(GL_n(\overline{\mathbb{F}_p}))$ of finite dimensional algebraic representations of $GL_n(\overline{\mathbb{F}_p})$. The paper \cite{RW} uses this fact to construct a categorical representation of the Hecke category on the principal block $\Rep_0(GL_n(\overline{\mathbb{F}_p}))$ of $\Rep(GL_n(\overline{\mathbb{F}_p}))$ for $p>n$. Their proof is in two steps. First they show that the action of $\widetilde{\mathfrak{sl}}_p$ on $\Rep(GL_n(\overline{\mathbb{F}_p}))$ induces an action of $\widetilde{\mathfrak{sl}}_n$ on some full subcategory of $\Rep(GL_n(\overline{\mathbb{F}_p}))$. The second step is to show that the action of $\widetilde{\mathfrak{sl}}_n$ constructed on the first step induces an action of the Hecke category on $\Rep_0(GL_n(\overline{\mathbb{F}_p}))$. The first step of their proof is essentially $p-n$ consecutive applications of Theorem \ref{ch3:thm_main-thm-categ-rep-int-art}.

The main difficulty in proving Theorem \ref{ch3:thm_main-thm-categ-rep-int-art} is showing that the action of the KLR algebra $R_d(A_{e}^{(1)})$ on $\overline F^d$, where $\overline F=\bigoplus_{i=0}^{e}\overline F_i$, yields an action of the KLR algebra $R_d(A_{e-1}^{(1)})$ on $F^d$. So, to prove the theorem, we need to compare the KLR algebra $R_d(A_{e}^{(1)})$ with the KLR algebra $R_d(A_{e-1}^{(1)})$. This is done in Section \ref{ch3:sec_KLR-Hecke}.

We introduce the abbreviations $\Gamma=A_{e-1}^{(1)}$ and $\overline\Gamma=A_{e}^{(1)}$. Let $\alpha=\sum_{i=0}^{e-1}d_i\alpha_i$ be a dimension vector of the quiver $\Gamma$. We consider the dimension vector $\overline\alpha$ of $\overline\Gamma$ defined by
$$
\overline\alpha=\sum_{i=0}^{k}d_i\alpha_i+\sum_{i=k+1}^{e}d_{i-1}\alpha_i.
$$

Let $R_\alpha(\Gamma)$ and $R_{\overline\alpha}(\overline\Gamma)$ be the KLR algebras associated with the quivers $\Gamma$ and $\overline\Gamma$ and the dimension vectors $\alpha$ and $\overline\alpha$. The algebra $R_{\overline\alpha}(\overline\Gamma)$ contains idempotents $e(\ui)$ parameterized by certain sequences $\ui$ of vertices of $\overline\Gamma$. In Section \ref{ch3:subs_bal-quot} we consider some sets of such sequences $\overline I^{\overline\alpha}_{\rm ord}$ and $\overline I^{\overline\alpha}_{\rm un}$. Set $\bfe=\sum_{\ui\in \overline I^{\overline\alpha}_{\rm ord}}e(\ui)\in R_{\overline\alpha}(\overline\Gamma)$ and
$$
{S}_{\overline\alpha}(\overline\Gamma)=\bfe{R}_{\overline\alpha}(\overline\Gamma)\bfe/\sum_{\ui\in \overline I^{\overline\alpha}_{\rm un}}\bfe R_{\overline\alpha}(\overline\Gamma) e(\ui)R_{\overline\alpha}(\overline\Gamma)\bfe.
$$
The main result of Section \ref{ch3:sec_KLR-Hecke} is the following theorem.

\smallskip
\begin{thm}
\label{ch3:thm-isom_KLR_e_e+1_intro}
There is an algebra isomorphism $R_{\alpha}(\Gamma)\simeq S_{\overline\alpha}(\overline\Gamma)$.
\qed
\end{thm}

The paper has the following structure. In Section \ref{ch3:sec_KLR-Hecke} we study KLR algebras. In particular, we prove Theorem \ref{ch3:thm-isom_KLR_e_e+1_intro}. In Section \ref{ch3:sec_catO} we study categorical representations. We prove our main result about categorical representations (Theorem \ref{ch3:thm_main-thm-categ-rep-int-art}). We also generalize this theorem to arbitrary symmetric Kac-Moody Lie algebras. In Appendix \ref{ch3:app-geom_constr} we give a geometric construction of the isomorphism in Theorem \ref{ch3:thm-isom_KLR_e_e+1_intro}. In Appendix \ref{ch3:app-local_ring}, we give some versions of Theorems \ref{ch3:thm_main-thm-categ-rep-int-art} and \ref{ch3:thm-isom_KLR_e_e+1_intro} in type $A$ over a local ring.

It is important to emphasize the relation between the present paper and \cite{Mak-Zuck}. That preprint contains (an earlier version of) the results of the present paper and an application of these results to the category $\calO$ for $\widehat{\mathfrak{gl}}_N$. The preprint \cite{Mak-Zuck} is expected to be published as two different papers. The present paper is the first of them. It contains the results of the preprint \cite{Mak-Zuck} about KLR algebras and categorical representations. The second paper will give an application of the results of the first paper to the affine category $\calO$.

\section{KLR algebras}
\label{ch3:sec_KLR-Hecke}

For a noetherian ring $A$ we denote by $\mod(A)$ the abelian category of left finitely generated $A$-modules. We denote by $\bbN$ the set of non-negative integers.

\subsection{Kac-Moody algebras associated with a quiver}
\label{ch3:subs_KM-quiv}
Let $\Gamma=(I,H)$ be a quiver without $1$-loops with the set of vertices $I$ and the set of arrows $H$. For $i,j\in I$ let $h_{i,j}$ be the number of arrows from $i$ to $j$ and set also $a_{i,j}=2\delta_{i,j}-h_{i,j}-h_{j,i}$. Let $\frakg_I$ be the Kac-Moody algebra over $\bbC$ associated with the matrix $(a_{i,j})$. Denote by $e_i$, $f_i$ for $i\in I$ the Serre generators of $\frakg_I$.

\begin{rk}
\label{ch3:rk-KM-Lie}
By the Kac-Moody Lie algebra associated with the Cartan matrix $(a_{i,j})$ we understand the Lie algebra with the set of generators $e_i$, $f_i$, $h_i$, $i\in I$, modulo the following defining relations
$$
\begin{array}{rclr}
[h_i,h_j]&=&0,\\

[h_i,e_j]&=&a_{i,j}e_j,\\

[h_i,f_j]&=&-a_{i,j}e_j,\\

[e_i,f_j]&=&\delta_{i,j}h_i,\\

(\ad (e_i))^{1-a_{i,j}}(e_j)&=&0, &~i\ne j,\\

(\ad (f_i))^{1-a_{i,j}}(f_j)&=&0, &~i\ne j.\\
\end{array}
$$

In particular, if $(a_{i,j})$ is the affine Cartan matrix of type $A_{e-1}^{(1)}$, then we get the Lie algebra $\widetilde{\mathfrak{sl}}_e(\bbC)=\mathfrak{sl}_e(\bbC)\otimes\bbC[t,t^{-1}]\oplus\bbC \bm{1}$ (not $\mathfrak{sl}_e(\bbC)\otimes\bbC[t,t^{-1}]\oplus\bbC \bm{1}\oplus \bbC\partial$).

\end{rk}

\medskip
For each $i\in I$, let $\alpha_i$ be the simple root corresponding to $e_i$. Set
$$
Q_I=\bigoplus_{i\in I}\bbZ\alpha_i,\quad Q^+_I=\bigoplus_{i\in I}\bbN\alpha_i.
$$

For $\alpha=\sum_{i\in I}d_i\alpha_i\in Q^+_I$ denote by $|\alpha|$ its height, i.e., we have $|\alpha|=\sum_{i\in I}d_i$. Set $I^\alpha=\{\ui=(i_1,\cdots,i_{|\alpha|})\in I^{|\alpha|};~ \sum_{r=1}^{|\alpha|}\alpha_{i_r}=\alpha\}$.

\subsection{Doubled quiver}
\label{ch3:subs_not-quiv-I-Ibar}
Let $\Gamma=(I,H)$ be a quiver without $1$-loops. Fix a decomposition $I=I_0\sqcup I_1$ such that there are no arrows between the vertices in $I_1$. In this section we define a \emph{doubled quiver} $\overline\Gamma=(\overline I,\overline H)$ associated with $(\Gamma,I_0,I_1)$. The idea is to "double" each vertex in the set $I_1$ (we do not touch the vertices from $I_0$). We replace each vertex $i\in I_1$ by a couple of vertices $i^1$ and $i^2$ with an arrow $i^1\to i^2$. Each arrow entering $i$ should be replaced by an arrow entering $i^1$, each arrow coming from $i$ should be replaced by an arrow coming from $i^2$.

Now we describe the construction of $\overline\Gamma=(\overline I,\overline H)$ formally. Let $\overline I_0$ be a set that is in bijection with $I_0$. Let $i^0$ be the element of $\overline I_0$ associated with an element $i\in I_0$. Similarly, let $\overline I_1$ and  $\overline I_2$ be sets that are in bijection with $I_1$. Denote by $i^1$ and $i^2$ the elements of $\overline I_1$ and  $\overline I_2$ respectively that correspond to an element $i\in I_1$. Put $\overline I=\overline I_0\sqcup\overline I_1\sqcup \overline I_2$.
We define $\overline H$ in the following way. The set $\overline H$ contains $4$ types of arrows:
\begin{itemize}
    \item[\textbullet] an arrow $i^0\to j^0$ for each arrow $i\to j$ in $H$ with $i,j\in I_0$,
    \item[\textbullet] an arrow $i^0\to j^1$ for each arrow $i\to j$ in $H$ with $i\in I_0,j\in I_1$,
    \item[\textbullet] an arrow $i^2\to j^0$ for each arrow $i\to j$ in $H$ with $i\in I_1,j\in I_0$,
    \item[\textbullet] an arrow $i^1\to i^2$ for each vertex $i\in I_1$.
\end{itemize}
%\begin{array}{ll}
%h_{i^0,j^0}=h_{i,j},\quad \forall i,j\in I_0,\\
%h_{i^0,j^1}
%\end{array}
%$$

%We consider the map
%$$
%\Delta\colon I\to \overline I,~i\mapsto
%\left\{\begin{array}{ll}
%i^0 &\mbox{ for } i\in I_0\\
%i^1 &\mbox{ for } i\in I_1.
%\end{array}\right.
%$$

Set $I^\infty=\coprod_{d\in \bbN} I^d$, $\overline I^\infty=\coprod_{d\in \bbN} \overline I^d$, where $I^d$, $\overline I^d$ are the cartesian products.
The concatenation yields a monoid structure on $I^\infty$ and $\overline I^\infty$.
Let $\phi\colon I^\infty\to \overline I^\infty$ be the unique morphism of monoids such that for $i\in I\subset I^\infty$ we have
$$
\phi(i)=
\left\{\begin{array}{ll}
i^0 &\mbox{ if }i\in I_0,\\
(i^1,i^2) &\mbox{ if } i\in I_1.
\end{array}\right.
$$

There is a unique $\bbZ$-linear map $\phi\colon Q_I\to Q_{\overline I}$ such that $\phi(I^\alpha)\subset {\overline I}^{\phi(\alpha)}$ for each $\alpha\in Q^+_I$. It is given by
%\begin{equation}
%\label{ch3:eq_phi(alpha)}
$$
\phi(\alpha_{i})=
\left\{\begin{array}{ll}
\alpha_{i^0} &\mbox{ if }i\in I_0,\\
\alpha_{i^1}+\alpha_{i^2}&\mbox{ if }i\in I_1.
\end{array}\right.
$$
%\end{equation}

\subsection{KLR algebras}
\label{ch3:subs_KLR}

Let $\bfk$ be a field. Let $\Gamma=(I,H)$ be a quiver without $1$-loops. For $r\in[1,d-1]$ let $s_r$ be the transposition $(r,r+1)\in\frakS_d$. For $\ui=(i_1,\cdots,i_d)\in I^d$ set $s_r(\ui)=(i_1,\cdots,i_{r-1},i_{r+1},i_r,i_{r+2},\cdots,i_d)$.
%Consider the following matrix of polynomials.
For $i,j\in I$ we set
$$
Q_{i,j}(u,v)=
\left\{\begin{array}{ll}
0& \mbox{ if }i=j,\\
(v-u)^{h_{i,j}}(u-v)^{h_{j,i}}& \mbox{ else}.
\end{array}\right.
$$

\smallskip
\begin{df}
\label{ch3:def_KLR}
Assume that the quiver $\Gamma$ is finite. The \emph{KLR-algebra} $R_{d,\bfk}(\Gamma)$ is the $\bfk$-algebra with the set of generators $\tau_1,\cdots,\tau_{d-1},x_1,\cdots,x_d,e(\ui)$ where $\ui\in I^d$,
modulo the following defining relations
\begin{itemize}
    \item[\textbullet] $e(\ui)e(\uj)=\delta_{\ui,\uj}e(\ui)$,
    \item[\textbullet] $\sum_{\ui\in I^d}e(\ui)=1$,
    \item[\textbullet] $x_re(\ui)=e(\ui)x_r$,
    \item[\textbullet] $\tau_re(\ui)=e(s_r(\ui))\tau_r$,
    \item[\textbullet] $x_rx_s=x_sx_r$,
    \item[\textbullet] $\tau_rx_{r+1}e(\ui)=(x_r\tau_r+\delta_{i_r,i_{r+1}})e(\ui)$,
    \item[\textbullet] $x_{r+1}\tau_re(\ui)=(\tau_rx_r+\delta_{i_r,i_{r+1}})e(\ui)$,
    \item[\textbullet] $\tau_rx_s=x_s\tau_r$, if $s\ne r,r+1$,
    \item[\textbullet] $\tau_r\tau_s=\tau_s\tau_r$, if $|r-s|>1$,
    \item[\textbullet] $\tau_r^2e(\ui)=
\left\{
\begin{array}{ll}
0 &\mbox{if } i_r=i_{r+1},\\
Q_{i_r,i_{r+1}}(x_r,x_{r+1})e(\ui) &\mbox{else},
\end{array}
\right.
$
    \item[\textbullet] $(\tau_r\tau_{r+1}\tau_r-\tau_{r+1}\tau_r\tau_{r+1})e(\ui)=$

$\left\{
\begin{array}{ll}
(x_{r+2}-x_{r})^{-1}(Q_{i_{r},i_{r+1}}(x_{r+2},x_{r+1})-Q_{i_r,i_{r+1}}(x_{r},x_{r+1}))e(\ui) &\mbox{if }i_r=i_{r+2},\\
0 &\mbox{else},
\end{array}
\right.$
\end{itemize}
for each $\ui$, $\uj$, $r$ and $s$. We may write $R_{d,\bfk}=R_{d,\bfk}(\Gamma)$. The algebra $R_{d,\bfk}$ admits a $\bbZ$-grading such that $\deg e(\ui)=0$, $\deg x_r=2$ and $\deg\tau_se(\ui)=-a_{i_s,i_{s+1}}$, for each $1\leqslant r\leqslant d$, $1\leqslant s< d$ and $\ui\in I^d$.
\end{df}

\smallskip
%Recall the notations from Section \ref{ch3:subs_KM-quiv}.
For each $\alpha\in Q^+_I$ such that $|\alpha|=d$ set $e(\alpha)=\sum_{\ui\in I^\alpha} e(\ui)\in R_{d,\bfk}$. It is a homogeneous central idempotent of degree zero. We have the following decomposition into a sum of unitary $\bfk$-algebras
$R_{d,\bfk}=\bigoplus_{|\alpha|=d}R_{\alpha,\bfk}$, where $R_{\alpha,\bfk}=e(\alpha)R_{d,\bfk}$.

Let $\bfk^{(I)}_d$ be the direct sum of copies of the ring $\bfk_d[x]:=\bfk[x_1,\cdots,x_d]$ labelled by $I^d$.
We write
\begin{equation}
\label{ch3:eq_k^I}
\bfk^{(I)}_d=\bigoplus_{\ui\in I^d}\bfk_d[x]e(\ui),
\end{equation}
where $e(\ui)$ is the idempotent of the ring $\bfk^{(I)}_d$ projecting to the component $\ui$. A polynomial in $\bfk_d[x]$ can be considered as an element of $\bfk^{(I)}_d$ via the diagonal inclusion.
For each $i,j\in I$ fix a polynomial $P_{i,j}(u,v)$ such that we have $Q_{i,j}(u,v)=P_{i,j}(u,v)P_{j,i}(v,u)$.

\smallskip
Denote by $\partial_r$ the Demazure operator on $\bfk_d[x]$, i.e., we have 
$$\partial_r(f)=(x_r-x_{r+1})^{-1}(s_r(f)-f).$$
The following is proved in \cite[Sec.~3.2]{Rouq-2KM}.

\begin{prop}
\label{ch3:prop_faith-rep-KLR}
The algebra $R_{d,\bfk}$ has a faithful representation in the vector space $\bfk^{(I)}_d$ such that the element $e(\ui)$ acts by the projection to $\bfk^{(I)}_de(\ui)$, the element $x_r$ acts by multiplication by $x_r$ and such that for $f\in\bfk_d[x]$ we have
\begin{equation}
\label{ch3:eq_action-on-polyn}
\tau_r\cdot fe(\ui)=
\left\{\begin{array}{ll}
\partial_r(f)e(\ui) &\mbox{ if }i_r=i_{r+1},\\
P_{i_r,i_{r+1}}(x_{r+1},x_r)s_r(f)e(s_r(\ui))&\mbox{ otherwise}.\\
\end{array}\right.
\end{equation}
\end{prop}

We will always choose $P_{i,j}$ in the following way:
$$
P_{i,j}(u,v)=(u-v)^{h_{j,i}}.
$$

\smallskip
\begin{rk}
\label{ch3:rk_basis-KLR}
There is an explicit construction of a basis of a KLR algebra (see \cite[Thm.~2.5]{KL}).
Assume $\ui,\uj\in I^\alpha$. Set $\frakS_{\ui,\uj}=\{w\in \frakS_d;~ w(\ui)=\uj\}$. For each permutation $w\in \frakS_{\ui,\uj}$ fix a reduced expression $w=s_{p_1}\cdots s_{p_r}$ and set $\tau_w=\tau_{p_1}\cdots\tau_{p_r}$.
Then the vector space $e(\uj)R_{\alpha,\bfk} e(\ui)$ has a basis $\{\tau_wx_1^{a_1}\cdots x_d^{a_d}e(\ui);~w\in\frakS_{\ui,\uj}, a_1,\cdots,a_d\in\bbN$\}. Note that the element $\tau_w$ depends on the reduced expression of $w$. %and we get a basis of $e(\uj)R_{\alpha,\bfk} e(\ui)$ for each combination of choices.
Moreover, if we change the reduced expression of $w$, then the element $\tau_we(\ui)$ is changed only by a linear combination of monomials of the form $\tau_{q_1}\cdots\tau_{q_t}x_1^{b_1}\cdots x_d^{b_d}e(\ui)$
with $t<\ell(w)$. Note also that if $s_{p_1}\cdots s_{p_r}$ is not a reduced expression, then the element $\tau_{p_1}\cdots\tau_{p_r}e(\ui)$ may be written as a linear combination of monomials of the form $\tau_{q_1}\cdots\tau_{q_t}x_1^{b_1}\cdots x_d^{b_d}e(\ui)$
with $t<r$. Moreover, in both situations above, the linear combination can be chosen in such a way that for each monomial $\tau_{q_1}\cdots\tau_{q_t}x_1^{b_1}\cdots x_d^{b_d}e(\ui)$ in the linear combination, the expression  $s_{q_1}\cdots s_{q_t}$ is reduced. 
\end{rk}

\begin{rk}
\label{ch3:rem_KLR-infinite}
The algebra $R_{d,\bfk}$ in Definition \ref{ch3:def_KLR} is well-defined only for a finite quiver because of the second relation. However, the algebra $R_{\alpha,\bfk}$ is well-defined even if the quiver is infinite because each $\alpha$ uses a finite set of vertices. Thus, for an infinite quiver we can define $R_{d,\bfk}$ as $R_{d,\bfk}=\bigoplus_{|\alpha|=d}R_{\alpha,\bfk}$. Hovewer, in this case the algebra $R_{d,\bfk}$ is not unitary. 
\end{rk}

\subsection{Balanced KLR algebras}
\label{ch3:subs_bal-quot}
From now on the quiver $\Gamma$ is assumed to be finite. Fix a decomposition $I=I_0\sqcup I_1$ as in Section \ref{ch3:subs_not-quiv-I-Ibar} and consider the quiver $\overline\Gamma=(\overline I,\overline H)$ as in Section \ref{ch3:subs_not-quiv-I-Ibar}. Recall the decomposition $\overline I=\overline I_0\sqcup \overline I_1\sqcup \overline I_2$. In this section we work with the KLR algebra associated with the quiver $\overline\Gamma$. %We will identify the set $Q_I^+$ with a subset of $Q^+_{\overline I}$ via the map $\phi$ from Section \ref{ch3:subs_not-quiv-I-Ibar}.

We say that a sequence $\ui=(i_1,i_2,\cdots,i_d)\in \overline I^d$ is \emph{unordered} if there is an index $r\in[1,d]$ such that the number of elements from $\overline I_2$ in the sequence $(i_1,i_2,\cdots,i_r)$ is strictly greater than the number of elements from $\overline I_1$. We say that it is \emph{well-ordered} if for each index $a$ such that $i_a=i^1$ for some $i\in I_1$, we have $a<d$ and $i_{a+1}=i^2$. %We say that a sequence $\ui=(i_1,\cdots,i_d)\in \overline I^d$ is \emph{almost ordered} is its subsequence formed by elements of $\overline I_1$ and $\overline I_2$ (we skip the elements of $\overline I_0$) is well-ordered.
We denote by $\overline I^\alpha_{\rm ord}$ and $\overline I^\alpha_{\rm un}$ the subsets of well-ordered and unordered sequences in $\overline I^\alpha$ respectively.

%\smallskip
%\begin{df}
%\label{ch3:def_bal-quot-gen}
%For $\overline\alpha\in Q^+_{\overline I}$, the \emph{balanced quotient} $\widetilde{S}_{\overline\alpha,\bfk}(\overline\Gamma)$ of $R_{\overline\alpha,\bfk}(\overline\Gamma)$ is the algebra
%$$
%R_{\overline\alpha,\bfk}(\overline\Gamma)/\sum_{\ui\in \overline I^{\overline\alpha}_{\rm un}}R_{\overline\alpha,\bfk}(\overline\Gamma) e(\ui)R_{\overline\alpha,\bfk}(\overline\Gamma).
%$$
%\end{df}

The map $\phi$ from Section \ref{ch3:subs_not-quiv-I-Ibar} yields a bijection
$$
\phi\colon Q^+_{I}\to\{\alpha=\sum_{i\in \overline I}d_i\alpha_i\in Q^+_{\overline I};~d_{i^1}=d_{i^2},~\forall i\in I_1\},\quad \alpha\mapsto\overline\alpha.
$$
Fix $\alpha\in Q^+_I$.
Set $\bfe=\sum_{\ui\in \overline I^{\overline\alpha}_{\rm ord}}e(\ui)\in R_{\overline\alpha,\bfk}(\overline\Gamma)$.

%We also denote by $\bfe$ the image of $\bfe$ in $\widetilde{S}_{\overline\alpha,\bfk}(\overline\Gamma)$.

\smallskip
\begin{df}
\label{ch3:def_bal-KLR}
For $\alpha\in Q^+_{I}$, the \emph{balanced KLR algebra} is the algebra
$$
{S}_{\overline\alpha,\bfk}(\overline\Gamma)=\bfe{R}_{\overline\alpha,\bfk}(\overline\Gamma)\bfe/\sum_{\ui\in \overline I^{\overline\alpha}_{\rm un}}\bfe R_{\overline\alpha,\bfk}(\overline\Gamma) e(\ui)R_{\overline\alpha,\bfk}(\overline\Gamma)\bfe.
$$
\end{df}
%In other words we have ${S}_{\overline\alpha,\bfk}(\overline\Gamma)=\bfe\widetilde{S}_{\overline\alpha,\bfk}(\overline\Gamma)\bfe$.

We may write ${S}_{\overline\alpha,\bfk}(\overline\Gamma)={S}_{\overline\alpha,\bfk}$.

\smallskip
\begin{rk}
\label{ch3:rem_xa=xa+1}
Assume that $\ui=(i_1,\cdots,i_d)\in \overline I^{\overline\alpha}_{\rm ord}$. Let $a$ be an index such that $i_a\in \overline I_1$. We have the relation $\tau_a^2e(\ui)=(x_{a+1}-x_a)e(\ui)$ in $R_{\overline\alpha,\bfk}$. Moreover, we have $\tau_a^2e(\ui)=\tau_ae(s_a(\ui))\tau_ae(\ui)$ and $s_a(\ui)$ is unordered. Thus we have $x_ae(\ui)=x_{a+1}e(\ui)$ in $S_{\overline\alpha,\bfk}$.
\end{rk}

%Our goal is to relate the KLR algebras with parameters $e$ and $e+1$. From now on we work with quivers $\Gamma=(I,H)$ and $\overline\Gamma=(\overline I,\overline H)$ from Section \ref{ch3:subs_not-e-e+1}. We will often use the notation introduced in Section \ref{ch3:subs_not-e-e+1}.

\subsection{The polynomial representation of $S_{\overline\alpha,\bfk}$}
\label{ch3:subs_pol-rep-BKLR}
We assume $\alpha=\sum_{i\in I}d_i\alpha_i\in Q^+_{I}$.
Let $\ui=(i_1,\cdots,i_d)\in \overline I^{\overline\alpha}_{\rm ord}$. Denote by $J(\ui)$ the ideal of the polynomial ring $\bfk_d[x]e(\ui)\subset \bfk^{(\overline I)}_d$ generated by the set
$$
\{(x_r-x_{r+1})e(\ui);~i_r\in \overline I_1\}.
$$

\smallskip
\begin{lem}
\label{ch3:lem_unord-mapsto-ideal}
Assume that $\ui\in \overline I^{\overline\alpha}_{\rm ord}$ and $\uj\in \overline I^{\overline\alpha}_{\rm un}$. Then each element of $e(\ui)R_{\overline\alpha,\bfk} e(\uj)$ maps $\bfk_d[x]e(\uj)$ to $J(\ui)$.
\end{lem}
\begin{proof}[Proof]
We will prove by induction on $k$ that for all $\ui\in \overline I^{\overline\alpha}_{\rm ord}$ and $\uj\in \overline I^{\overline\alpha}_{\rm un}$ and all $p_1,\cdots,p_k$ such that the permutation $w=s_{p_1}\cdots s_{p_k}\in\frakS_d$ satisfies $w(\uj)=\ui$, the monomial $\tau_{p_1}\cdots\tau_{p_k}$  maps $\bfk_d[x]e(\uj)$ to $J(\ui)$. %We do the induction by the length $k$.

Assume $k=1$. Write $p=p_1$. Let us write $\ui=(i_1,\cdots,i_d)$ and $\uj=(j_1,\cdots,j_d)$. Then we have $\ui=s_p(\uj)$. By assumptions on $\ui$ and $\uj$ we know that there exists $i\in I_1$ such that $i_{p}=j_{p+1}=i^1$ and $i_{p+1}=j_{p}=i^2$. In this case the statement is obvious because $\tau_p$ maps $fe(\uj)\in\bfk_d[x]e(\uj)$ to $(x_{p+1}-x_{p})s_p(f)e(\ui)$ by (\ref{ch3:eq_action-on-polyn}).

Now consider a monomial $\tau_{p_1}\cdots\tau_{p_k}$ such that the permutation $w=s_{p_1}\cdots s_{p_k}$ satisfies $w(\uj)=\ui$ and assume that the statement is true for all such monomials of smaller length. By assumptions on $\ui$ and $\uj$ there is an index $r\in [1,d]$ such that $i_r=i^1$ for some $i\in I_1$ and $w^{-1}(r+1)<w^{-1}(r)$. Thus $w$ has a reduced expression of the form $w=s_rs_{r_1}\cdots s_{r_h}$. This implies that $\tau_{p_1}\cdots\tau_{p_k}e(\uj)$ is equal to a monomial of the form $\tau_{r}\tau_{r_1}\cdots\tau_{r_h} e(\uj)$ modulo monomials of the form $\tau_{q_1}\cdots\tau_{q_t}x_1^{b_1}\cdots x_d^{b_d}e(\uj)$ with $t<k$, see Remark \ref{ch3:rk_basis-KLR}. As the sequence $s_r(\ui)$ is unordered, the case $k=1$ and the induction hypothesis imply the statement.
\end{proof}

\smallskip
\begin{lem}
\label{ch3:lem_KLR-pres-ideal}
Assume that $\ui,\uj\in \overline I^{\overline\alpha}_{\rm ord}$. Then each element of $e(\ui)R_{\overline\alpha,\bfk} e(\uj)$ maps $J(\uj)$ into $J(\ui)$.
\end{lem}
\begin{proof}[Proof]
Take $y\in e(\ui)R_{\overline\alpha,\bfk} e(\uj)$. We must prove that for each $r\in[1,d]$ such that $j_r=i^1$ for some $i\in I_1$ and each $f\in\bfk_{d}[x]$ we have $y((x_{r}-x_{r+1})fe(\uj))\in J(\ui)$. We have $(x_{r}-x_{r+1})fe(\uj)=-\tau_{r}^2(fe(\ui))$ (see Remark \ref{ch3:rem_xa=xa+1}). This implies
$$
y((x_{r}-x_{r+1})fe(\uj))=-y\tau_r^2(fe(\uj))=-y\tau_re(s_r(\uj))(\tau_r(fe(\uj))).
$$
Thus Lemma \ref{ch3:lem_unord-mapsto-ideal} implies the statement because the sequence $s_r(\uj)$ is unordered.
\end{proof}

%Set $\bfe=\sum_{\ui\in \overline I^\alpha_{\rm ord}}e(\uj)$. The algebra $\bfe R_{\alpha,\bfk}\bfe$ has a representation in the vector space
%$$
%\bfk_\alpha[x]_{\mathrm{ord}}=\bigoplus_{\ui\in \overline I^\alpha_{\rm ord}}\bfk[x_1,\cdots,x_d]e(\ui).
%$$
The representation of $R_{\overline\alpha,\bfk}$ on
$$
\bfk_{\overline\alpha}^{(\overline I)}:=\bigoplus_{\ui\in \overline I^{\overline\alpha}}\bfk_{|\overline\alpha|}[x]e(\ui)
$$
yields a representation of $\bfe R_{\overline\alpha,\bfk}\bfe$ on
$$
\bfk_{\overline\alpha,{\mathrm{ord}}}^{(\overline I)}:=\bigoplus_{\ui\in \overline I^{\overline\alpha}_{\rm ord}}\bfk_{|\overline\alpha|}[x]e(\ui).
$$

Set $J_{\overline\alpha,\mathrm{ord}}=\bigoplus_{\ui\in\overline I^{\overline\alpha}_{\rm ord}}J(\ui)$. From Lemmas \ref{ch3:lem_unord-mapsto-ideal} and \ref{ch3:lem_KLR-pres-ideal} we deduce the following.

\smallskip
\begin{lem}
\label{ch3:lem_pol-rep-of-S-defined+faith}
The representation of $R_{\overline\alpha,\bfk}$ on $\bfk_{\overline\alpha}^{(\overline I)}$ factors through a representation of $S_{\overline\alpha,\bfk}$ on $\bfk_{\overline\alpha,{\mathrm{ord}}}^{(\overline I)}/J_{\overline\alpha,\rm ord}$. This representation is faithful.
\end{lem}
\begin{proof}[Proof]
The faithfulness is proved in the proof of Theorem \ref{ch3:thm_KLR-e-e+1}.
\end{proof}

\subsection{The comparison of the polynomial representations}
\label{ch3:subs_comp-pol-reps}
Fix $\alpha\in Q^+_I$. Set $d=|\alpha|$ and $\overline d=|\overline\alpha|$. For each sequence $\ui=(i_1,\cdots,i_d)\in I^\alpha$ and $r\in[1,d]$ we denote by $r'$ or $r'_\ui$ the positive integer such that $r'-1$ is the length of the sequence $\phi(i_1,\cdots,i_{r-1})\in \overline I^\infty$.

For $r\in[1,d]$ (resp. $r\in[1,d-1]$) consider the element $x^*_{r}\in S_{\overline\alpha,\bfk}$ (resp. $\tau^*_{r}\in S_{\overline\alpha,\bfk}$) such that for each $\ui\in I^\alpha$ we have
$$
x^*_re(\phi(\ui))= x_{r'}e(\phi(\ui)),\\
$$
$$
\tau^*_re(\phi(\ui))=
\left\{\begin{array}{ll}
\tau_{r'}e(\phi(\ui)),& \mbox{ if }i_r,i_{r+1}\in I_0,\\
\tau_{r'}\tau_{r'+1}e(\phi(\ui))& \mbox{ if }i_r\in I_1,i_{r+1}\in I_0,\\
\tau_{r'+1}\tau_{r'}e(\phi(\ui))& \mbox{ if }i_r\in I_0,i_{r+1}\in I_1,\\
\tau_{r'+1}\tau_{r'+2}\tau_{r'}\tau_{r'+1}e(\phi(\ui))& \mbox{ if }i_r,i_{r+1}\in I_1,i_r\ne i_{r+1},\\
-\tau_{r'+1}\tau_{r'+2}\tau_{r'}\tau_{r'+1}e(\phi(\ui))& \mbox{ if }i_r=i_{r+1}\in I_1.\\
\end{array}\right.
$$

%In this section we construct an isomorphism of algebras $R_\alpha\simeq S_{\overline\alpha}$ by comparing their polynomial representations.

For each $\ui\in I^\alpha$ we have the algebra isomorphism
$$
\bfk_d[x]e(\ui)\simeq \bfk_{\overline d}[x]e(\phi(\ui))/J(\phi(\ui)),\quad x_re(\ui)\mapsto x_{r'}e(\phi(\ui)).
$$
We will always identify $\bfk^{(I)}_\alpha$ with $\bfk^{(\overline I)}_{\overline \alpha,\mathrm{ord}}/J_{\overline\alpha,\mathrm{ord}}$ via this isomorphism.

\smallskip
\begin{lem}
\label{ch3:lem_comp-operators}
The action of the elements $e(\ui)$, $x_re(\ui)$ and $\tau_re(\ui)$ of $R_{\alpha,\bfk}$ on $\bfk^{(I)}_\alpha$ is the same as the action of the elements $e(\phi(\ui))$, $x^*_{r}e(\phi(\ui))$ and $\tau^*_{r}e(\phi(\ui))$ of $S_{\overline\alpha,\bfk}$ on $\bfk^{(\overline I)}_{\overline \alpha,\mathrm{ord}}/J_{\overline\alpha,\mathrm{ord}}$.
\end{lem}

\begin{proof}[Proof]
The proof is based on the observation that by construction for each $i\in I_1$ and $j\in I_0$ we have
\begin{equation}
\label{ch3:eq_P10}
P_{i^1,j^0}(u,v)P_{i^2,j^0}(u,v)=P_{i,j}(u,v),
\end{equation}
$$
%\label{ch3:eq_P01}
P_{j^0,i^1}(u,v)P_{j^0,i^2}(u,v)=P_{j,i}(u,v).
$$

For each $\ui\in I^\alpha$, we write $\phi(\ui)=(i'_1,i'_2,\cdots,i'_{\overline d})$.
The only difficult part concerns the operator $\tau_re(\ui)$ when at least one of the elements $i_r$ or $i_{r+1}$ is in $I_1$. 
Assume that $i_r\in I_1$ and $i_{r+1}\in I_0$. In this case we have
$$
i'_{r'}=(i_r)^1\in \overline I_1,\quad i'_{r'+1}=(i_r)^2\in \overline I_2, \quad i'_{r'+2}=(i_{r+1})^0\in \overline I_0.
$$
In particular, the element $i'_{r'+2}$ is different from $i'_{r'}$ and $i'_{r'+1}$. Then, by (\ref{ch3:eq_action-on-polyn}), for each $f\in\bfk_{\overline d}[x]$ the element $\tau^*_re(\phi(\ui))=\tau_{r'}\tau_{r'+1}e(\phi(\ui))$ maps $fe(\phi(\ui))\in \bfk^{(\overline I)}_{\overline \alpha,\mathrm{ord}}/J_{\overline\alpha,\mathrm{ord}}$ to
$$
\begin{array}{lll}
&&P_{i'_{r'},i'_{r'+2}}(x_{r'+1},x_{r'})s_{r'}\left(P_{i'_{r'+1},i'_{r'+2}}(x_{r'+2},x_{r'+1})s_{r'+1}(f)\right)e(s_{r'}s_{r'+1}(\phi(\ui)))\\
&=&P_{i'_{r'},i'_{r'+2}}(x_{r'+1},x_{r'})P_{i'_{r'+1},i'_{r'+2}}(x_{r'+2},x_{r'})s_{r'}s_{r'+1}(f)e(\phi(s_r(\ui)))\\
&=&P_{i_r,i_{r+1}}(x_{r'+1},x_{r'})s_{r'}s_{r'+1}(f)e(\phi(s_r(\ui))),
\end{array}
$$
where the last equality holds by (\ref{ch3:eq_P10}).
Thus we see that the action of $\tau^*_re(\phi(\ui))$ on the polynomial representation is the same as the action of $\tau_re(\ui)$. The case when $i_r\in I_0$ and $i_{r+1}\in I_1$ can be done similarly.

Assume now that $i_r\ne i_{r+1}$ are both in $I_1$. By the assumption on the quiver $\Gamma$ (see Section \ref{ch3:subs_not-quiv-I-Ibar}), there are no arrows in $\Gamma$ between $i_r$ and $i_{r+1}$. Thus there are no arrows in $\overline \Gamma$ between any of the vertices $(i_r)^1=i'_{r'}$ or $(i_{r})^2=i'_{r'+1}$ and any of the vertices $(i_{r+1})^1=i'_{r'+2}$ or $(i_{r+1})^2=i'_{r'+3}$. Then, by (\ref{ch3:eq_action-on-polyn}), for each $f\in\bfk_{\overline d}[x]$ the element $\tau^*_re(\ui)=\tau_{r'+1}\tau_{r'+2}\tau_{r'}\tau_{r'+1}e(\phi(\ui))$ maps $fe(\phi(\ui))$ to
$$
\begin{array}{lll}
s_{r'+1}s_{r'+2}s_{r'}s_{r'+1}(f)e(\phi(s_r(\ui))).
\end{array}
$$
Thus we see that the action of $\tau^*_re(\phi(\ui))$ on the polynomial representation is the same as that of $\tau_re(\ui)$.

Finally, assume that $i_r=i_{r+1}\in I_1$. In this case we have
$$
(i_{r})^1=i'_{r'}=(i_{r+1})^1=i'_{r'+2},\qquad  (i_{r})^2=i'_{r'+1}=(i_{r+1})^2=i'_{r'+3}.
$$
Then, by (\ref{ch3:eq_action-on-polyn}), for each $f\in\bfk_{\overline d}[x]$ the element $\tau^*_re(\phi(\ui))=-\tau_{r'+1}\tau_{r'+2}\tau_{r'}\tau_{r'+1}e(\phi(\ui))$ maps $fe(\phi(\ui))$ to
$$
\begin{array}{lll}
s_{r'+1}\partial_{r'+2}\partial_{r'}(x_{r'+1}-x_{r'+2})s_{r'+1}(f)e(\phi(s_r(\ui))),
\end{array}
$$
where $\partial_r$ is the Demazure operator (see the definition before Proposition \ref{ch3:prop_faith-rep-KLR}).
To prove that this gives the same result as for $\tau_re(\ui)$, it is enough to check this on monomials $x_{r}^nx_{r+1}^me(\ui)$. Assume for simplicity that $n\geqslant m$. The situation $n\leqslant m$ can be treated similarly. The element $\tau_re(\ui)$ maps this monomial to
$$
\partial_r(x^n_rx^m_{r+1})e(\ui)=-\sum_{a=m}^{n-1}x_{r}^ax_{r+1}^{n+m-1-a}e(\ui).
$$
Here the symbol $\sum\limits_{a=x}^y$ means $0$ when $y=x-1$.
The element $\tau^*_re(\phi(\ui))$ maps $x_{r'+1}^nx_{r'+2}^me(\phi(\ui))$ to $s_{r'+1}\partial_{r'+2}\partial_{r'}[x_{r'+1}^{m+1}x_{r'+2}^{n}-x_{r'+1}^{m}x_{r'+2}^{n+1}]e(\phi(\ui))$, which equals 
$$
\begin{array}{llll}
&&s_{r'+1}\left[-\left(\sum\limits_{a=0}^mx_{r'}^ax_{r'+1}^{m-a}\right)\left(\sum\limits_{b=0}^{n-1}x_{r'+2}^bx_{r'+3}^{n-1-b}\right)+\left(\sum\limits_{a=0}^{m-1}x_{r'}^ax_{r'+1}^{m-1-a}\right)\left(\sum\limits_{b=0}^{n}x_{r'+2}^bx_{r'+3}^{n-b}\right)\right]e(\phi(\ui))\\
&=&\left[-\left(\sum\limits_{a=0}^mx_{r'}^ax_{r'+2}^{m-a}\right)\left(\sum\limits_{b=0}^{n-1}x_{r'+1}^bx_{r'+3}^{n-1-b}\right)+\left(\sum\limits_{a=0}^{m-1}x_{r'}^ax_{r'+2}^{m-1-a}\right)\left(\sum\limits_{b=0}^{n}x_{r'+1}^bx_{r'+3}^{n-b}\right)\right]e(\phi(\ui))\\
&=&\left[-x_{r'}^m\left(\sum\limits_{b=0}^{n-1}x_{r'+1}^bx_{r'+3}^{n-1-b}\right)+x_{r'+1}^n\left(\sum\limits_{a=0}^{m-1}x_{r'}^ax_{r'+2}^{m-1-a}\right)\right]e(\phi(\ui))\\
&=&\left[-x_{r'+1}^m\left(\sum\limits_{b=0}^{n-1}x_{r'+1}^bx_{r'+2}^{n-1-b}\right)+x_{r'+1}^n\left(\sum\limits_{a=0}^{m-1}x_{r'+1}^ax_{r'+2}^{m-1-a}\right)\right]e(\phi(\ui))\\
&=&-\left(\sum\limits_{a=m}^{n-1}x_{r'+1}^ax_{r'+2}^{m+n-1-a}\right)e(\phi(\ui)).
\end{array}
$$

Here the first equality follows from the following property of the Demazure operator
$$
\partial_r(x_{r+1}^n)=-\partial_r(x_{r}^n)=\sum_{a=0}^{n-1}x_r^ax_{r+1}^{n-1-a},
$$
the fourth equality follows from Remark \ref{ch3:rem_xa=xa+1}. Other equalities are obtained by elementary manipulations with sums.
\end{proof}

\subsection{Isomorphism $\Phi$}
\label{ch3:subs-isom_Phi}

\smallskip
\begin{thm}
\label{ch3:thm_KLR-e-e+1}
For each $\alpha\in Q_I^+$, there is an algebra isomorphism $\Phi_{\alpha,\bfk}\colon R_{\alpha,\bfk}\to S_{\overline\alpha,\bfk}$ such that
$$
e(\ui)\mapsto e(\phi(\ui)),\\
$$
$$
x_re(\ui)\mapsto x^*_{r}e(\phi(\ui)),\\
$$
$$
\tau_re(\ui)\mapsto \tau^*_re(\phi(\ui)).
%\left\{\begin{array}{ll}
%\tau_{r'}e(\phi(\ui)),& \mbox{ if }i_r,i_{r+1}\in I_0,\\
%\tau_{r'}\tau_{r'+1}e(\phi(\ui))& \mbox{ if }i_r\in I_1,i_{r+1}\in I_0,\\
%\tau_{r'+1}\tau_{r'}e(\phi(\ui))& \mbox{ if }i_r\in I_0,i_{r+1}\in I_1,\\
%\tau_{r'+1}\tau_{r'+2}\tau_{r'}\tau_{r'+1}e(\phi(\ui))& \mbox{ if }i_r,i_{r+1}\in I_1,i_r\ne i_{r+1},\\
%-\tau_{r'+1}\tau_{r'+2}\tau_{r'}\tau_{r'+1}e(\phi(\ui))& \mbox{ if }i_r=i_{r+1}\in I_1.\\
%\end{array}\right.
$$
%Here $\ui=(i_1,\cdots,i_d)$ and $r'=r'_\ui$ is as in Section \ref{ch3:subs_comp-pol-reps}.
\end{thm}
\begin{proof}[Proof]
By Proposition \ref{ch3:prop_faith-rep-KLR}, the representation $\bfk_\alpha^{(I)}$ of $R_{\alpha,\bfk}$ is faithful. Now, in view of Lemma \ref{ch3:lem_comp-operators}, it is enough to prove the following two facts:
\begin{itemize}
\item[\textbullet] the elements $e(\phi(\ui))$, $x^*_r$, $\tau^*_r$ generate $S_{\overline\alpha,\bfk}$,
\item[\textbullet] the representation $\bfk^{(\overline I)}_{\overline \alpha,\mathrm{ord}}/J_{\overline\alpha,\mathrm{ord}}$ of $S_{\overline\alpha,\bfk}$ is faithful.
\end{itemize}

Fix $\ui,\uj\in I^\alpha$.
Set $\ui'=(i'_1,\cdots,i'_{\overline d})=\phi(\ui)$, $\uj'=\phi(\uj)$. Let $\bfB$ and $\bfB'$ be the bases of $e(\uj)R_{\alpha,\bfk}e(\ui)$ and $e(\uj')R_{\overline\alpha,\bfk}e(\ui')$, respectively, as in Remark \ref{ch3:rk_basis-KLR}. These bases depend on some choices of reduced expressions. We will make some special choices later.
For each element $b=\tau_{w}x_1^{a_1}\cdots x_d^{a_d}e(\ui)\in \bfB$ we construct an element $b^*\in e(\uj')S_{\overline\alpha,\bfk}e(\ui')$ that acts by the same operator on the polynomial representation. We set
$$
b^*=\tau_{p_1}^*\cdots\tau_{p_k}^*(x_1^*)^{a_1}\cdots(x_{d}^*)^{a_{d}}e(\ui')\in e(\uj')S_{\overline\alpha,\bfk}e(\ui'),
$$
where $w=s_{p_1}\cdots s_{p_k}$ is a reduced expression (as we said above, some special choice of reduced expressions will be fixed later).

Let us call the permutation $w\in\frakS_{\ui',\uj'}$ \emph{balanced} if we have $w(a+1)=w(a)+1$ for each $a$ such that $i'_a=i^1$ for some $i\in I$ (and thus $i'_{a+1}=i^2$). Otherwise we say that $w$ is \emph{unbalanced}. There exists a unique map $u\colon \frakS_{\ui,\uj}\to \frakS_{\ui',\uj'}$ such that for each $w\in \frakS_{\ui,\uj}$ the permutation $u(w)$ is balanced and $w(r)<w(t)$ if and only if $u(w)(r')<u(w)(t')$ for each $r,t\in[1,d]$, where $r'=r'_\ui$ and $t'=t'_\ui$ are as in Section \ref{ch3:subs_comp-pol-reps}. The image of $u$ is exactly the set of all balanced permutations in $\frakS_{\ui',\uj'}$.

Assume that $w\in\frakS_{\ui',\uj'}$ is unbalanced. We claim that there exists an index $a$ such that $i'_a\in \overline I_1$ and $w(a)>w(a+1)$. Indeed, let $J$ be the set of indices $a\in[1,\overline d]$ such that $i'_a\in \overline I_1$. As $\uj'$ is well-ordered, we have $\sum_{a\in J}(w(a+1)-w(a))=\#J$. As $w$ is unbalanced, not all summands in this sum are equal to $1$. Then one of the summands must be negative. Let $a\in J$ be an index such that $w(a)>w(a+1)$. We can assume that the reduced expression of $w$ is of the form $w=s_{p_1}\cdots s_{p_k}s_{a}$. In this case the element $\tau_we(\ui')$ is zero in $S_{\overline\alpha,\bfk}$ because the sequence $s_a(\ui')$ is unordered.

Assume that $w\in\frakS_{\ui',\uj'}$ is balanced. Thus, there exists some $\widetilde w\in \frakS_{\ui,\uj}$ such that $u(\widetilde w)=w$. We choose an arbitrary reduced expression  $\widetilde w=s_{p_1}\cdots s_{p_k}$ and we choose the reduced expression $w=s_{q_1}\cdots s_{q_r}$ of $w$ obtained from the reduced expression of $\widetilde w$ in the following way. For $t\in \{1,\cdots,k\}$ set $\ui^t=s_{p_{t+1}}\cdots s_{p_k}(\ui)$ (in particular, we have $\ui^k=\ui$). We write $\ui^t=(i^t_1,\cdots,i^t_d)$. We construct the reduced expression of $w$ as $w=\hat s_{p_1}\cdots \hat s_{p_k}$, where for $a=p_t$ we have 
$$
\hat s_{a}=
\left\{
\begin{array}{lll}
s_{a'} &\mbox{ if } i^t_a,i^t_{a+1}\in I_0,\\
s_{a'+1}s_{a'} &\mbox{ if } i^t_a\in I_0,i^t_{a+1}\in I_1,\\
s_{a'}s_{a'+1}&\mbox{ if } i^t_a\in I_1,i^t_{a+1}\in I_0,\\
s_{a'+1}s_{a'}s_{a'+2}s_{a'+1}&\mbox{ if } i^t_a,i^t_{a+1}\in I_1,
\end{array}
\right.
$$
where $a'=a'_{\ui^r}$ is as in Section \ref{ch3:subs_comp-pol-reps}.
Let us explain why the obtained expression of $w$ is reduced. The fact that the expression $\widetilde w=s_{p_1}\cdots s_{p_k}$ is reduced means the following. When we apply the transpositions $s_{p_{k}}$, $s_{p_{k-1}}$, $\cdots$,$s_{p_1}$ consecutively to the $d$-tuple $(1,2,\cdots,d)$, if two elements of the set $\{1,2,\cdots,d\}$ are exchanged once by some $s$, then these two elements are never exchanged again by another $s$ later. It is clear that the expression $w=s_{q_1}\cdots s_{q_r}=\hat s_{p_1}\cdots \hat s_{p_k}$ inherits the same property from $\widetilde w=s_{p_1}\cdots s_{p_k}$ because for each $a,b\in \{1,2,\cdots,d\}$, $a\ne b$ we have the following (we set $a'=a'_\ui$, $b'=b'_\ui$).
\begin{itemize}
\item If $i_a,i_b\in I_0$, then if the reduced expression of $\widetilde w$ exchanges $a$ and $b$ exactly once or never exchanges them then the expression of $w$ exchanges $a'$ and $b'$ exactly once or never exchanges them, respectively.

\item If $i_a\in I_0$ and $i_b\in I_1$, then if the reduced expression of $\widetilde w$ exchanges $a$ and $b$ exactly once or never exchanges them then the expression of $w$ exchanges $a'$ and $b'$ exactly once or never exchanges them, respectively, and it also exchanges $a'$ with $b'+1$ exactly once or, respectively, never exchanges them.

\item If $i_a\in I_1$ and $i_b\in I_0$, then if the reduced expression of $\widetilde w$ exchanges $a$ and $b$ exactly once or never exchanges them then the expression of $w$ exchanges $a'$ and $b'$ exactly once or never exchanges them, respectively, and it also exchanges $a'+1$ with $b'$ exactly once or, respectively, never exchanges them.

\item If $i_a,i_b\in I_1$, then if the reduced expression of $\widetilde w$ exchanges $a$ and $b$ exactly once or never exchanges them then the expression of $w$ exchanges $a'$ and $b'$ exactly once or never exchanges them, respectively, and the same thing for $a'$ and $b'+1$, for $a'+1$ and $b'$, and for $a'+1$ and $b'+1$.
\end{itemize}
If the reduced expressions are chosen as above, then the element $\tau_we(\ui')=\tau_{q_1}\cdots\tau_{q_r}e(\ui')\in S_{\alpha,\bfk}$ is equal to $\pm (\tau_{p_1}\cdots \tau_{p_k}e(\ui))^*$.

The discussion above shows that the image of an element $b'\in\bfB'$ in $e(\uj')S_{\overline\alpha,\bfk}e(\ui')$ is either zero or of the form $\pm b^*$ for some $b\in\bfB$. Moreover, each $b^*$ for $b\in\bfB$ can be obtained in  such a way.
%This completes the proof because
Now we get the following.
\begin{itemize}
\item[\textbullet] The elements $e(\phi(\ui))$, $x^*_r$ and $\tau^*_r$ generate $S_{\overline\alpha,\bfk}$ because the image of each element of $\bfB'$ in $e(\uj')S_{\overline\alpha,\bfk}e(\ui')$ is either zero or a monomial in $e(\phi(\ui))$, $x^*_r$, $\tau^*_r$.
\item[\textbullet] The representation $\bfk^{(\overline I)}_{\overline \alpha,\mathrm{ord}}/J_{\overline\alpha,\mathrm{ord}}$ of $S_{\overline\alpha,\bfk}$ is faithful because the spanning set $\{b^*;~b\in\bfB\}$ of $e(\uj')S_{\overline\alpha,\bfk}e(\ui')$ acts on the polynomial representation by linearly independent operators (because the polynomial representation of $R_{\alpha,\bfk}$ in Proposition \ref{ch3:prop_faith-rep-KLR} is faithful).
\end{itemize}
\end{proof}

\begin{rk}
$(a)$Note that Theorem \ref{ch3:thm_KLR-e-e+1} also remains true for an infinite quiver $\Gamma$ because $\alpha$ is supported on a finite number of vertices (see also Remark \ref{ch3:rem_KLR-infinite}).

$(b)$The formulas that define the isomorphism $\Phi_{\alpha,\bfk}$ become more natural if we look at them from the point of view of Khovanov-Lauda diagrams (see \cite{KL}). Diagrammatically, the isomorphism $\Phi_{\alpha,\bfk}$ looks in the following way. It sends a diagram representing an element of $R_{\alpha,\bfk}$ to the diagram (sometimes with a sign) obtained by replacing each strand with label $k\in I_1$ by two parallel strands with labels $k^1$ and $k^2$ (if there is a dot on the strand with label $k$, it should be moved to the strand with label $k^1$). For example, if $i,j\in I_0$ and $k\in I_1$, we have
$$
    
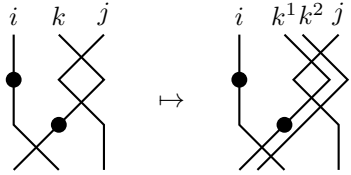
\begin{tikzpicture}[scale=0.6, thick]

    \draw (0,0) -- (2,2) -- (1,3) node[above,at end]{$k$}; 
    \draw (1,0) -- (0,1) -- (0,3) node[above,at end]{$i$};
    \draw (2,0) -- (2,1) -- (1,2) -- (2,3) node[above,at end]{$j$};
    \fill (0,2) circle (5pt);
    \fill (1,1) circle (5pt);
    
%    \node at (0,1.5) {$\mapsto$};
    
    \node at (3.5,1.5) {$\mapsto$};
    
    \draw (5,0) -- (7,2) -- (6,3) node[above,at end]{$k^1$}; 
    \draw (5.4,0) -- (7.4,2) -- (6.4,3);
    \node at (6.6,3.47) {$k^2$}; 
    \draw (6,0) -- (5,1) -- (5,3) node[above,at end]{$i$};
    \draw (7.2,0) -- (7.2,1) -- (6.2,2) -- (7.2,3) node[above,at end]{$j$};
    \fill (5,2) circle (5pt);
    \fill (6,1) circle (5pt);

%      \draw (-3,0)  +(-1,-1) -- +(0,1);
%      \draw (-3,0)  +(0,-1) .. controls (-4,0) ..  +(0,1);
%      \draw (-3,0) +(0,-1) -- +(-1,1);
%      \draw (-3,0) +(1,-1) .. controls (-3.5,0.5) .. +(1,1);
%      \node at (-1,0) {\Large $=$};  
\end{tikzpicture}
$$
\end{rk}

\section{Categorical representations}
\label{ch3:sec_catO}

\subsection{The standard representation of $\widetilde{\mathfrak{sl}}_e$}
\label{ch3:subs_stand-rep-aff}

Consider the affine Lie algebra (over $\bbC$) $\widetilde{\mathfrak{sl}}_e=\mathfrak{sl}_e\otimes\bbC[t,t^{-1}]\oplus\bbC \bm{1}$. Let $e_i$, $f_i$, $h_i$, $i=0,1,\ldots,e-1$, be the standard generators of $\widetilde{\mathfrak{sl}}_e$ (see Remark \ref{ch3:rk-KM-Lie}).
Let $V_e$ be a $\bbC$-vector space with canonical basis $\{v_1,\cdots,v_e\}$ and set $U_e=V_e\otimes \bbC[z,z^{-1}]$. The vector space $U_e$ has a basis $\{u_r;~r\in\bbZ\}$ where $u_{a+eb}=v_a\otimes z^{-b}$ for $a\in[1,e]$, $b\in\bbZ$. It has a structure of an $\widetilde{\mathfrak{sl}}_e$-module such that
$$
f_i(u_r)=\delta_{i\equiv r}u_{r+1}\quad \mbox{and}\quad e_i(u_r)=\delta_{i\equiv r-1}u_{r-1}.
$$
Let $\{v'_1,\cdots,v'_{e+1}\}$ and $\{u'_r;r\in\bbZ\}$ denote the bases of $V_{e+1}$ and $U_{e+1}$.

%Similarly, let $V_{e+1}$ be a $\bbC$-vector space with basis $\{v'_1,\cdots,v'_{e+1}\}$ and set $U_{e+1}=V_{e+1}\otimes \bbC[z,z^{-1}]$, $u'_{a+(e+1)b}=v'_a\otimes z^{-b}$.
Fix an integer $0\leqslant k<e$. Consider the following inclusion of vector spaces
$$
V_e\subset V_{e+1}, ~v_r\mapsto
\left\{\begin{array}{ll}
v'_r &\mbox{ if }r\leqslant k,\\
v'_{r+1} &\mbox{ if }r>k.
\end{array}\right.
$$
It yields an inclusion $\mathfrak{sl}_e\subset\mathfrak{sl}_{e+1}$ such that
$$
e_r\mapsto
\left\{\begin{array}{rl}
e_r &\mbox{ if }r\in[1,k-1],\\
{[e_k,e_{k+1}]} &\mbox{ if }r=k,\\
e_{r+1} &\mbox{ if }r\in[k+1,e-1],
\end{array}\right.
$$
$$
f_r\mapsto
\left\{\begin{array}{rl}
f_r &\mbox{ if }r\in[1,k-1],\\
{[f_{k+1},f_k]} &\mbox{ if }r=k,\\
f_{r+1} &\mbox{ if }r\in[k+1,e-1],
\end{array}\right.
$$
$$
h_r\mapsto
\left\{\begin{array}{rl}
h_r &\mbox{ if }r\in[1,k-1],\\
h_k+h_{k+1} &\mbox{ if }r=k,\\
h_{r+1} &\mbox{ if }r\in[k+1,e-1].
\end{array}\right.
$$

This inclusion lifts uniquely to an inclusion  $\widetilde{\mathfrak{sl}}_e\subset\widetilde{\mathfrak{sl}}_{e+1}$ such that
$$
e_0\mapsto
\left\{\begin{array}{rl}
e_0 &\mbox{ if }k\ne 0,\\
{[e_0,e_1]} &\mbox{ else},\\
\end{array}\right.
$$
$$
f_0\mapsto
\left\{\begin{array}{rl}
f_0 &\mbox{ if }k\ne 0,\\
{[f_1,f_0]} &\mbox{ else},\\
\end{array}\right.
$$
$$
h_0\mapsto
\left\{\begin{array}{rl}
h_0 &\mbox{ if }k\ne 0,\\
{h_0+h_1} &\mbox{ else}.\\
\end{array}\right.
$$

Consider the inclusion $U_e\subset U_{e+1}$ such that $u_{r}\mapsto u'_{\Upsilon(r)}$, where $\Upsilon$ is defined in (\ref{ch3:eq_upsilon}).

\smallskip
\begin{lem}
\label{ch3:lem-embed-rep-e-e+1}
The embeddings $V_e\subset V_{e+1}$ and $U_e\subset U_{e+1}$ are compatible with the actions of $\mathfrak{sl}_e\subset \mathfrak{sl}_{e+1}$ and $\widetilde{\mathfrak{sl}}_e\subset \widetilde{\mathfrak{sl}}_{e+1}$ respectively.
\qed
\end{lem}

%\subsection{The Fock space}
%\label{ch3:subs_Fock-space}

%Let $V_e$, $U_e$ be as above.

%\begin{df}
%Let $m$ be an integer. The \emph{Fock space of level $1$} is a vector space $\calF_m$ of semi-infinite wedges $u_{r_1}\wedge u_{r_2}\wedge u_{r_3}\wedge\cdots$ such that $r_s=m-s+1$ for all but finitely many $s$. This vector space is provided with a natural $\widetilde{\mathfrak{sl}}_e$-action of level $1$.
%\end{df}
%To the semi-infinite wedge above we associate a partition $\lambda=(\lambda_1,\lambda_2,\cdots,)$ with $\lambda_s=r_s-m+s-1$.  For a partition $\lambda=(\lambda_1,\cdots,\lambda_s)\in\calP$ denote by $|\lambda,m\rangle=u_{r_1}\wedge u_{r_2}\wedge u_{r_3}\wedge\cdots$. The elements $|\lambda,m\rangle$ for $\lambda\in\calP$ form a basis in $\calF_m$.

%\begin{df}
%Let $\bfm=(m_1,\cdots,m_l)$ be an $l$-tuple of integers. The \emph{Fock space of level $l$} is the $\widetilde{\mathfrak{sl}}_e$-module $\calF_\bfm=\calF_{m_1}\otimes\cdots\otimes\calF_{m_l}$.
% \end{df}

%For an $l$-partition $\lambda=(\lambda^{(1)},\cdots,\lambda^{(l)})\in\calP^l$ set $|\lambda,\bfm\rangle=|\lambda^{(1)},m_1\rangle\otimes\cdots\otimes|\lambda^{(l)},m_l\rangle$.  The elements $|\lambda,\bfm\rangle$ for $\lambda\in\calP^l$ form a basis in $\calF_\bfm$.

\subsection{Type A quivers}
\label{ch3:subs-typeA_quiv}

Let $\Gamma_\infty=(I_\infty,H_\infty)$ be the quiver with the set of vertices $I_\infty=\bbZ$ and the set of arrows $H_\infty=\{i\to i+1;~i\in I_\infty\}$. Assume that $e>1$ is an integer. Let $\Gamma_e=(I_e,H_e)$ be the quiver with the set of vertices $I_e=\bbZ/e\bbZ$ and the set of arrows $H_e=\{i\to i+1;~i\in I_e\}$. Then $\frakg_{I_e}$ is the Lie algebra $\widetilde{\mathfrak{sl}}_e=\mathfrak{sl}_e\otimes\bbC[t,t^{-1}]\oplus\bbC \bm{1}$ (see Remark \ref{ch3:rk-KM-Lie}).

Assume that $\Gamma=(I,H)$ is a quiver whose connected components are of the form $\Gamma_e$, with $e\in\bbN$, $e>1$ or $e=\infty$. For $i\in I$ denote by $i+1$ and $i-1$ the (unique) vertices in $I$ such that there are arrows $i\to i+1$ and $i-1\to i$.

Let $X_I$ be the free abelian group with basis $\{\varepsilon_i;~i\in I\}$. Set also
\begin{equation}
\label{ch3:eq_X+}
X^+_I=\bigoplus_{i\in I}\bbN\varepsilon_i.
\end{equation}
Let us also consider the following additive map
$$
\iota\colon Q_I\to X_I,\quad\alpha_i\mapsto \varepsilon_i-\varepsilon_{i+1}.
$$
We may omit the symbol $\iota$ and write $\alpha$ instead of $\iota(\alpha)$.
Let $\phi$ denote also the unique additive embedding
\begin{equation}
\label{ch3:eq_phi(mu)}
\phi\colon X_I\to X_{\overline I}, \quad\varepsilon_i\mapsto \varepsilon_{i'},
\end{equation}
where
$$
i'=
\left\{\begin{array}{ll}
i^0 &\mbox{ if }i\in I_0,\\
i^1&\mbox{ if }i\in I_1.
\end{array}\right.
$$

\subsection{Categorical representations}
\label{ch3:subs_categ-action}

\medskip

Let $\Gamma=(I,H)$ be a quiver as in Section \ref{ch3:subs-typeA_quiv}. Let $\bfk$ be a field. Assume that $\calC$ is a $\Hom$-finite $\bfk$-linear abelian category.

\smallskip
\begin{df}
\label{ch3:def-categ_action-KLR} A $\mathfrak{g}_I$-categorical representation $(E,F,x,\tau)$ in $\calC$ is the following data:
\begin{itemize}
    \item[(1)] a decomposition $\calC=\bigoplus_{\mu\in X_I}\calC_\mu$,
    \item[(2)] a pair of biadjoint exact endofunctors $(E,F)$ of $\calC$,
    \item[(3)] morphisms of functors $x\colon F\to F$ and $\tau\colon F^2\to F^2$,
    \item[(4)] decompositions $E=\bigoplus_{i\in I}E_i$ and $F=\bigoplus_{i\in I}F_i$,
\end{itemize}
satisfying the following conditions.
\begin{itemize}
    \item[(a)] We have $E_i(\calC_\mu)\subset\calC_{\mu+\alpha_i}$, $F_i(\calC_\mu)\subset\calC_{\mu-\alpha_i}$.

    \item[(b)]

For each $d\in\bbN$ there is an algebra homomorphism $\psi_d\colon R_{d,\bfk}\to \End(F^d)^{\rm op}$ such that
$\psi_d(e(\ui))$ is the projector to $F_{i_d}\cdots F_{i_1}$, where $\ui=(i_1,\cdots,i_d)$ and
$$
\psi_d(x_r)=F^{d-r}x F^{r-1},
\qquad
\psi_d(\tau_r)=F^{d-r-1}\tau F^{r-1}.
$$
    \item[(c)] For each $M\in\calC$ the endomorphism of $F(M)$ induced by $x$ is nilpotent.
\end{itemize}

\end{df}

\smallskip
\begin{rk}
\label{ch3:rk_df-repl-F-by-E}
$(a)$ For a pair of adjoint functors $(E,F)$ we have an isomorphism $\End(E^d)\simeq \End(F^d)^{\rm op}$. In particular, the algebra homomorphism $R_{d,\bfk}\to \End(F^d)^{\rm op}$ in Definition \ref{ch3:def-categ_action-KLR} yields an algebra homomorphism $R_{d,\bfk}\to \End(E^d)$.

$(b)$If the quiver $\Gamma$ is infinite, the direct sums in $(4)$ should be understood in the following way. For each object $M\in \calC$, there is only a finite number of $i\in I$ such that $E_i(M)$ and $F_i(M)$ are nonzero.
\end{rk}

%Assume that there is an isomorphism of quivers $I\simeq\scrF$. Then Definitions \ref{ch3:def-categ_action-Hecke}, \ref{ch3:def-categ_action-KLR} are equivalent by Proposition \ref{ch3:prop-isom_Hekce-KLR-widetilde}. %\cite[Thm.~3.16]{Rouq-2KM}.

\subsection{From $\widetilde{\mathfrak{sl}}_{e+1}$-categorical representations to $\widetilde{\mathfrak{sl}}_{e}$-categorical representations}
\label{ch3:subs_cat-lem}

%As above we assume that $R$ is a local ring with residue field $\bfk$. Let $\overline\scrF$ be a subset of $\bfk$ stable by multiplication by $q$ and $q^{-1}$ and such that the set $\overline\scrF/q^{\bbZ}$ is finite.

%We construct a subset $\scrF\subset \bfk$ in the following way.

As in Section \ref{ch3:subs_stand-rep-aff}, we fix $0\leqslant k<e$. Only in Section \ref{ch3:subs_cat-lem}, we assume that $\Gamma=(I,H)$ and $\overline\Gamma=(\overline I,\overline H)$ are fixed as in as in Section \ref{ch3:subs_not-e-e+1} (i.e., we have $\Gamma=\Gamma_e$, $I_1=\{k\}$ and we idenfity $\overline\Gamma$ with $\Gamma_{e+1}$).

Let $\overline\calC$ be a $\Hom$-finite abelian $\bfk$-linear category. Let
$$
\overline E=\overline E_0\oplus\overline E_1\oplus\cdots\oplus\overline E_e,\qquad \overline F=\overline F_0\oplus\overline F_1\oplus\cdots\oplus\overline F_e
$$
be endofunctors defining a $\widetilde{\mathfrak{sl}}_{e+1}$-categorical representation in $\overline\calC$.
%(in the sense of Definition \ref{ch3:def-categ_action-KLR}) with respect to algebra homomorphisms
Let $\overline\psi_d\colon R_{d,\bfk}\to \End(\overline F^d)^{\rm op}$ be the corresponding algebra homomorphism.
We set $\overline F_\ui=\overline F_{i_d}\cdots \overline F_{i_1}$ for any tuple $\ui=(i_1,\cdots,i_d)\in \overline I^d$ and $\overline F_{\overline\alpha}=\bigoplus_{\ui\in \overline I^{\overline\alpha}}\overline F_\ui$ for any element $\overline\alpha\in Q_{\overline I}^+$.
If $|\overline\alpha|=d$ let $\overline \psi_{\overline\alpha}\colon R_{\overline\alpha,\bfk}\to \End(\overline F_{\overline\alpha})^{\rm op}$ be the $\overline\alpha$-component of $\overline \psi_d$.

Now, recall the notation $X_{\overline I}^+$ from (\ref{ch3:eq_X+}). Assume that we have
\begin{equation}
\label{ch3:eq_ass-Cmu-0}
\overline\calC_\mu=0, \quad\forall\mu\in X_{\overline I}\backslash X^+_{\overline I}.
\end{equation}
For $\mu\in X^+_I$ set $\calC_\mu=\overline\calC_{\phi(\mu)}$, where the map $\phi$ is as in (\ref{ch3:eq_phi(mu)}). Let $\calC=\bigoplus_{\mu\in X^+_I}\calC_\mu$.

\smallskip
\begin{rk}
%The subcategory $\widetilde\calC$ of $\overline\calC$ satisfies the following conditions.
$(a)$ $\calC$ is stable by $\overline F_i$, $\overline E_i$ for each $i\ne k,k+1$,

$(b)$ $\calC$ is stable by $\overline F_{k+1}\overline F_k$, $\overline E_k\overline E_{k+1}$,

$(c)$ $\overline F_{i_d}\overline F_{i_{d-1}}\cdots\overline F_{i_1}(M)=0$ for each $M\in \calC$ whenever the sequence $(i_1,\cdots,i_d)$ is unordered (see Section \ref{ch3:subs_bal-quot}).
%The condition $(c)$ follows from (\ref{ch3:eq_ass-Cmu-0}) because for each $\mu\in X_{I}$ the coefficient of $\overline\varepsilon_{k+1}$ in the decomposition of $\phi(\mu)\in X_{\overline I}$ in the standard basis is zero .
\end{rk}

\smallskip
Consider the following endofunctors of $\calC$:
$$
E_i=
\left\{
\begin{array}{lll}
\restr{\overline E_i}{\calC} &\mbox{ if } 0\leqslant i<k,\\
\restr{\overline E_{k}\overline E_{k+1}}{\calC} &\mbox{ if } i=k,\\
\restr{\overline E_{i+1}}{\calC} &\mbox{ if } k<i<e,
\end{array}
\right.
$$
$$
F_i=
\left\{
\begin{array}{lll}
\restr{\overline F_i}{\calC} &\mbox{ if } 0\leqslant i<k,\\
\restr{\overline F_{k+1}\overline F_k}{\calC} &\mbox{ if } i=k,\\
\restr{\overline F_{i+1}}{\calC} &\mbox{ if } k<i<e.
\end{array}
\right.
$$
Similarly to the notations above we set $F_\ui=F_{i_d}\cdots F_{i_1}$ for any tuple $\ui=(i_1,\cdots,i_d)\in I^d$ and $F_\alpha=\bigoplus_{\ui\in I^\alpha}F_\ui$ for any element $\alpha\in Q_{I}^+$.
Note that we have $F_\ui=\restr{{\overline F_{\phi(\ui)}}}{\calC}$ for each $\ui\in I^\alpha$.

Let $\alpha\in Q^+_I$ and $\overline\alpha=\phi(\alpha)$. Note that we have
$$
F_\alpha=\bigoplus_{\ui\in \overline I^{\overline\alpha}_{\rm ord}} \restr{\overline F_\ui}{\calC}.
$$ The homomorphism $\overline\psi_{\overline\alpha}$ yields a homomorphism $\bfe R_{\overline\alpha,\bfk}\bfe\to \End(F_{\alpha})^{\rm op}$, where $\bfe=\sum_{\ui\in \overline I^{\overline\alpha}_{\rm ord}}e(\ui)$. By (c), the homomorphism $\bfe R_{\overline\alpha,\bfk}\bfe\to \End(F_{\alpha})^{\rm op}$ factors through a homomorphism $S_{\overline\alpha,\bfk}\to \End(F_{\alpha})^{\rm op}$. Let us call it $\overline\psi'_{\overline\alpha}$. Then we can define an algebra homomorphism $\psi_{\alpha}\colon R_{\alpha,\bfk}\to\End(F_\alpha)^{\rm op}$ by setting $\psi_\alpha=\overline\psi'_{\overline\alpha}\circ\Phi_{\alpha,\bfk}$.

Now, Theorem \ref{ch3:thm_KLR-e-e+1} implies the following result.

\smallskip
\begin{thm}
\label{ch3:thm_categ-e-e+1}
For each category $\overline\calC$, defined as above, that satisfies (\ref{ch3:eq_ass-Cmu-0}), we have a categorical representation of $\widetilde{\mathfrak{sl}}_e$ in the subcategory $\calC$ of $\overline\calC$ given by functors $F_i$ and $E_i$ and the algebra homomorphisms $\psi_\alpha\colon R_{\alpha,\bfk}\to \End(F_\alpha)^{\rm op}$.
\qed
\end{thm}

Now, we describe the example that motivated us to prove Theorem \ref{ch3:thm_categ-e-e+1}. See \cite{Mak-Zuck} for details.
\begin{ex}
Let $U_e$, $V_e$ be as in Section \ref{ch3:subs_stand-rep-aff}. Fix $\nu=(\nu_1,\cdots,\nu_l)\in\bbN^l$ and put $N=\sum_{r=1}^l \nu_r$. Set $\wedge^\nu U_e=\wedge^{\nu_1}U_e\otimes\cdots\otimes \wedge^{\nu_l} U_e$.
%The obvious $\widetilde{\mathfrak{sl}}_e$-action on $U_e$ yields an   $\widetilde{\mathfrak{sl}}_e$-action on $\wedge^\nu U_e$. 

Let $O^\nu_{-e}$ be the parabolic category $\calO$ for $\widehat{\mathfrak{gl}}_N$ with parabolic type $\nu$ at level $-e-N$. The categorical representation of $\widetilde{\mathfrak{sl}}_{e}$ in $O^\nu_{-e}$ (constructed in \cite{RSVV}) yields an $\widetilde{\mathfrak{sl}}_{e}$-module structure on the (complexified) Grothendieck group $[O^\nu_{-e}]$ of $O^\nu_{-e}$. This module is isomorphic to $\wedge^\nu U_e$. 

Let us apply Theorem \ref{ch3:thm_main-thm-categ-rep-int-art} to $\overline\calC=O^\nu_{-(e+1)}$. It happens that in this case the subcategory $\calC\subset \overline\calC$ defined as above is equivalent to $O^\nu_{-e}$. The embedding of categories $O^\nu_{-e}\subset O^\nu_{-(e+1)}$ categorifies the embedding $\wedge^\nu U_e\subset \wedge^\nu U_{e+1}$ (see also Lemma \ref{ch3:lem-embed-rep-e-e+1}).

\end{ex}
 
\subsection{Reduction of the number of idempotents}
In this section we show that it is possible to reduce the number of idempotents in the quotient in Definition \ref{ch3:def_bal-KLR}. This is necessary to generalise Theorem \ref{ch3:thm_categ-e-e+1}. Here we assume the quivers $\Gamma=(I,H)$ and $\overline\Gamma=(\overline I,\overline H)$ are as in Section \ref{ch3:subs_not-quiv-I-Ibar}.

We fix $\alpha\in Q_I^+$ and put $\overline\alpha=\phi(\alpha)$. We say that the sequence $\ui\in \overline I^{\overline\alpha}$ is \emph{almost ordered} if there exists a well-ordered sequence $\uj\in \overline I^{\overline\alpha}$ such that there exists an index $r$ such that $j_r\in \overline I_1$ and $\ui=s_r(\uj)$. It is clear from the definition that each almost ordered sequence is unordered because the subsequence $(i_1,i_2,\cdots,i_r)$ of $\ui$ contains more elements from $\overline I_2$ than from $\overline I_1$. The following lemma reduces the number of generators of the kernel of $\bfe R_{\overline\alpha,\bfk}\bfe\to S_{\overline\alpha,\bfk}$ (see Definition \ref{ch3:def_bal-KLR}).

\smallskip
\begin{lem}
\label{ch3:lem-red_number_idemp}
The kernel of the homomorphism $\bfe R_{\overline\alpha,\bfk}\bfe\to S_{\overline\alpha,\bfk}$ is equal to $\sum_{\ui}\bfe R_{\overline\alpha,\bfk}e(\ui) R_{\overline\alpha,\bfk}\bfe$, where $\ui$ runs over the set of all almost ordered sequences in $\overline I^{\overline\alpha}$.
\end{lem}
\begin{proof}
%We will see the elements of KLR algebras as linear combinations of Khovanov-Lauda diagrams. By definition, each element of the kernel of $\bfe R_{\overline\alpha,\bfk}\bfe\to S_{\overline\alpha,\bfk}$ is a linear combination of diagrams such that each diagram contains a horizontal slice such that the labels of the strands in this slice yields an unordered sequence. We claim that each such diagram has also a vertical slice with an almost ordered sequence.
Denote by $J$ the ideal $\sum_{\ui}\bfe R_{\overline\alpha,\bfk}e(\ui) R_{\overline\alpha,\bfk}\bfe$ of $\bfe R_{\overline\alpha,\bfk}\bfe$, where $\ui$ runs over the set of all almost ordered sequences in $\overline I^{\overline\alpha}$.

By definition, each element of the kernel of $\bfe R_{\overline\alpha,\bfk}\bfe\to S_{\overline\alpha,\bfk}$ is a linear combination of elements of the form $\bfe a e(\uj)b\bfe$, where $a$ and $b$ are in $R_{\overline\alpha,\bfk}$ and the sequence $\uj$ is unordered. By Remark \ref{ch3:rk_basis-KLR}, it is enough to prove that for each $\ui\in \overline I^{\overline\alpha}_{\rm ord}$, $\uj\in \overline I^{\overline\alpha}_{\rm un}$, $b\in R_{\overline\alpha,\bfk}$ and indices $p_1,\cdots,p_k$ the element $e(\ui)\tau_{p_1}\cdots\tau_{p_k}e(\uj)b\bfe$ is in $J$. We will prove this statement by induction on $k$.

Assume that $k=1$. Write $p=p_1$. The element $e(\ui)\tau_{p}e(\uj)b\bfe$ may be nonzero only if $\ui=s_p(\uj)$. This is possible only if the sequence $\uj$ is almost ordered. Thus the element $e(\ui)\tau_{p}e(\uj)b\bfe$ is in $J$.

Now, assume that $k>1$ and that the statement is true for each value $<k$. Set $w=s_{p_1}\cdots s_{p_k}$. We may assume that $\ui=w(\uj)$, otherwise the element $e(\ui)\tau_{p_1}\cdots\tau_{p_k}e(\uj)b\bfe$ is zero.
By assumptions on $\ui$ and $\uj$ there is an index $r\in [1,d]$ such that $i_r\in \overline I_1$ and $w^{-1}(r+1)<w^{-1}(r)$. Thus $w$ has a reduced expression of the form $w=s_rs_{r_1}\cdots s_{r_h}$. This implies that $\tau_{p_1}\cdots\tau_{p_k}e(\uj)$ is equal to a monomial of the form $\tau_{r}\tau_{r_1}\cdots\tau_{r_h} e(\uj)$ modulo monomials of the form $\tau_{q_1}\cdots\tau_{q_t}x_1^{b_1}\cdots x_d^{b_d}e(\uj)$ with $t<k$, see Remark \ref{ch3:rk_basis-KLR}. Thus the element $e(\ui)\tau_1\cdots\tau_ke(\uj)b\bfe$ is equal to $e(\ui)\tau_{r}\tau_{r_1}\cdots\tau_{r_h} e(\uj)b\bfe$ modulo the elements of the same form $e(\ui)\tau_{p_1}\cdots\tau_{p_k}e(\uj)b\bfe$ with smaller $k$. The element $e(\ui)\tau_{r}\tau_{r_1}\cdots\tau_{r_h} e(\uj)b\bfe$ is in $J$ because the sequence $s_r(\ui)$ is almost ordered and the additional terms are in $J$ by the induction assumption.
\end{proof}

\subsection{Generalization of Theorem \ref{ch3:thm_categ-e-e+1}}
In this section we modify slightly the definition of a categorical representation given in Definition \ref{ch3:def-categ_action-KLR}. The only difference is that we use the lattice $Q_I$ instead of $X_I$. This new definition is not equivalent to Definition \ref{ch3:def-categ_action-KLR}.  In this section we work with an arbitrary quiver $\Gamma=(I,H)$ without $1$-loops.

Let $\bfk$ be a field. Let $\calC$ be a $\bfk$-linear $\Hom$-finite category.

%The definition of a categorical representation used in this section is not exactly the same as the definition 

\smallskip
\begin{df}
A $\mathfrak{g}_I$-quasi-categorical representation $(E,F,x,\tau)$ in $\calC$ is the following data:
\begin{itemize}
    \item[(1)] a decomposition $\calC=\bigoplus_{\alpha\in Q_I}\calC_\alpha$,
    \item[(2)] a pair of biadjoint exact endofunctors $(E,F)$ of $\calC$,
    \item[(3)] morphisms of functors $x\colon F\to F$, $\tau\colon F^2\to F^2$,
    \item[(4)] decompositions $E=\bigoplus_{i\in I}E_i$, $F=\bigoplus_{i\in I}F_i$,
\end{itemize}
satisfying the following conditions.
\begin{itemize}
    \item[(a)] We have $E_i(\calC_\alpha)\subset\calC_{\alpha-\alpha_i}$, $F_i(\calC_\alpha)\subset\calC_{\alpha+\alpha_i}$.

    \item[(b)]

For each $d\in\bbN$ there is an algebra homomorphism $\psi_d\colon R_{d,\bfk}\to \End(F^d)^{\rm op}$ such that
$\psi_d(e(\ui))$ is the projector to $F_{i_d}\cdots F_{i_1}$, where $\ui=(i_1,\cdots,i_d)$ and
$$
\psi_d(x_r)=F^{d-r}x F^{r-1},
\qquad
\psi_d(\tau_r)=F^{d-r-1}\tau F^{r-1}.
$$
    \item[(c)] For each $M\in\calC$ the endomorphism of $F(M)$ induced by $x$ is nilpotent.    
\end{itemize}
If the quiver $\Gamma$ is infinite, condition $(4)$ should be understood in the same way as in Remark \ref{ch3:rk_df-repl-F-by-E} $(b)$.
\end{df}

%\smallskip
%Note that the Lie algebra $\frakg_I$ is independent of the orientation of the quiver $\Gamma$. However, the definition of a categorical representation of $\frakg_I$ depends on the orientation.

\smallskip
Now, fix a decomposition $I=I_0\coprod I_1$ as in Section \ref{ch3:subs_not-quiv-I-Ibar}. We consider the quiver $\overline\Gamma=(\overline I,\overline H)$ and the map $\phi$ as in Section \ref{ch3:subs_not-quiv-I-Ibar}. To distinguish the elements of $Q_I$ and $Q_{\overline I}$, we write $Q_{\overline I}=\bigoplus_{i\in \overline I}\bbZ\overline\alpha_i$. For each $\alpha\in Q_I$ we set $\overline\alpha=\phi(\alpha)\in Q_{\overline I}$. (See Section \ref{ch3:subs_not-quiv-I-Ibar} for the notation.) However we can sometimes use the symbol $\overline\alpha$ for an arbitrary element of $Q_{\overline I}$ that is not associated with some $\alpha$ in $Q_I$.
Let $\overline\calC$ be a $\Hom$-finite abelian $\bfk$-linear category. Let
$\overline E=\bigoplus_{i\in \overline I}\overline E_i$ and $\overline F=\bigoplus_{i\in \overline I}\overline F_i$
be endofunctors defining a $\frakg_{\overline I}$-quasi-categorical representation in $\overline\calC$.
Let $\overline\psi_d\colon R_{d,\bfk}(\overline\Gamma)\to \End(\overline F^d)^{\rm op}$ be the corresponding algebra homomorphism.
We set $\overline F_\ui=\overline F_{i_d}\cdots \overline F_{i_1}$ for any tuple $\ui=(i_1,\cdots,i_d)\in \overline I^d$ and $\overline F_{\overline\alpha}=\bigoplus_{\ui\in \overline I^{\overline\alpha}}\overline F_\ui$ for any element $\overline\alpha\in Q_{\overline I}^+$.
If $|\overline\alpha|=d$, let $\overline \psi_{\overline\alpha}\colon R_{\overline\alpha,\bfk}\to \End(\overline F_{\overline\alpha})^{\rm op}$ be the $\overline\alpha$-component of $\overline \psi_d$.

Assume that $\calC$ is an abelian subcategory of $\overline\calC$ satisfying the following conditions:
\begin{itemize}
\item[(a)] $\calC$ is stable by $\overline F_i$ and $\overline E_i$ for each $i\in I_0$,

\item[(b)] $\calC$ is stable by $\overline F_{i^2}\overline F_{i^1}$ and $\overline E_{i^1}\overline E_{i^2}$ for each $i\in I_1$,

\item[(c)] we have $\overline F_{i^2}(\calC)=0$ for each $i\in I_1$,

\item[(d)] we have $\calC=\bigoplus_{\alpha\in Q_{I}} \calC\cap\overline\calC_{\overline\alpha}$.
\end{itemize}

By $(d)$, we get a decomposition $\calC=\bigoplus_{\alpha\in Q_I}\calC_{\alpha}$, where $\calC_\alpha=\calC\cap \overline\calC_{\overline\alpha}$.
For each $i\in I$ we consider the following endofunctors $E_i$ and $F_i$ of $\calC$:
$$
F_i=
\left\{
\begin{array}{ll}
\restr{\overline F_i}{\calC}& \mbox{ if }i\in I_0,\\
\restr{\overline F_{i^2}\overline F_{i^1}}{\calC}& \mbox{ if }i\in I_1,\\
\end{array}
\right.
$$
$$
E_i=
\left\{
\begin{array}{ll}
\restr{\overline E_i}{\calC}& \mbox{ if }i\in I_0,\\
\restr{\overline E_{i^1}\overline E_{i^2}}{\calC}& \mbox{ if }i\in I_1.\\
\end{array}
\right.
$$
As in the notations above we set $F_\ui=F_{i_d}\cdots F_{i_1}$ for any tuple $\ui=(i_1,\cdots,i_d)\in I^d$ and $F_\alpha=\bigoplus_{\ui\in I^\alpha}F_\ui$ for any element $\alpha\in Q_{I}^+$.
Note that we have $F_\ui=\restr{{\overline F_{\phi(\ui)}}}{\calC}$ for each $\ui\in I^\alpha$.

Let $\alpha\in Q^+_I$. We have
$$
F_\alpha=\bigoplus_{\ui\in \overline I^{\overline\alpha}_{\rm ord}} \restr{\overline F_\ui}{\calC}.
$$ The homomorphism $\overline\psi_{\overline\alpha}$ yields a homomorphism $\bfe R_{\overline\alpha,\bfk}\bfe\to \End(F_{\alpha})^{\rm op}$, where $\bfe=\sum_{\ui\in \overline I^{\overline\alpha}_{\rm ord}}e(\ui)$. 

Since the category $\calC$ satisfies $(a)$, $(b)$ and $(c)$, for each almost ordered sequence $\ui=(i_1,\cdots,i_d)\in I^\alpha$ we have $\overline F_{i_d}\cdots \overline F_{i_1}(\calC)=0$. By Lemma \ref{ch3:lem-red_number_idemp}, this implies that the homomorphism $\bfe R_{\overline\alpha,\bfk}\bfe\to \End(F_{\alpha})^{\rm op}$ factors through a homomorphism $S_{\overline\alpha,\bfk}\to \End(F_{\alpha})^{\rm op}$. Let us call it $\overline\psi'_{\overline\alpha}$. Then we can define an algebra homomorphism $\psi_{\alpha}\colon R_{\alpha,\bfk}\to\End(F_\alpha)^{\rm op}$ by setting $\psi_\alpha=\overline\psi'_{\overline\alpha}\circ\Phi_{\alpha,\bfk}$.

Now, Theorem \ref{ch3:thm_KLR-e-e+1} implies the following result.

\smallskip
\begin{thm}
\label{ch3:thm_categ-e-e+1-app}
For each abelian subcategory $\calC\subset\overline\calC$ as above, that satisfies $(a)-(d)$, we have a $\frakg_I$-quasi-categorical representation in $\calC$ given by functors $F_i$ and $E_i$ and the algebra homomorphisms $\psi_\alpha\colon R_{\alpha,\bfk}\to \End(F_\alpha)^{\rm op}$.
\qed
\end{thm}

\smallskip
\begin{rk}
Assume that the category $\overline\calC$ is such that we have $\overline\calC_{\overline\alpha}=0$ whenever $\overline\alpha=\sum_{i\in \overline I}d_i\overline\alpha_i\in Q_{\overline I}$ is such that $d_{i^1}<d_{i^2}$ for some $i\in I_1$. In this case the subcategory $\calC\subset \overline\calC$ defined by $\calC=\bigoplus_{\alpha\in Q_I}\overline\calC_{\overline\alpha}$ satisfies conditions $(a)-(d)$.
\end{rk}

\bigskip
\bigskip
\appendix
\centerline{\bf\Large Appendix}
\section{The geometric construction of the isomorphism $\Phi$}
\label{ch3:app-geom_constr}
The goal of this section is to give a geometric construction of the isomorphism $\Phi$ in Theorem \ref{ch3:thm_KLR-e-e+1}.

\renewcommand{\thesubsection}{\thesection\arabic{subsection}}

\subsection{The geometric construction of the KLR algebra}
Let $\bfk$ be a field. Let $\Gamma=(I,H)$ be a quiver without $1$-loops. See Section \ref{ch3:subs_KM-quiv} for the notations related to quivers. For an arrow $h\in H$
we will write $h'$ and $h''$ for its source and target respectively. Fix $\alpha=\sum_{i\in I}d_i\alpha_i\in Q^+_I$ and set $d=|\alpha|$. Set also 
$$
E_\alpha=\bigoplus_{h\in H}{\rm Hom}(V_{h'},V_{h''}),\qquad V_i=\bbC^{d_i},\qquad V=\bigoplus_{i\in I}V_i.
$$
The group $G_\alpha=\prod_{i\in I}GL(V_i)$ acts on $E_\alpha$ by base changes.

Set 
$$
I^\alpha=\{\ui=(i_1,\cdots,i_d)\in I^d;~\sum_{r=1}^d \alpha_{i_r}=\alpha\}.
$$
We denote by $F_\ui$ the variety of all flags 
$$
\phi=(V=V^0\supset V^1\supset\cdots\supset V^d=\{0\})
$$
in $V$ that are homogeneous with respect to the decomposition $V=\bigoplus_{i\in I}V_i$ and such that the $I$-graded vector space
$V^{r-1}/V^{r}$ has graded dimension $i_r$ for $r\in [1,d]$.
We denote by $\widetilde{F}_{\ui}$ the variety of pairs $(x,\phi)\in
E_\alpha\times F_\ui$ such that $x$ preserves $\phi$, i.e., we have $x(V^r)\subset V^r$ for
$r\in\{0,1,\cdots,m\}$. Let $\pi_{\ui}$ be the natural projection from
$\widetilde{F}_{\ui}$ to $E_\alpha$, i.e., $\pi_{\ui}:\widetilde{F}_\ui\to E_\alpha,~(x,\phi)\mapsto x. $
For $\ui,\uj\in I^\alpha$ we denote by $Z_{\ui,\uj}$ the variety of
triples $(x,\phi_1,\phi_2)\in E_\alpha\times F_{\ui}\times F_{\uj}$ such
that $x$ preserves $\phi_1$ and $\phi_2$ (i.e., we have $Z_{\ui,\uj}=\widetilde F_\ui\times_{E_\alpha}\widetilde F_\uj$). Set 
$$
Z_{\alpha}=\coprod_{\ui,\uj\in I^\alpha}Z_{\ui,\uj}, \qquad \widetilde F_\alpha=\coprod_{\ui\in I^\alpha} \widetilde F_\ui.
$$
We have an algebra structure on $H_*^{G_\alpha}(Z_\alpha,\bfk)$ such that the multiplication is the convolution product with respect to the inclusion $Z_{\alpha}\subset \widetilde F_{\alpha}\times \widetilde F_{\alpha}$. Here $H_*^{G_\alpha}(\bullet,\bfk)$ denotes the $G_\alpha$-equivariant Borel-Moore homology with coefficients in $\bfk$.  See \cite[Sec.~2.7]{CG} for the definition of the convolution product. 
% For $\ui\in I^\nu$ and $l=1,2,\cdots,d$, we define $\mathcal{O}_{\widetilde{F}_{\ui}}(l)$ to be the $G_V$-equivariant line bundle over $\widetilde{F}_{\ui}$ whose fiber at the flag $\phi$ is equal to $V^l/V^{l-1}$.

The following result is proved by Rouquier \cite{Rouq-2KM} and by Varagnolo-Vasserot \cite{VV} in the situation ${\rm char}~\bfk=0$. See \cite{Mak-parity} for the proof over an arbitrary field.

\smallskip
\begin{prop}
\label{ch3:prop-geom_KLR}
There is an algebra isomorphism $R_{\alpha,\bfk}(\Gamma)\simeq H^{G_\alpha}_*(Z_\alpha,\bfk)$. Moreover, for each $\ui,\uj\in I^\alpha$, the vector subspace $e(\ui)R_{\alpha,\bfk}(\Gamma)e(\uj)\subset R_{\alpha,\bfk}(\Gamma)$ corresponds to the vector subspace $H_*^{G_\alpha}(Z_{\ui,\uj},\bfk)\subset H_*^{G_\alpha}(Z_{\alpha},\bfk)$.

\qed
\end{prop}

\subsection{The geometric construction of the isomorphism $\Phi$}
As in Section \ref{ch3:subs_not-quiv-I-Ibar}, fix a decomposition $I=I_0\coprod I_1$ and consider the quiver $\overline\Gamma=(\overline I,\overline H)$; also fix $\alpha\in Q^+_I$ and consider $\overline\alpha=\phi(\alpha)\in Q^+_{\overline I}$. 

We start from the variety $Z_{\overline\alpha}$ defined with respect to the quiver $\overline\Gamma$. By Proposition \ref{ch3:prop-geom_KLR}, we have an algebra isomorphism $R_{\overline\alpha,\bfk}(\overline\Gamma)\simeq H_*^{G_{\overline\alpha}}(Z_{\overline\alpha},\bfk)$.
We have an obvious projection $p\colon Z_{\overline\alpha}\to E_{\overline\alpha}$ defined by $(x,\phi_1,\phi_2)\mapsto x$. For each $i\in I_1$ denote by $h_i$ the unique arrow in $\overline \Gamma$ that goes from $i^1$ to $i^2$. Consider the following open subset of $E_{\overline\alpha}$: $E^0_{\overline\alpha}=\{x\in E_{\overline\alpha};~ x_{h_i} \mbox{ is invertible } \forall i\in I_1\}$. Set $Z^0_{\overline\alpha}=p^{-1}(E^0_{\overline\alpha})$. The pullback with respect to the inclusion $Z^0_{\overline\alpha}\subset Z_{\overline\alpha}$ yields an algebra homomorphism $H_*^{G_{\overline\alpha}}(Z_{\overline\alpha},\bfk)\to H_*^{G_{\overline\alpha}}(Z^0_{\overline\alpha},\bfk)$ (see \cite[Lem.~2.7.46]{CG}).

\smallskip
\begin{rk}
\label{ch3:rk-unord-idemp-in-kernel}
If the sequence $\ui\in \overline I^{\overline\alpha}$ is unordered, then a flag from $F_\ui$ is never preserved by an element from $E^0_{\overline\alpha}$. This implies that $Z_{\ui,\uj}\cap Z^0_{\alpha}=\emptyset$ if $\ui$ or $\uj$ is unordered. Thus for each $\ui\in \overline I^{\overline\alpha}_{\rm un}$, the idempotent $e(\ui)$ is in the kernel of the homomorphism $H_*^{G_{\overline\alpha}}(Z_{\overline\alpha},\bfk)\to H_*^{G_{\overline\alpha}}(Z^0_{\overline\alpha},\bfk)$.
\end{rk}

\smallskip
Let $\bfe$ be the idempotent as in Definition \ref{ch3:def_bal-KLR}. Consider the following subset of $Z_{\overline\alpha}$:
$$
Z'_{\overline\alpha}=\coprod_{\ui,\uj\in \overline I^{\overline\alpha}_{\rm ord}}Z_{\ui,\uj}.
$$
The algebra isomorphism $R_{\overline\alpha,\bfk}(\overline\Gamma)\simeq H_*^{G_{\overline\alpha}}(Z_{\overline\alpha},\bfk)$ above restricts to an algebra isomorphism $\bfe R_{\overline\alpha}(\overline\Gamma)\bfe\simeq H_*^{G_{\overline\alpha}}(Z'_{\overline\alpha},\bfk)$.

Now, set $Z'^0_{\alpha}=Z'_{\alpha}\cap Z^0_{\alpha}$. Similarly to the construction above, we have an algebra homomorphism $H_*^{G_{\overline\alpha}}(Z'_{\overline\alpha},\bfk)\to H_*^{G_{\overline\alpha}}(Z'^0_{\overline\alpha},\bfk)$. By Remark \ref{ch3:rk-unord-idemp-in-kernel}, the kernel of this homomorphism contains the kernel of $\bfe R_{\overline\alpha,\bfk}(\overline\Gamma)\bfe\to R_{\alpha,\bfk}(\Gamma)$ (see Theorem \ref{ch3:thm_KLR-e-e+1}). The following result implies that these kernels are the same.

\smallskip
\begin{lem}
\label{ch3:lem-isom-Stein-Strein0}
We have the following algebra isomorphism $R_{\alpha,\bfk}(\Gamma)\simeq H_*^{G_{\overline\alpha}}(Z'^0_{\overline\alpha},\bfk)$.
\end{lem} 
\begin{proof}
For each $i\in I_0$ we identify $V_i\simeq V_{i^0}$. For each $i\in I_1$ we identify $V_i\simeq V_{i^1}\simeq V_{i^2}$. We have a diagonal inclusion $G_\alpha\subset G_{\overline\alpha}$, i.e., the component $GL(V_i)$ of $G_\alpha$ with $i\in I_0$ goes to $GL(V_{i^0})$ and the component $GL(V_i)$ with $i\in I_1$ goes diagonally to $GL(V_{i^1})\times GL(V_{i^2})$. 

Set $G^{\rm bis}_{\alpha}=\prod_{i\in I_1}GL(V_{i^2})\subset G_{\overline\alpha}$. We have an obvious group isomorphism $G_{\overline\alpha}/G_{\alpha}^{\rm bis}\simeq G_{\alpha} $. 

Let us denote by $X$ the choice of isomorphisms $V_{i^1}\simeq V_{i^2}$  mentioned above. Let $E^X_{\overline\alpha}$ be the subset of $E_{\overline\alpha}$ that contains only $x\in E_{\overline\alpha}$ such that for each $i\in I_1$ the component $x_{h_i}$ is the isomorphism chosen in $X$.

The group $G^{\rm bis}_\alpha$ acts freely on $E^0_{\overline\alpha}$ such that each orbit intersects $E^X_{\overline\alpha}$ once. This implies that we have an isomorphism of algebraic varieties $E^0_{\overline\alpha}/G^{\rm bis}_{\alpha}\simeq E^X_{\overline\alpha}$. Now, set $Z'^X_{\overline\alpha}=p^{-1}(E^X_{\overline\alpha})$. The same argument as above yields $Z'^0_{\overline\alpha}/G^{\rm bis}_{\alpha}\simeq Z'^X_{\overline\alpha}$. We get the following chain of algebra isomorphsims
$$
H^{G_{\overline\alpha}}_*(Z'^0_{\overline\alpha},\bfk)\simeq H^{G_{\overline\alpha}/ G^{\rm bis}_\alpha}_*(Z'^0_{\overline\alpha}/G^{\rm bis}_\alpha,\bfk)\simeq H_*^{G_\alpha}(Z'^X_{\overline\alpha},\bfk).
$$

To complete the proof we have to show that the $G_\alpha$-variety $Z'^X_{\overline\alpha}$ is isomorphic to $Z_\alpha$. Each element of $I^{\overline\alpha}_{\rm ord}$ is of the form $\phi(\ui)$ for a unique $\ui\in I^\alpha$, where $\phi$ is as in Section \ref{ch3:subs_not-quiv-I-Ibar}. Let us abbreviate $\ui'=\phi(\ui)$. By definition we have 
$$
Z'_{\overline\alpha}=\coprod_{\ui,\uj\in I^\alpha} Z_{\ui',\uj'}.
$$
Set $Z^X_{\ui',\uj'}=Z_{\ui',\uj'}\cap Z'^X_{\overline\alpha}$. We have an obvious isomorphism of $G_\alpha$-varieties $Z^X_{\ui',\uj'}\simeq Z_{\ui,\uj}$. (Beware, the variety $Z_{\ui,\uj}$ is defined with respect to the quiver $\Gamma$ and the variety $Z_{\ui',\uj'}$ is defined with respect to the quiver $\overline\Gamma$.) Taking the union for all $\ui,\uj\in I^\alpha$ yields an isomorphism of $G_\alpha$-varieties $Z'^X_{\overline\alpha}\simeq Z_\alpha$.  
\end{proof}

\smallskip
\begin{coro}
We have the following commutative diagram. 
$$
\begin{CD}
\bfe R_{\overline\alpha,\bfk}(\overline\Gamma)\bfe @>>> R_{\alpha,\bfk}(\Gamma)\\
@VVV                                 @VVV\\
H_*^{G_{\overline\alpha}}(Z'_{\overline\alpha},\bfk) @>>> H_*^{G_{\overline\alpha}}(Z'^0_{\overline\alpha},\bfk).
\end{CD}
$$
Here the left vertical map is the isomorphism from Proposition \ref{ch3:prop-geom_KLR}, the right vertical map is the isomorphism from Lemma \ref{ch3:lem-isom-Stein-Strein0}, the top horizontal map is obtained from Theorem \ref{ch3:thm_KLR-e-e+1} and the bottom horizontal map is the pullback with respect to the inclusion $Z'^0_{\overline\alpha}\subset Z'_{\overline\alpha}$.
\end{coro}
\begin{proof}
The result follows directly from Lemma \ref{ch3:lem-isom-Stein-Strein0}. The commutativity of the diagram is easy to see on the generators of $R_{\overline\alpha,\bfk}(\overline\Gamma)$.

Indeed, the isomorphism $R_{\alpha,\bfk}\simeq H_*^{G_{\alpha}}(Z_{\alpha},\bfk)$ is defined in the following way (see \cite[Sec.~2.9,~Thm.~2.4]{Mak-parity} for more details). The element $e(\ui)$ corresponds to the fundamental class $[Z_{\ui,\ui}]$. The element $x_re(\ui)$ corresponds to the first Chern class of some line bundle on $Z_{\ui,\ui}$. The element $\psi_re(\ui)$ corresponds to the fundamental class of some correspondence in $Z_{s_{r}(\ui),\ui}$. The commutativity of the diagram in the statement follows from standard properties of Chern classes and fundamental classes.
\end{proof}

\section{A local ring version in type A}
\label{ch3:app-local_ring}

In this appendix we give some versions of the main results of the paper (Theorems \ref{ch3:thm_KLR-e-e+1} and \ref{ch3:thm_categ-e-e+1}) over a local ring. These ring versions are interesting because the study of the category $\calO$ for $\widehat{\mathfrak{gl}}_N$ in \cite{Mak-Zuck} uses a deformation argument. For this we need a version of Theorem \ref{ch3:thm-isom_KLR_e_e+1_intro} over a local ring.

It is known that the affine Hecke algebra over a field is related with the KLR algebra (see Propositions \ref{ch3:prop-isom_Hekce-KLR-loc}, \ref{ch3:prop-isom_Hekce-KLR-widehat}). This allows to reformulate the definition of a categorical representation (see Definition \ref{ch3:def-categ_action-KLR}) that is given in term of KLR algebras in an equivalent way in terms of Hecke algebras (see Definition \ref{ch3:def-categ_action-Hecke}). The main difficulty is that there is no known relation between Hecke and KLR algebras over a ring. Over a local ring, we can give a definition of a categorical representation using the Hecke algebra (see Definition \ref{ch3:def-categ_rep_R}). But we have no equivalent definition in terms of KLR algebras. That is why, Proposition \ref{ch3:prop_morph-Phi-over-ring}, that is a ring analogue of Theorem \ref{ch3:thm_KLR-e-e+1}, is formulated in terms of Hecke algebras and not in terms of KLR algebras.

\subsection{Intertwining operators}

The center of the algebra $R_{\alpha,\bfk}$ is the ring of symmetric polynomials $\bfk_d[x]^{\frakS_d}$, see \cite[Prop.~3.9]{Rouq-2KM}. Thus $S_{\overline\alpha,\bfk}$ is a $\bfk_d[x]^{\frakS_d}$-algebra under the isomorphism $\Phi_{\alpha,\bfk}$ in Section \ref{ch3:subs-isom_Phi}. Let $\Sigma$ be the polynomial $\prod_{a<b}(x_a-x_b)^2\in\bfk_d[x]^{\frakS_d}$. Let $R_{\alpha,\bfk}[\Sigma^{-1}]$ and $S_{\overline\alpha,\bfk}[\Sigma^{-1}]$ be the rings of quotients of $R_{\alpha,\bfk}$ and $S_{\overline\alpha,\bfk}$ obtained by inverting $\Sigma$. We can extend the isomorphism $\Phi_{\alpha,\bfk}$ from Theorem \ref{ch3:thm_KLR-e-e+1} to an algebra isomorphism 
$$\Phi_{\alpha,\bfk}\colon R_{\alpha,\bfk}[\Sigma^{-1}]\to S_{\overline\alpha,\bfk}[\Sigma^{-1}].$$

Assume that the connected components of the quiver $\Gamma$ are of the form $\Gamma_a$ for $a\in\bbN$, $a>1$ or $a=\infty$. (The quiver $\Gamma_a$ is defined in Section \ref{ch3:subs-typeA_quiv}.)

Note that there is an action of the symmetric group $\frakS_d$ on $\bfk^{(I)}_{d}$ permuting the variables and the components of $\ui$.
Consider the following element in $R_{\alpha,\bfk}[\Sigma^{-1}]$:
$$
\Psi_re(\ui)=
\left\{\begin{array}{ll}
((x_r-x_{r+1})\tau_r+1)e(\ui)&\mbox{ if } i_{r+1}=i_{r},\\
-(x_r-x_{r+1})^{-1}\tau_r e(\ui)&\mbox{ if } i_{r+1}=i_{r}-1,\\
\tau_r e(\ui)&\mbox{ else}.\\
\end{array}\right.
$$
The element $\Psi_re(\ui)$ is called \emph{intertwining operator}. Using the formulas (\ref{ch3:eq_action-on-polyn}) we can check that $\Psi_re(\ui)$ still acts on the polynomial representation and the corresponding operator is equal to $s_re(\ui)$. Note also that $\widetilde\Psi_r=(x_r-x_{r+1})\Psi_r$ is an element of $R_{\alpha,\bfk}$.

\smallskip
\begin{lem}
\label{ch3:lem_morph-Phi-intert-op}
The images of intertwining operators by $\Phi_{\alpha,\bfk}\colon R_{\alpha,\bfk}\to S_{\overline\alpha,\bfk}$ can be described in the following way.
For $\ui\in I^\alpha$ such that $i_r-1\ne i_{r+1}$ we have
$$
\Phi_{\alpha,\bfk}(\Psi_re(\ui))=
\left\{\begin{array}{ll}
\Psi_{r'}e(\phi(\ui)),& \mbox{ if }i_r,i_{r+1}\in I_0,\\
\Psi_{r'}\Psi_{r'+1}e(\phi(\ui))& \mbox{ if }i_r\in I_1,i_{r+1}\in I_0,\\
\Psi_{r'+1}\Psi_{r'}e(\phi(\ui))& \mbox{ if }i_r\in I_0,i_{r+1}\in I_1,\\
\Psi_{r'+1}\Psi_{r'+2}\Psi_{r'}\Psi_{r'+1}e(\phi(\ui))& \mbox{ if }i_r,i_{r+1}\in I_1.\\
\end{array}\right.
$$
For $\ui\in I^\alpha$ such that $i_r-1= i_{r+1}$ we have
$$
\Phi_{\alpha,\bfk}(\widetilde\Psi_re(\ui))=
\left\{\begin{array}{ll}
\widetilde\Psi_{r'}e(\phi(\ui)),& \mbox{ if }i_r,i_{r+1}\in I_0,\\
\widetilde\Psi_{r'}\Psi_{r'+1}e(\phi(\ui))& \mbox{ if }i_r\in I_1,i_{r+1}\in I_0,\\
\Psi_{r'+1}\widetilde\Psi_{r'}e(\phi(\ui))& \mbox{ if }i_r\in I_0,i_{r+1}\in I_1.\\
%\Psi_{r'+1}\Psi_{r'+2}\Psi_{r'}\Psi_{r'+1}e(\phi(\ui))& \mbox{ if }i_r,i_{r+1}\in I_1.\\
\end{array}\right.
$$
Here $r'=r'_\ui$ is as in Section \ref{ch3:subs_comp-pol-reps}.
\end{lem}
\begin{proof}[Proof]
By construction of $\Phi_{\alpha,\bfk}$, the elements $\Phi_{\alpha,\bfk}(\Psi_re(\ui))$ and $\Phi_{\alpha,\bfk}(\widetilde\Psi_re(\ui))$ are the unique elements of $S_{\overline\alpha,\bfk}$ that act on the polynomial representation by the same operator as $\Psi_re(\ui)$ and $\widetilde \Psi_re(\ui)$, respectively. 

The right hand side in the formulas for $\Phi_{\alpha,\bfk}(\Psi_re(\ui))$ (or resp. $\Phi_{\alpha,\bfk}(\widetilde\Psi_re(\ui))$) in the statement is an element $X$ in $S_{\overline\alpha,\bfk}[\Sigma^{-1}]$. To complete the proof we have to show that
\begin{itemize}
\item[(1)] $X$ acts by the same operator as $\Psi_re(\ui)$ or $\widetilde\Psi_re(\ui)$, respectively, on the polynomial representation,
\item[(2)] $X$ is in $S_{\overline\alpha,\bfk}$.
\end{itemize}

Part $(1)$ is obvious. Part $(2)$ follows from part $(1)$ and from the faithfulness of the polynomial representation of $S_{\overline\alpha,\bfk}[\Sigma^{-1}]$ (see Lemma \ref{ch3:lem_pol-rep-of-S-defined+faith}). (In fact, part $(2)$ is not obvious only in the case $i_r=i_{r+1}\in I_1$.)

%The right hand side in the formulas for $\Phi_{\alpha,\bfk}(\Psi_re(\ui))$ (or resp. $\Phi_{\alpha,\bfk}(\widetilde\Psi_re(\ui))$) is an element $X$ in $S_{\overline\alpha,\bfk}[\Sigma^{-1}]$ that acts by the same operator than $\Psi_re(\ui)$ (or resp. $\Psi_re(\ui$) on the polynomial representation. However, since the polynomial representation of $S_{\overline\alpha,\bfk}[\Sigma^{-1}]$ is faithful by Lemma \ref{ch3:lem_pol-rep-of-S-defined+faith}, and since there is a well-defined element of $S_{\overline\alpha,\bfk}$ which acts in the same way as $X$ in this polynomial representation, we conclude that $X$ is in $S_{\overline\alpha,\bfk}$.

\end{proof}

\subsection{Special quivers}
\label{ch3:subs_not-e-e+1}
From now on we will be interested only in some special types of quivers.

\medskip
First, consider the quiver $\Gamma=\Gamma_{e}$, where $e$ is an integer $>1$. In particular, from now on we fix $I=\bbZ/e\bbZ$. Fix $k\in [0,e-1]$ and set $I_1=\{k\}$ and $I_0=I\backslash\{k\}$. In this case the quiver $\overline\Gamma$ is isomorphic to $\Gamma_{e+1}$. More precisely,
%$$
%\begin{array}{ll}
%i^0=i\in\bbZ/(e+1)\bbZ &\mbox{ if } i\in [0,k-1]\subset \bbZ/e\bbZ,\\
%i^0=i+1\in\bbZ/(e+1)\bbZ &\mbox{ if } i\in [k+1,e-1]\subset \bbZ/e\bbZ,\\
%k^1=k\in\bbZ/(e+1)\bbZ,\\
%k^2=k+1\in\bbZ/(e+1)\bbZ.
%\end{array}
%$$
the decomposition $\overline I=\overline I_0\sqcup\overline I_1\sqcup\overline I_2$ is such that $\overline I_1=\{k\}$ and $\overline I_2=\{k+1\}$. To avoid confusion, for $i\in \overline I$ we will write $\overline\alpha_i$ and $\overline\varepsilon_i$ for $\alpha_i$ and $\varepsilon_i$ respectively.

\smallskip
\begin{rk}
\label{ch3:rk_def-ordered-Ibar}
If $\Gamma$ is as above, a sequence $\ui=(i_1,\cdots,i_d)\in \overline I^d$ is well-ordered if for each index $a$ such that $i_a=k$ we have $a<d$ and $i_{a+1}=k+1$. The sequence $\ui$ is unordered if there is $r\leqslant d$ such that the subsequence $(i_1,\cdots,i_r)$ contains more elements equal to $k+1$ than elements equal to $k$.
\end{rk}

\smallskip
Let $\Upsilon\colon\bbZ\to\bbZ$ be the map given for $a\in\bbZ,b\in[0,e-1]$ by
\begin{equation}
\label{ch3:eq_upsilon}
\Upsilon(ae+b)=
\left\{\begin{array}{ll}
a(e+1)+b &\mbox{ if }b\in[0,k],\\
a(e+1)+b+1 &\mbox{ if }b\in[k+1,e-1].
\end{array}\right.
\end{equation}

\medskip
Now, consider the quiver $\widetilde\Gamma=(\Gamma_\infty)^{\sqcup l}$ (i.e., $\widetilde\Gamma$ is a disjoint union of $l$ copies of $\Gamma_\infty$).
Set $\widetilde\Gamma=(\widetilde I,\widetilde H)$ and write $\widetilde\alpha_{i}$ and $\widetilde\varepsilon_{i}$ and for $\alpha_{i}$ and $\varepsilon_{i}$ respectively for each $i\in \widetilde I$.
We identify an element of $\widetilde I$ with an element $(a,b)\in \bbZ\times [1,l]$ in the obvious way.
Consider the decomposition $\widetilde I=\widetilde I_0\sqcup \widetilde I_1$ such that $(a,b)\in \widetilde I_1$ if and only if $a\equiv k~\mod~e$.
In this case the quiver $\overline{\widetilde\Gamma}$ is isomorphic to $\widetilde\Gamma$. We will often write $\widetilde\Gamma$ instead of $\overline{\widetilde \Gamma}$ (but sometimes, if confusion is possible, we will use the notation $\overline{\widetilde\Gamma}$ to stress that we work with the doubled quiver).  More precisely, in this case we have
$$
\begin{array}{ll}
(a,b)^0=(\Upsilon(a),b),\\
(a,b)^1=(\Upsilon(a),b),\\
(a,b)^2=(\Upsilon(a)+1,b).\\
\end{array}
$$
To distinguish notations, we will always write $\widetilde\phi$ for any of the maps $\widetilde\phi\colon\widetilde I^\infty\to\widetilde I^\infty$, $Q_{\widetilde I}\to Q_{\widetilde I}$, $X_{\widetilde I}\to X_{\widetilde I}$ in Section \ref{ch3:subs_not-quiv-I-Ibar}.

\medskip
From now on we write $\Gamma=\Gamma_e$, $\overline\Gamma=\Gamma_{e+1}$ and $\widetilde\Gamma=(\Gamma_\infty)^{\sqcup l}$. Recall that
$$
I=I_e=\bbZ/e\bbZ,\quad \overline I=I_{e+1}=\bbZ/(e+1)\bbZ, \quad\widetilde I=(I_\infty)^{\sqcup l}=\bbZ\times [1,l].
$$
Consider the quiver homomorphism $\pi_e\colon \widetilde \Gamma\to \Gamma$ such that
$$
\pi_e\colon\widetilde I\to I,~(a,b)\mapsto a~\mod~e.
%,\qquad \pi_{e+1}\colon\widetilde I\to\overline I~(a,b)\mapsto a~\mod~e+1.
$$
Then $\pi_{e+1}$ is a quiver homomorphism $\pi_{e+1}\colon\widetilde\Gamma\to\overline\Gamma$.
They yield $\bbZ$-linear maps
$$
\pi_e\colon Q_{\widetilde I}\to Q_{I}, \quad \pi_e\colon X_{\widetilde I}\to X_{I},\quad \pi_{e+1}\colon Q_{\widetilde I}\to Q_{\overline I}, \quad \pi_{e+1}\colon X_{\widetilde I}\to X_{\overline I}.
$$

The following diagrams are commutative for $\alpha\in Q^+_I$ and $\widetilde\alpha\in Q^+_{\widetilde I}$ such that $\pi_e(\widetilde\alpha)=\alpha$,
$$
\begin{CD}
Q_{\widetilde I} @>{\widetilde\phi}>> Q_{\widetilde I}\\
@V{\pi_e}VV                                  @V{\pi_{e+1}}VV\\
Q_I              @>{\phi}>>           Q_{\overline I}\\
\end{CD}
\qquad\qquad
\begin{CD}
X_{\widetilde I} @>{\widetilde\phi}>> X_{\widetilde I}\\
@V{\pi_e}VV                                  @V{\pi_{e+1}}VV\\
X_I              @>{\phi}>>           X_{\overline I}\\
\end{CD}
\qquad\qquad
\begin{CD}
\widetilde I^{\widetilde\alpha} @>{\widetilde\phi}>> \widetilde I^{\widetilde\phi(\widetilde\alpha)}\\
@V{\pi_e}VV                                  @V{\pi_{e+1}}VV\\
I^{\alpha}         @>{\phi}>>           \overline I^{\phi(\alpha)}\\
\end{CD}
$$

The quiver $\widetilde \Gamma$ is infinite. We will sometimes use its truncated version. Fix a positive integer $N$. Denote by $\widetilde\Gamma^{\leqslant N}$ the full subquiver (i.e., a quiver with a smaller set of vertices and the same arrows between these vertices) of $\widetilde\Gamma$ that contains only vertices $(a,b)$ such that $|a|\leqslant eN$. Let $\overline{\widetilde\Gamma}^{\leqslant N}$ be the doubled quiver associated with $\widetilde\Gamma^{\leqslant N}$. We can see the quiver $\overline{\widetilde\Gamma}^{\leqslant N}$ as a full subquiver of $\overline{\widetilde\Gamma}$ that contains only vertices $(a,b)$ such that we have
$$
\left\{
\begin{array}{lll}
-(e+1)N \leqslant a\leqslant (e+1)N& \mbox{ if } k\ne 0,\\
-(e+1)N \leqslant a\leqslant (e+1)N+1& \mbox{ else}.
\end{array}
\right.
$$
(Attention, it is not true that the isomorphism of quivers $\widetilde\Gamma\simeq \overline{\widetilde\Gamma}$ takes $\widetilde\Gamma^{\leqslant N}$ to $\overline{\widetilde\Gamma}^{\leqslant N}$.)

\subsection{Hecke algebras}
\label{ch3:subs_Hecke}
Let $R$ be a commutative ring with $1$. Fix an element $q\in R$.

\smallskip
\begin{df}
The \emph{affine Hecke algebra} $H_{R,d}(q)$ is the $R$-algebra generated by $T_1,\cdots,T_{d-1}$ and the invertible elements $X_1,\cdots,X_d$ modulo the following defining relations
$$
\begin{array}{lllll}
X_rX_s=X_sX_r, &T_rX_r=X_rT_r ~\mbox{ if }|r-s|>1,\\
T_{r}T_{s}=T_{s}T_{r} \mbox{ if }|r-s|>1,& T_{r}T_{r+1}T_{r}=T_{r+1}T_{r}T_{r+1},\\
 T_rX_{r+1}=X_rT_r+(q-1)X_{r+1}, &T_rX_{r}=X_{r+1}T_r-(q-1)X_{r+1},\\
(T_r-q)(T_r+1)=0.\\
\end{array}
$$
\end{df}

Assume that $R=\bfk$ is a field and $q\ne 0, 1$.
The algebra $H_{d,\bfk}(q)$ has a faithful representation (see \cite[Prop.~3.11]{MS}) in the vector space $\bfk[X^{\pm 1}_1,\cdots,X^{\pm 1}_d]$ such that $X^{\pm 1}_r$ acts by multiplication by $X^{\pm 1}_r$ and $T_r$ by
$$
T_r(P)=qs_r(P)+(q-1)X_{r+1}(X_{r}-X_{r+1})^{-1}(s_r(P)-P).
$$

The following operator acts on $\bfk[X_1^{\pm 1},\cdots,X_d^{\pm 1}]$ as the reflection $s_r$
$$
\Psi_r=\frac{X_r-X_{r+1}}{qX_r-X_{r+1}}(T_r-q)+1=(T_r+1)\frac{X_r-X_{r+1}}{X_r-qX_{r+1}}-1.
$$
For a future use, consider the element $\widetilde\Psi_r\in H_{d,\bfk}$ given by
$$
\widetilde\Psi_r=(qX_r-X_{r+1})\Psi_r=(X_{r}-X_{r+1})T_r+(q-1)X_{r+1}.
$$

%{\color{red}
%Fix nonzero elements  $Q_1,\cdots,Q_l\in \bfk$. \todo{Parameters are from a field!}
%Now,
%consider the subset $\scrF=\scrF(q,Q)$ of $\bfk$ 
%given by
%\begin{equation}
%\label{ch3:eq_F(q,Q)}
%\scrF(q,Q)=\bigcup_{r\in \bbZ,t\in[1,l]}\{q^rQ_t\}.
%\end{equation}
%We can consider $\scrF$ as the vertex set of a quiver with an arrow $i\to j$ if and only if $j=qi$.
%}

\subsection{The isomorphism between Hecke and KLR algebras}
\label{ch3:subs_isom-KLR-Hecke-gen}
%Assume that $R=\bfk$ is a field and $q\ne 0, 1$.

First, we define some localized versions of Hecke algebras and KLR algebras.
Let $\scrF$ be a finite subset of $\bfk^\times$. 
%such that $q^\bbZ\scrF/q^\bbZ$ is finite. 
We view $\scrF$ as the vertex set of a quiver with an arrow $i\to j$ if and only if $j=qi$.
Consider the algebra
$$
A_1=\bigoplus_{\ui\in \scrF^d}\bfk[X_1^{\pm 1},\cdots,X_d^{\pm 1}][(X_r-X_t)^{-1},(qX_r-X_t)^{-1};~r\ne t]e(\ui),
$$
where $e(\ui)$ are orthogonal idempotents and $X_r$ commutes with $e(\ui)$.
Let $H_{d,\bfk}^{\rm loc}(q)$ be the $A_1$-module given by the extension of scalars from the $\bfk[X_1^{\pm 1},\cdots,X_d^{\pm 1}]$-module $H_{d,\bfk}(q)$. It has a $\bfk$-algebra structure such that
$$
T_re(\ui)-e(s_r(\ui))T_r=(1-q)X_{r+1}(X_r-X_{r+1})^{-1}(e(\ui)-e(s_r(\ui)))
$$
and
$$
Z^{-1}T_r=T_rZ^{-1},\quad \mbox{ where } Z=\prod_{r<t}(X_r-X_t)^2\prod_{r\ne t}(qX_r-X_t)^2.
$$

In this section the KLR algebras are always defined with respect to the quiver $\scrF$. We consider the algebra
$$
A_2=\bigoplus_{\ui\in \scrF^d}\bfk[x_1,\cdots,x_d][S_\ui^{-1}]e(\ui),
$$
where
$$
S_\ui=\{(x_r+1),(i_r(x_r+1)-i_t(x_t+1)),(qi_r(x_r+1)-i_t(x_t+1);~r\ne t)\}.
$$
Consider the following central element in $R_{d,\bfk}$
$$
z=\prod_{r}(x_r+1)\prod_{i,j\in\scrF,r\ne t}(i(x_r+1)-j(x_t+1)).
$$ 
The $A_2$-module $R_{d,\bfk}^{\rm loc}=A_2\otimes_{\bfk^{(\scrF)}_d}R_{d,\bfk}$ has a $\bfk$-algebra structure because it is a subalgebra in $R_{d,\bfk}[z^{-1}]$, where $\bfk_d^{(\scrF)}$ is as in (\ref{ch3:eq_k^I}).

\begin{rk}
We assumed above that the set $\calF$ is finite. This assumption is important because it implies that $A_1$ contains $\bfk[X_1^{\pm 1},\cdots,X_d^{\pm 1}]$ and $A_2$ contains $\bfk[x_1,\cdots,x_d]$. However, it is possible to define the algebras above ($A_1$, $A_2$, $H_{d,\bfk}^{\rm loc}(q)$ and $R_{d,\bfk}^{\rm loc}$) for arbitrary $\calF\subset \bfk^\times$. Indeed, if $\calF_1\subset \calF_2$ are finite, then the algebra defined with respect to $\calF_1$ is obviously a non-unitary subalgebra of the algebra defined with respect to $\calF_2$. Then we can define the algebras $A_1$, $A_2$, $H_{d,\bfk}^{\rm loc}(q)$ and $R_{d,\bfk}^{\rm loc}$ with respect to any arbitrary $\calF$. For example, we define the algebra $R_{d,\bfk}^{\rm loc}$ associated with $\calF$ as
$$
R_{d,\bfk}^{\rm loc}(\calF)=\lim_{\stackrel{\longrightarrow}{\calF_0\subset \calF}}R_{d,\bfk}^{\rm loc}(\calF_0),
$$   
where the direct limit is taken over all finite subsets $\calF_0$ of $\calF$.
Note that if the set $\calF$ is infinite, then the algebras $A_1$, $A_2$, $H_{d,\bfk}^{\rm loc}(q)$ and $R_{d,\bfk}^{\rm loc}$ are not unitary.
\end{rk}

From now on we assume that $\calF$ is an arbitrary subset of $\bfk^\times$.

\begin{prop}
\label{ch3:prop-isom_Hekce-KLR-loc}
There is an isomorphism of $\bfk$-algebras $R_{d,\bfk}^{\rm loc}\simeq H_{d,\bfk}^{\rm loc}(q)$ such that
$$
e(\ui)\mapsto e(\ui),
$$
$$
x_re(\ui)\mapsto (i_r^{-1}X_r-1)e(\ui),\\
$$
$$
\Psi_re(\ui)\mapsto \Psi_re(\ui).
$$
\end{prop}
\begin{proof}
The polynomial representations of $H_{d,\bfk}(q)$ and $R_{d,\bfk}$ yield faithful representations of $H_{d,\bfk}^{\rm loc}(q)$ and $R_{d,\bfk}^{\rm loc}$ on $A_1$ and $A_2$ respectively. Moreover, there is an isomorphism of $\bfk$-algebras $A_2\simeq A_1$ given by $x_re(\ui)\mapsto (i_r^{-1}X_r-1)e(\ui)$. 

This implies the statement. Indeed, the elements $e(\ui)\in R_{d,\bfk}^{\rm loc}$ and $e(\ui)\in H_{d,\bfk}^{\rm loc}(q)$ act on $A_2\simeq A_1$ by the same operators. The elements $x_re(\ui)\in R_{d,\bfk}^{\rm loc}$ and $(i_r^{-1}X_r-1)e(\ui)\in H_{d,\bfk}^{\rm loc}(q)$ act on $A_2\simeq A_1$ by the same operators. Finally, the elements $\Psi_re(\ui)\in R_{d,\bfk}^{\rm loc}$ and $\Psi_re(\ui)\in H_{d,\bfk}^{\rm loc}(q)$ also act on $A_2\simeq A_1$ by the same operators. The elements above generate the algebras $R_{d,\bfk}^{\rm loc}$ and $H_{d,\bfk}^{\rm loc}(q)$.
\end{proof}

Now, we consider the subalgebra $\widehat R_{d,\bfk}$ of $R_{d,\bfk}^{\rm loc}$ generated by
\begin{itemize}
\item[\textbullet] the elements of $R_{d,\bfk}$,
\item[\textbullet] the elements $(x_r+1)^{-1}$,
\item[\textbullet] the elements of the form $(i_r(x_r+1)-i_t(x_t+1))^{-1}e(\ui)$ such that $r\ne t$ and $i_r\ne i_t$,
\item[\textbullet]  the elements of the form $(q i_r(x_r+1)-i_t(x_t+1))^{-1}e(\ui)$ such that $r\ne t$ and $qi_r\ne i_t$.
\end{itemize}
Similarly, consider the subalgebra $\widehat H_{d,\bfk}(q)$ of $H_{d,\bfk}^{\rm loc}(q)$ generated by
\begin{itemize}
\item[\textbullet] the elements of $H_{d,\bfk}(q)$,
\item[\textbullet] the elements of the form $(X_r-X_t)^{-1}e(\ui)$ such that $r\ne t$ and $i_r\ne i_t$,
\item[\textbullet] the elements of the form $(qX_r-X_t)^{-1}e(\ui)$ such that $r\ne t$ and $qi_r\ne i_t$.
\end{itemize}

\smallskip
Note that the element $\Psi_re(\ui)\in H^{\rm loc}_{d,\bfk}(q)$ belongs to $\widehat H_{d,\bfk}(q)$ if $i_r\ne qi_{r+1}$.
We have the following proposition, see also \cite[Sec.~3.2]{Rouq-2KM}.

\begin{prop}
\label{ch3:prop-isom_Hekce-KLR-widehat}
The isomorphism $R_{d,\bfk}^{\rm loc}\simeq H_{d,\bfk}^{\rm loc}(q)$ from Proposition \ref{ch3:prop-isom_Hekce-KLR-loc} restricts to an isomorphism $\widehat R_{d,\bfk}\simeq \widehat H_{d,\bfk}(q)$.
\qed
\end{prop}

\subsection{Deformation rings}
\label{ch3:subs_def-ring}

In this section we introduce some general definitions from \cite{RSVV} for a later use.

We call the \emph{deformation ring} $(R,\kappa,\kappa_1,\cdots,\kappa_l)$ a regular commutative noetherian $\bbC$-algebra $R$ with $1$ equipped with a homomorphism $\bbC[\kappa^{\pm 1},\kappa_1,\cdots,\kappa_l]\to R$. Let $\kappa,\kappa_1,\cdots,\kappa_l$ also denote the images of $\kappa,\kappa_1,\cdots,\kappa_l$ in $R$.
A deformation ring is \emph{in general position} if any two elements of the set
$$
\{\kappa_u-\kappa_v+a\kappa+b,\kappa-c;~a,b\in\bbZ,c\in \bbQ,u\ne v\}
$$
have no common non-trivial divisors. 
A \emph{local deformation ring} is a deformation ring which is a local ring such that $\kappa_1,\cdots,\kappa_l, \kappa-e$ belong to the maximal ideal of $R$. %Note that this definition depends on the choice of a positive integer $e$.
Note that each $\bbC$-algebra that is a field has a \emph{trivial} local deformation ring structure, i.e., such that $\kappa_1=\cdots=\kappa_l=0$ and $\kappa=e$. We always consider $\bbC$ as a local deformation ring with a trivial deformation ring structure.

We will write $\overline\kappa=\kappa(e+1)/e$ and $\overline\kappa_r=\kappa_r(e+1)/e$.
We will abbreviate $R$ for $(R,\kappa,\kappa_1,\cdots,\kappa_l)$ and $\overline R$ for $(R,\overline\kappa,\overline\kappa_1,\cdots,\overline\kappa_l)$.

%Let $R$ be a commutative unital domain and let $\kappa,\tau_1,\cdots,\tau_l$ be elements of $R$ such that the following conditions are satisfied.
%\begin{itemize}
%    \item[\textbullet] $R$ is local with residue field $\bfk=\bbC$ and the field of fractions $K$.
%    \item[\textbullet] The image of $\kappa$ in $\bfk$ is $e$.
%    \item[\textbullet] The elements $\{\tau_u-\tau_v+a\kappa+b,\kappa-c;a,b\in\bbZ,c\in \bbQ,u\ne v\}$ are pairwise coprime.
%    \item[\textbullet] $R$ is complete.
%\end{itemize}
Let $R$ be a complete local deformation ring with residue field $\bfk$. Consider the elements $q_{e}=\exp(2\pi \sqrt{-1}/\kappa)$ and $q_{e+1}=\exp(2\pi \sqrt{-1}/\overline\kappa)$ in $R$. These elements specialize to $\zeta_{e}=\exp(2\pi \sqrt{-1}/e)$ and $\zeta_{e+1}=\exp(2\pi \sqrt{-1}/(e+1))$ in $\bfk$.

\subsection{The choice of $\calF$}
\label{ch3:subs_param-Hecke}
From now on we assume that $R$ is a complete local deformation ring in general position with residue field $\bfk$ and field of fractions $K$. In this section we define some special choice of the set $\calF$. This choice of parameters is particularly interesting because it is related with the categorical action on the category $\calO$ for $\widehat{\mathfrak{gl}}_N$, see \cite{RSVV}.

Fix a tuple $\nu=(\nu_1,\cdots,\nu_l)\in\bbZ^l$. Put $Q_r=\exp(2\pi \sqrt{-1}(\nu_r+\kappa_r)/\kappa)$ for $r\in[1,l]$.
The canonical homomorphism $R\to \bfk$ maps $q_e$ to $\zeta_e$ and $Q_r$ to
$\zeta_e^{\nu_r}$.

Now, consider the subset $\scrF$ of $R$ given by
$$
%\label{ch3:eq_F(q,Q)}
\scrF=\bigcup_{r\in \bbZ,t\in[1,l]}\{q_e^rQ_t\}.
$$
Denote by $\calF_\bfk$ the image of $\calF$ in $\bfk$ with respect to the surjection $R\to \bfk$. Recall from Section \ref{ch3:subs_isom-KLR-Hecke-gen} that we consider $\calF$ (and $\calF_\bfk)$ as a vertex set of a quiver. The set $\calF$ is a vertex set of a quiver that is a disjoint union if $l$ infinite linear quivers. The set $\calF_\bfk$ is a vertex set of a cyclic quiver of length $e$.

Fix $k\in[0,e-1]$. To this $k$ we associate a map $\Upsilon\colon \bbZ\to\bbZ$ as in (\ref{ch3:eq_upsilon}).
Now, consider the tuple
$$
%\label{ch3:eq_nu-bar}
\overline\nu=(\overline\nu_1,\cdots,\overline\nu_l)\in\bbZ^l,\qquad \overline\nu_r=\Upsilon(\nu_r)~\forall r\in[1,l].
$$
Let $\overline R$ be as in the previous section. Let $\overline\bfk$ and $\overline K$ be the residue field and the field of fractions of $\overline R$ respectively.
Now, consider $\overline Q=(\overline Q_1,\cdots,\overline Q_l)$, where $\overline Q_r=\exp(2\pi \sqrt{-1}(\overline\nu_r+\overline\kappa_r)/\overline\kappa)$ and $\overline\kappa$ and $\overline\kappa_r$ are defined in Section \ref{ch3:subs_def-ring}. 
Consider the subset $\overline\scrF$ of $\overline R$ given by
$$
\overline\scrF=\bigcup_{r\in \bbZ,t\in[1,l]}\{q_{e+1}^r\overline Q_t\}.
$$
Denote by $\overline{\calF}_{\overline\bfk}$ the image of $\overline{\calF}$ in $\overline\bfk$ with respect to the surjection $\overline R\to \overline\bfk$. The set $\overline\calF$ is a vertex set of a quiver that is a disjoint union of $l$ infinite linear quivers. The set $\overline\calF_{\overline\bfk}$ is a vertex set of a cyclic quiver of length $e+1$.

\subsection{Algebras $\widehat H$, $\widehat{SH}$, $\widehat R$ and $\widehat{S}$}

Let $\Gamma=(I,H)$, $\overline\Gamma=(\overline I,\overline H)$ and  $\widetilde\Gamma=(\widetilde I,\widetilde H)$ be as in Section \ref{ch3:subs_not-e-e+1}.
%As in previous section, we fix $\nu=(\nu_1,\cdots,\nu_l)$ and 

We will use the notation $\calF$, $\calF_\bfk$, $\overline\calF$ and $\overline\calF_{\overline\bfk}$ as in previous section. (In particular, we fix some $\nu=(\nu_1,\cdots,\nu_l)$.)

We have the following isomorphisms of quivers
$$
%\label{ch3:eq_p-I-tilde}
\widetilde I\simeq \scrF,\quad i=(a,b)\mapsto p_{i}:=\exp(2\pi \sqrt{-1} (a+\kappa_b)/\kappa),
$$
$$
%\label{ch3:eq_p-I-tilde-bar}
\widetilde I\simeq \overline\scrF,\quad i=(a,b)\mapsto \overline p_i:=\exp(2\pi \sqrt{-1} (a+\overline\kappa_b)/\overline\kappa),
$$
$$
%\label{ch3:eq_p-I}
I\simeq \scrF_\bfk,\quad i\mapsto p_{i}:=\zeta_e^i,
$$
$$
%\label{ch3:eq_p-I-bar}
\overline I\simeq \overline\scrF_{\overline\bfk},\quad i\mapsto \overline p_i:=\zeta_{e+1}^i.
$$
%where $Q$ and $\overline Q$ are as in Section \ref{ch3:subs_param-Hecke}, $a\in\bbZ$, $b\in[1,l]$. Notice that $Q_r=p_{(\nu_r,r)}$ and $\overline Q_r=\overline p_{(\overline\nu_r,r)}$ for each $r$.

These isomorphisms yield the following commutative diagrams
$$
\begin{CD}
\widetilde I @>{\sim}>> \scrF\\
@V{\pi_e}VV              @VVV\\
I            @>{\sim}>> \scrF_\bfk,
\end{CD}
\qquad\qquad
\begin{CD}
\widetilde I @>{\sim}>> \overline\scrF\\
@V{\pi_{e+1}}VV              @VVV\\
\overline I            @>{\sim}>> \overline\scrF_{\overline\bfk}.
\end{CD}
$$

We will identify
$$
%\label{ch3:eq_ident-quivers}
I\simeq \scrF_\bfk,\quad \overline I\simeq \overline\scrF_{\overline\bfk}, \quad \widetilde I\simeq \scrF, \quad \widetilde I\simeq \overline\scrF
$$
as above.

%{\color{red}
%Recall from Section \ref{ch3:subs_Hecke} that the cyclotomic Hecke algebra $H^\nu_{d,\bfk}(\zeta_e)$ contains some idempotents $e(\ui)$ for $\ui\in I^d$. %Similarly we have $e(\ui')\in H^{\overline\nu}_{d,\bfk}(\zeta_{e+1})$ for $\ui'\in \overline I^d$.
%These idempotents lift to idempotents in the $R$-algebra $H^\nu_{d,R}(q_e)$
%because this algebra is free over $R$ by \cite[Thm.~2.2]{Mat} and $R$ is a local algebra, see, e.g., \cite[Ex.~6.16]{CR}. We also denote these idempotents $e(\ui)$.

%By base change we have the $K$-algebra $H^{\nu}_{d,K}(q_{e})$
%which contains some idempotents $e(\uj)\in H^\nu_{d,K}(q_e)$ for $\uj\in \widetilde I^d$.
%The idempotent
%$e(\ui)\in H^\nu_{d,R}(q_e)$
%decomposes in $H^\nu_{d,K}(q_e)$
%in the following way
%$$
%e(\ui)=\sum_{\uj\in \widetilde I^d,\pi_e(\uj)=\ui}e(\uj).%\qquad %e(\ui')=\sum_{\uj'\in \widetilde I^d,\pi_{e+1}(\uj')=\ui'}e(\uj').
%$$

%We have the decompositions
%$$
%H^\nu_{d,R}(q_e)=\bigoplus_{\alpha\in Q^+_{I},|\alpha|=d}H^\nu_{\alpha,R}(q_e),
%$$
%$$
%H^\nu_{d,K}(q_e)=\bigoplus_{\alpha\in Q^+_{I},|\alpha|=d}H^\nu_{\alpha,K}(q_e),
%$$
%where
%$$
%H^\nu_{\alpha,K}(q_e)=\bigoplus_{\widetilde\alpha\in Q^+_{\widetilde I},\pi_e(\widetilde\alpha)=\alpha}H^\nu_{\widetilde\alpha,K}(q_e).
%$$
%}

Our goal is to obtain an analogue of Theorem \ref{ch3:thm_KLR-e-e+1} over the ring $R$.
First, consider the algebras $\widehat H_{d,\bfk}(\zeta_e)$ and $\widehat H_{d,K}(q_e)$ defined in the same way as in Section \ref{ch3:subs_isom-KLR-Hecke-gen} with respect to the sets $\scrF_\bfk\subset \bfk$ and $\scrF\subset K$. We can consider the $R$-algebra $\widehat H_{d,R}(q_e)$ defined in a similar way with respect to the same set of idempotents as $\widehat H_{d,\bfk}(\zeta_e)$ (i.e., with respect to the set $\calF_\bfk$, not $\calF$). %\todo{Can do this automatically. In fact, this has no sence any more.} Sometimes we will write $\widehat H^\nu_{d,\bfk}(\zeta_e)$, $\widehat H^\nu_{d,R}(q_e)$ and $\widehat H^\nu_{d,K}(q_e)$ to specify the parameter $\nu$ (the $l$-tuple $Q$ in Section \ref{ch3:subs_param-Hecke} depends on the $l$-tuple $\nu$). 

The algebra $\widehat H_{d,K}(q_e)$ is not unitary because the quiver $\widetilde \Gamma$ is infinite. To avoid this problem we consider the truncated version of this algebra. Let $\widehat H^{\leqslant N}_{d, K}(q_{e})$ be the quotient of $\widehat H_{d, K}(q_{e})$ by the two-sided ideal generated by the idempotents $e(\uj)\in \widetilde I^d$ such that $\uj$ contains a component that is not a vertex of the truncated quiver ${\widetilde \Gamma}^{\leqslant N}$ (see Section \ref{ch3:subs_not-e-e+1}).  (In fact, the algebra $\widehat H^{\leqslant N}_{d, K}(q_{e})$ is isomorphic to a direct summand of $\widehat H_{d, K}(q_{e})$).

Similarly, we define the algebras $\widehat H_{d,\overline \bfk}(\zeta_{e+1})$, $\widehat H_{d,\overline K}(q_{e+1})$ and $\widehat H_{d,\overline R}(q_{e+1})$ using the sets $\overline\calF$ and $\overline\calF_{\overline\bfk}$ instead of $\calF$ and $\calF_\bfk$. We  define a truncation $\widehat H^{\leqslant N}_{d,\overline K}(q_{e+1})$ of $\widehat H_{d,\overline K}(q_{e+1})$ using the quiver $\overline{\widetilde{\Gamma}}^{\leqslant N}$.

For each $\ui\in I^{d}$ we consider the following idempotent in $\widehat H^{\leqslant N}_{d, K}(q_{e})$:
$$
e(\ui)=\sum_{\uj\in \widetilde I^d,\pi_e(\uj)=\ui}e(\uj).%\qquad
$$  
Here we mean that $e(\uj)$ is zero if $\uj$ contains a vertex that is not in the truncated quiver ${\widetilde \Gamma}^{\leqslant N}$. The idempotent $e(\ui)$ is well-defined because only a finite number of terms in the sum are nonzero. For each $\ui\in \overline I^d$ we can define an idempotent $e(\ui)\in \widehat H^{\leqslant N}_{d,\overline K}(q_{e+1})$ in a similar way.

\begin{lem}
\label{ch3:lem-inj_R_K_H}
There is an injective algebra homomorphism $\widehat H_{d, R}(q_{e})\to \widehat H^{\leqslant N}_{d, K}(q_{e})$ such that $e(\ui)\mapsto e(\ui)$, $X_re(\ui)\mapsto X_re(\ui)$ and $T_re(\ui)\mapsto T_re(\ui)$.
\end{lem}
\begin{proof}

It is clear that we have an algebra homomorphism $\widehat H_{d, R}(q_{e})\to \widehat H^{\leqslant N}_{d, K}(q_{e})$ as in the statement. We only have to check the injectivity.

For each $w\in \frakS_d$ we have an element $T_w\in H_{d,R}(q)$ defined in the following way. We have $T_w=T_{i_1}\cdots T_{i_r}$, where $w=s_{i_1}\cdots s_{i_r}$ is a reduced expression. It is well-known that $T_w$ is independent of the choice of the reduced expression. Moreover, the algebra $H_{d,R}(q)$ is free over $R[X_1^{\pm 1},\cdots,X_d^{\pm 1}]$ with a basis $\{T_w;~w\in\frakS_d\}$.

%Since we have $\widehat H_{d, R}(q_{e})=A_{1,R}\otimes _$
Set
$$
B=\bigoplus_{\ui\in \scrF_\bfk^d}R[X_1^{\pm 1},\cdots,X_d^{\pm 1}][(X_r-X_t)^{-1},(q_eX_r-X_t)^{-1};~r\ne t]e(\ui),
$$
where we invert $(X_r-X_t)$ only if $i_r\ne i_t$ and we invert $(q_eX_r-X_t)$ only if $\zeta_ei_r\ne i_t$. We have $\widehat H_{d, R}(q_{e})=B\otimes_{R[X_1^{\pm 1},\cdots,X_d^{\pm 1}]} H_{d,R}(q_e)$. This implies that the $B$-module $\widehat H_{d, R}(q_{e})$ is free with a basis $\{T_w;~w\in\frakS_d\}$. 

Similarly, we can show that the algebra $\widehat H^{\leqslant N}_{d, K}(q_{e})$ is free (with a basis $\{T_w;~w\in\frakS_d\}$) over
$$
B'=\bigoplus_{\uj\in \scrF^d}K[X_1^{\pm 1},\cdots,X_d^{\pm 1}][(X_r-X_t)^{-1},(q_eX_r-X_t)^{-1};~r\ne t]e(\uj),
$$
where we invert $(X_r-X_t)$ only if $j_r\ne j_t$ and we invert $(q_eX_r-X_t)$ only if $q_ej_r\ne j_t$, and we take only $\uj$ that are supported on the vertices of the truncated quiver $\Gamma^{\leqslant N}$.

Now, the injectivity of the homomorphism follows from the fact that it takes a $B$-basis of $\widehat H_{d, R}(q_{e})$ to a $B'$-linearly independent set in $\widehat H^{\leqslant N}_{d, K}(q_{e})$.

\end{proof}

Now we define the algebra $\widehat{SH}_{\overline\alpha,\overline\bfk}(\zeta_{e+1})$ that is a Hecke analogue of a localization of the balanced KLR algebra $S_{\overline\alpha,\bfk}$. To do so, consider the idempotent $\bfe=\sum_{\ui\in \overline I^{\overline\alpha}_{\rm ord}}e(\ui)$ in $\widehat H_{\overline\alpha,\overline\bfk}(\zeta_{e+1})$. We set
$$
\widehat{SH}_{\overline\alpha,\overline\bfk}(\zeta_{e+1})=\bfe \widehat H_{\overline\alpha,\overline \bfk}(\zeta_{e+1})\bfe/\sum_{\uj\in \overline
I^{\overline\alpha}_{\rm un}}\bfe \widehat H_{\overline\alpha,\overline\bfk}(\zeta_{e+1})e(\uj)\widehat H_{\overline\alpha,\overline \bfk}(\zeta_{e+1})\bfe.
$$

%\todo{Why don't we use $d$ instead of $\alpha$?}

Now, we define a similar algebra over $K$. To do this, we need to introduce some additional notation. 
Denote by $Q^+_{\widetilde I,{\rm eq}}$ the subset of $Q^+_{\widetilde I}$ that contains only $\widetilde\alpha$ such that for each $k\in \widetilde I_1$, the dimension vector $\widetilde\alpha$ has the same dimensions at vertices $k^1$ and $k^2$. %Denote by $Q^+_{\widetilde I,{\rm neq}}$ the complement of $Q^+_{\widetilde I,{\rm eq}}$ in $Q^+_{\widetilde I}$, i.e., we have $Q^+_{\widetilde I}=Q^+_{\widetilde I,{\rm eq}}\coprod Q^+_{\widetilde I,{\rm neq}}$. Set 
%$$
%\widetilde I^{\overline\alpha}=\coprod_{\widetilde\alpha\in Q^+_{\widetilde I,{\rm eq}},\pi_{e+1}(\widetilde\alpha)=\overline\alpha}\widetilde I^{\widetilde\alpha}.
%$$
%\todo{Use neq somewhere?}

%\{\uj\in\widetilde I^d;~\pi_{e+1}(\uj)\in \overline I^{\overline\alpha}\}$.

%Note that, by definition, for $k\in I_1$, each sequence $\ui\in \overline I$ has the same number of $k^1$ and $k^2$. But a similar property can be false for an element $\uj\in \widetilde I^{\overline\alpha}$. This motivates the following notation. Denote by $\widetilde I^{\overline\alpha}_{\rm eq}$ the subset of $\widetilde I^{\overline\alpha}$ that contains only sequences $\uj$ such that for each $k\in \widetilde I_1$ the elements $k^1$ and $k^2$ have the same multiplicity in $\uj$. Denote by $\widetilde I^{\overline\alpha}_{\rm neq}$ the complementary of $\widetilde I^{\overline\alpha}_{\rm eq}$ in $\widetilde I^{\overline\alpha}$, i.e. we have $\widetilde I^{\overline\alpha}=\widetilde I^{\overline\alpha}_{\rm eq}\coprod \widetilde I^{\overline\alpha}_{\rm neq}$. Similarly, 

Set
$$
\widehat{H}^{\leqslant N}_{\overline\alpha,\overline K}(q_{e+1})=\bigoplus_{\pi_{e+1}(\widetilde\alpha)=\overline{\alpha}}\widehat{H}_{\widetilde\alpha,\overline K}(q_{e+1}),
$$
$$
\widehat{SH}^{\leqslant N}_{\overline\alpha,\overline K}(q_{e+1})=\bigoplus_{\pi_{e+1}(\widetilde\alpha)=\overline{\alpha}}\widehat{SH}_{\widetilde\alpha,\overline K}(q_{e+1}),
$$
where in the sums we take only $\widetilde \alpha\in Q^+_{\widetilde I, {\rm eq}}$ that are supported on the vertices of the truncated quiver $\overline{\widetilde\Gamma}^{\leqslant N}$ and $\widehat{SH}_{\widetilde\alpha,\overline K}(q_{e+1})$ is defined similarly to $\widehat{SH}_{\overline\alpha,\overline\bfk}(\zeta_{e+1})$. More precisely, we have
$$
\widehat{SH}_{\widetilde\alpha,\overline K}(q_{e+1})=
\widetilde\bfe_{\widetilde\alpha} H_{\widetilde\alpha,\overline K}(q_{e+1})\widetilde\bfe_{\widetilde\alpha}/\sum_{\uj\in \widetilde I^{\widetilde\alpha}_{\rm un}}\widetilde\bfe_{\widetilde\alpha} H_{\widetilde\alpha,\overline K}(q_{e+1})e(\uj) H_{\widetilde\alpha,\overline K}(q_{e+1})\widetilde\bfe_{\widetilde\alpha},
$$
where $\widetilde \bfe_{\widetilde\alpha}=\sum_{\uj\in \widetilde I^{\widetilde\alpha}_{\rm ord}}e(\uj)$.

\begin{rk}
Consider the following idempotents in $\widehat{H}^{\leqslant N}_{\overline\alpha,\overline K}(q_{e+1})$: 
$$\widetilde\bfe=\sum_{\pi_{e+1}(\widetilde\alpha)=\overline\alpha}\widetilde\bfe_{\widetilde\alpha}, \qquad\bfe=\sum_{\ui\in \overline I^{\overline\alpha}_{\rm ord}}e(\ui),
$$ 
where the first sum is taken only by $\widetilde\alpha\in Q^+_{\widetilde I, {\rm eq}}$. (Note that $\widehat{H}^{\leqslant N}_{\overline\alpha,\overline K}(q_{e+1})$ was defined as a quotient of $\widehat{H}_{\overline\alpha,\overline K}(q_{e+1})$. So, if $\widetilde\alpha$ is not supported on $\overline{\widetilde \Gamma}^{\leqslant N}$, then the idempotent $\widetilde\bfe_{\widetilde\alpha}$  is zero by definition. In particular, the sum has a finite number of nonzero terms.)
Set also $\widetilde I^{\overline\alpha}=\coprod_{\pi_{e+1}(\widetilde\alpha)=\overline\alpha}\widetilde I^{\widetilde \alpha}$, where the sum is taken only by $\widetilde\alpha\in Q^+_{\widetilde I, {\rm eq}}$. By definition, the algebra $\widehat{SH}^{\leqslant N}_{\overline\alpha,\overline K}(q_{e+1})$ is a quotient of $\widetilde\bfe\widehat{H}^{\leqslant N}_{\overline\alpha,\overline K}(q_{e+1})\widetilde\bfe$. But we can see this algebra as the same quotient of $\bfe\widehat{H}^{\leqslant N}_{\overline\alpha,\overline K}(q_{e+1})\bfe$ (we do the quotient with respect to the same idempotents). Indeed, the idempotent $\bfe$ is a sum of a bigger number of standard idempotents $e(\uj)$, $\uj\in \widetilde I^{\overline\alpha}$ than the idempotent $\widetilde\bfe$. More precisely, the idempotent $\widetilde\bfe$ is the sum all $e(\uj)$ such that $\uj$ is well-ordered while $\bfe$ is the sum of all $e(\uj)$ such that $\pi_{e+1}(\uj)$ is well-ordered.  
But each $\uj\in \widetilde I^{\overline\alpha}$ such that $\pi_{e+1}(\uj)$ is well-ordered and $\uj$ is not well-ordered must be unordered. Then such $e(\uj)$ becomes zero after taking the quotient.

\end{rk}

\smallskip
Finally, we define the $R$-algebra $\widehat{SH}^N_{\overline\alpha,\overline R}(q_{e+1})$ as the image in $\widehat{SH}^{\leqslant N}_{\overline\alpha,\overline K}(q_{e+1})$ of the following composition of homomorphisms
$$
\bfe\widehat{H}_{\overline\alpha,\overline R}(q_{e+1})\bfe\to\bfe\widehat{H}^{\leqslant N}_{\overline\alpha,\overline K}(q_{e+1})\bfe\to\widehat{SH}^{\leqslant N}_{\overline\alpha,\overline K}(q_{e+1}).
$$

The lemma below shows that the algebra $\widehat{SH}^N_{\overline\alpha,\overline R}(q_{e+1})$ is independent of $N$ for $N$ large enough. So, we can write simply $\widehat{SH}_{\overline\alpha,\overline R}(q_{e+1})$ instead of $\widehat{SH}^N_{\overline\alpha,\overline R}(q_{e+1})$ for $N$ large enough.

\smallskip
\begin{lem}
\label{ch3:lem-SH_R_indep_N}
Assume $N\geqslant 2d$. Then the algebra $\widehat{SH}^N_{\overline\alpha,\overline R}(q_{e+1})$ is independent of $N$.
\end{lem}

\begin{proof}
Denote by $J_N$ the kernel of $\bfe\widehat{H}_{\overline\alpha,\overline R}(q_{e+1})\bfe\to\widehat{SH}^{\leqslant N}_{\overline\alpha,\overline K}(q_{e+1})$. Take $M>N$. It is clear that we have $J_M\subset J_N$.

Let us show that we also have an opposite inclusion if $N\geqslant 2d$. We want to show that each element $x\in J_N$ is also in $J_M$. It is enough to show this for $x$ of the form $x=Xe(\ui)$, where $\ui\in I^{\overline\alpha}_{\rm ord}$ and $X$ is composed of the elements of the form $T_r$ and $X_r$.  Then $Xe(\ui)\in J_N$ means that the element $Xe(\uj)\in \widehat{SH}^{\leqslant N}_{\overline\alpha,\overline K}(q_{e+1})$ is zero for each $\uj\in \widetilde I^{\overline\alpha}$ supported on $\overline{\widetilde \Gamma}^{\leqslant N}$ such that $\pi_{e+1}(\uj)=\ui$. To show that we have $Xe(\ui)\in J_M$ we must check that the element $Xe(\uj)\in \widehat{SH}^{\leqslant M}_{\overline\alpha,\overline K}(q_{e+1})$ is zero for each $\uj\in \widetilde I^{\overline\alpha}$ supported on $\overline{\widetilde \Gamma}^{\leqslant M}$ such that $\pi_{e+1}(\uj)=\ui$.

Let $\widetilde\alpha\in Q^+_{\widetilde I,{\rm eq}}$ be such that $\uj\in \widetilde I^{\widetilde\alpha}$. It is clear that we can find $\widetilde\alpha'\in Q^+_{\widetilde I,{\rm eq}}$ supported on $\overline{\widetilde\Gamma}^{\leqslant 2d}$ such that we have an isomorphism $\widehat{H}_{\widetilde\alpha,\overline K}(q_{e+1})\simeq \widehat{H}_{\widetilde\alpha',\overline K}(q_{e+1})$ that induces an isomorphism $\widehat{SH}_{\widetilde\alpha,\overline K}(q_{e+1})\simeq \widehat{SH}_{\widetilde\alpha',\overline K}(q_{e+1})$ and such that this isomorphism preserves the generators $X_r$ and $T_r$ and sends the idempotent $e(\uj)$ to some idempotent  $e(\uj')$ such that $\uj'$ is supported on $\overline{\widetilde\Gamma}^{\leqslant 2d}$ and $\pi_{e+1}(\uj)=\pi_{e+1}(\uj')$. Then the element $Xe(\uj)\in \widehat{SH}^{\leqslant M}_{\overline\alpha,\overline K}(q_{e+1})$ is zero because $Xe(\uj')\in \widehat{SH}^{\leqslant M}_{\overline\alpha,\overline K}(q_{e+1})$ is zero. This implies $x\in J_M$. 

\end{proof}

\smallskip
Now we define the KLR versions of the algebras $\widehat{SH}_{\overline\alpha,\overline \bfk}(\zeta_{e+1})$ and $\widehat{SH}^{\leqslant N}_{\overline\alpha,\overline K}(q_{e+1})$.
As for the Hecke version, we denote by $\bfe$ the idempotent $\sum_{\ui\in \overline I^{\overline\alpha}_{\rm ord}}e(\ui)$ in $\widehat R_{\overline\alpha,\bfk}(\overline \Gamma)$. 
Set
$$
\widehat S_{\overline\alpha,\bfk}(\overline \Gamma)=\bfe\widehat R_{\overline\alpha,\bfk}(\overline \Gamma)\bfe/\sum_{\ui\in \overline I^{\overline\alpha}_{\rm un}}\bfe\widehat R_{\overline\alpha,\bfk}(\overline \Gamma)e(\ui)R_{\overline\alpha,\bfk}(\overline \Gamma)\bfe.
$$
%or $\widehat R_{\overline\alpha,K}(\overline {\widetilde\Gamma}^{\leqslant N})$. 
For each $\widetilde\alpha\in Q^+_{\widetilde I,{\rm eq}}$ we consider the idempotent $\widetilde \bfe_{\widetilde\alpha}=\sum_{\uj\in \widetilde I^{\widetilde\alpha}_{\rm ord}}e(\uj)$ in $\widehat R_{\widetilde\alpha,K}(\overline {\widetilde\Gamma})$.
%Set
%$$
%\widehat S_{\overline\alpha,K}(\overline {\widetilde\Gamma})=\widetilde{\bfe}_{\widetilde\alpha}\widehat R_{\overline\alpha,\bfk}(\overline {\widetilde\Gamma})\widetilde{\bfe}_{\widetilde\alpha}/\sum_{\ui\in \widetilde I^{\widetilde\alpha}_{\rm un}}\widetilde{\bfe}_{\widetilde\alpha}\widehat R_{\overline\alpha,\bfk}(\overline {\widetilde\Gamma})e(\ui)R_{\overline\alpha,\bfk}(\overline {\widetilde\Gamma})\widetilde{\bfe}_{\widetilde\alpha}.
%$$ 
Set
$$
\widehat S_{\overline\alpha,K}(\overline {\widetilde\Gamma}^{\leqslant N})=\bigoplus_{\pi_{e+1}(\widetilde\alpha)=\overline\alpha}\widehat S_{\widetilde\alpha,K}(\overline {\widetilde\Gamma}),
$$
where we take only $\widetilde \alpha\in Q^+_{\widetilde I,{\rm eq}}$ that are supported on the vertices of the truncated quiver $\overline{\widetilde\Gamma}^{\leqslant N}$
and
$$
\widehat S_{\widetilde\alpha,K}(\overline {\widetilde\Gamma})=\widetilde\bfe_{\widetilde\alpha}\widehat R_{\widetilde\alpha,K}(\overline {\widetilde\Gamma})\widetilde\bfe_{\widetilde\alpha}/\sum_{\uj\in \widetilde I^{\widetilde\alpha}_{\rm un}}\widetilde\bfe_{\widetilde\alpha}\widehat R_{\widetilde\alpha,K}(\overline {\widetilde\Gamma})e(\uj)R_{\widetilde\alpha,K}(\overline {\widetilde\Gamma})\widetilde\bfe_{\widetilde\alpha}.
$$

%{\color{red}
%\begin{rk}
%Two sets of idempotents
%\end{rk}

%$$
%\widehat S_{\overline\alpha,\bfk}(\widetilde \Gamma)=\bfe\widehat R_{\overline\alpha,\bfk}(\widetilde \Gamma)\bfe/\sum_{\ui\in \overline I^{\overline\alpha}_{\rm un}}\bfe\widehat R_{\overline\alpha,\bfk}(\widetilde \Gamma)e(\ui)R_{\overline\alpha,\bfk}(\widetilde \Gamma)\bfe
%$$
%}

\smallskip
\begin{rk}
\label{ch3:rk_4-isom-KLR-Hecke}

By Proposition \ref{ch3:prop-isom_Hekce-KLR-widehat} we have algebra isomorphisms
$$
\widehat R_{\alpha,\bfk}(\Gamma) \simeq \widehat H_{\alpha,\bfk}(\zeta_e),
\quad \widehat R_{\alpha,K}(\widetilde \Gamma^{\leqslant N})\simeq \widehat H^{\leqslant N}_{\alpha,K}(q_{e}),
$$
$$
\widehat R_{\overline\alpha,\bfk}(\overline \Gamma)\simeq \widehat H_{\overline\alpha,\overline\bfk}(\zeta_{e+1}), \quad \widehat R_{\overline\alpha,K}(\overline {\widetilde\Gamma}^{\leqslant N})\simeq \widehat H^{\leqslant N}_{\overline\alpha,\overline K}(q_{e+1}),
$$
from which we deduce the isomorphisms
$$
\widehat S_{\overline\alpha,\bfk}(\overline \Gamma)\simeq \widehat {SH}_{\overline\alpha,\overline\bfk}(\zeta_{e+1}), \quad \widehat S_{\overline\alpha,K}(\overline {\widetilde\Gamma}^{\leqslant N})\simeq \widehat {SH}^{\leqslant N}_{\overline\alpha,\overline K}(q_{e+1}).
$$
\end{rk}

\smallskip
We may use these isomorphisms without mentioning them explicitly. Using the identifications above between KLR algebras and Hecke algebras, a localization of the isomorphism in Theorem \ref{ch3:thm_KLR-e-e+1} yields an isomorphism
$$
%\label{ch3:eq_Phi-Hecke-k}
\Phi_{\alpha,\bfk}\colon \widehat H_{\alpha,\bfk}(\zeta_e)\to \widehat {SH}_{\overline\alpha,\overline\bfk}(\zeta_{e+1}).
$$
In the same way we also obtain an algebra isomorphism
$$
\Phi_{\widetilde\alpha,K}\colon \widehat H_{\widetilde\alpha,K}(q_e)\to \widehat{SH}_{\widetilde\phi(\widetilde\alpha),\overline K}(q_{e+1})
$$
for each $\widetilde\alpha\in Q^+_{\widetilde I}$.
Taking the sum over all $\widetilde\alpha\in Q^+_{\widetilde I}$ such that $\pi_{e}(\widetilde\alpha)=\alpha$ and such that $\widetilde\alpha$ is supported on the vertices of the truncated quiver $\widetilde \Gamma^{\leqslant N}$ yields an isomorphism
$$
%\label{ch3:eq_Phi-Hecke-K}
\Phi_{\alpha,K}\colon \widehat H^{\leqslant N}_{\alpha,K}(q_e) \to \widehat{SH}^{\leqslant N}_{\overline\alpha,\overline K}(q_{e+1}).
$$

\begin{lem}
The homomorphism $\bfe\widehat{H}_{\overline\alpha,\overline R}(q_{e+1})\bfe\to \bfe\widehat{H}_{\overline\alpha,\overline\bfk}(\zeta_{e+1})\bfe$
factors through a homomorphism $\widehat{SH}_{\overline\alpha,\overline R}(q_{e+1})\to \widehat{SH}_{\overline\alpha,\overline\bfk}(\zeta_{e+1})$.
\end{lem}

\begin{proof}[Proof]
In Section \ref{ch3:subs_pol-rep-BKLR} we constructed a faithful polynomial representation of $S_{\overline\alpha,\bfk}$. Let us call it $\Pol_\bfk$. It is constructed as a quotient of the standard polynomial representation of $\bfe R_{\overline\alpha,\bfk}\bfe$. After localization we get a faithful representation $\widehat\Pol_\bfk$ of $\widehat S_{\overline\alpha,\bfk}$. Thus the kernel of the algebra homomorphism  $\bfe \widehat R_{\overline\alpha,\bfk}\bfe\to \widehat S_{\overline\alpha,\bfk}$ is the annihilator of the representation $\widehat\Pol_\bfk$. We can transfer this to the Hecke side (because the isomorphism in Proposition \ref{ch3:prop-isom_Hekce-KLR-widehat} comes from the identification of the polynomial representations) and we obtain that the kernel of the algebra homomorphism  $\bfe \widehat H_{\overline\alpha,\overline\bfk}(\zeta_{e+1})\bfe\to \widehat {SH}_{\overline\alpha,\overline\bfk}(\zeta_{e+1})$ is the annihilator of the representation $\widehat\Pol_\bfk$. Similarly, we can characterize the kernel of the algebra homomorphism $\bfe \widehat H^{\leqslant N}_{\overline\alpha,\overline K}(q_{e+1})\bfe\to \widehat{SH}^{\leqslant N}_{\overline\alpha,\overline K}(q_{e+1})$ as the annihilator of a similar representation $\widehat\Pol^{\leqslant N}_K$.

The $K$-vector space $\widehat\Pol^{\leqslant N}_K$ has an $R$-submodule $\widehat\Pol_R$ stable by the action of $\bfe\widehat H_{\overline\alpha,\overline R}(q_{e+1})\bfe$ such that $\bfk\otimes_R \widehat\Pol_R=\widehat\Pol_\bfk$ and it is compatible with the algebra homomorphism $\bfe \widehat H_{\overline\alpha,\overline R}(q_{e+1})\bfe\to \bfe \widehat H_{\overline\alpha,\overline \bfk}(\zeta_{e+1})\bfe$. By definition of $\widehat{SH}_{\overline\alpha,\overline R}(q_{e+1})$ and the discussion above, the kernel of the algebra homomorphism $\bfe \widehat H_{\overline\alpha,\overline R}(q_{e+1})\bfe\to \widehat{SH}_{\overline\alpha,\overline R}(q_{e+1})$ is formed by the elements that act by zero on $\widehat\Pol^{\leqslant N}_K$ (we assume that $N$ is big enough). Thus each element of this kernel acts by zero on $\widehat\Pol_R$. This implies, that an element of the kernel of $\bfe \widehat H_{\overline\alpha,\overline R}(q_{e+1})\bfe\to \widehat{SH}_{\overline\alpha,\overline R}(q_{e+1})$ specializes to an element of the kernel of $\bfe \widehat H_{\overline\alpha,\overline \bfk}(\zeta_{e+1})\bfe\to \widehat{SH}_{\overline\alpha,\overline \bfk}(\zeta_{e+1})$. This proves the statement.
\end{proof}

\subsection{The deformation of the isomorphism $\Phi$}
\label{ch3:subs_deform-Phi}

\smallskip
\begin{prop}
\label{ch3:prop_morph-Phi-over-ring}
There is a unique algebra homomorphism $\Phi_{\alpha,R}\colon \widehat H_{\alpha,R}(q_e)\to \widehat{SH}_{\overline\alpha,R}(q_{e+1})$ such that the following diagram is commutative
$$
\begin{CD}
\widehat H_{\alpha,\bfk}(\zeta_e) @>{\Phi_{\alpha,\bfk}}>> \widehat{SH}_{\overline\alpha,\overline\bfk}(\zeta_{e+1})\\
@AAA                         @AAA\\
\widehat H_{\alpha,R}(q_e) @>{\Phi_{\alpha,R}}>> \widehat{SH}_{\overline\alpha,\overline R}(q_{e+1})\\
@VVV                         @VVV\\
\widehat H^{\leqslant N}_{\alpha,K}(q_e) @>{\Phi_{\alpha,K}}>> \widehat{SH}^{\leqslant N}_{\overline\alpha,\overline K}(q_{e+1}).\\
\end{CD}
$$
\end{prop}
\begin{proof}[Proof]
First we consider the algebras $H_{\alpha,\bfk}^{\rm loc}(\zeta_e)$, $H_{\alpha,R}^{\rm loc}(q_e)$ and $H_{\alpha,K}^{\rm loc,\leqslant N}(q_e)$ obtained from $\widehat H_{\alpha,\bfk}(\zeta_e)$, $\widehat H_{\alpha,R}(q_e)$ and $\widehat H^{\leqslant N}_{\alpha,K}(q_e)$ by inverting
\begin{itemize}
\item[\textbullet] $(X_r-X_t)$ and $(\zeta_eX_r-X_t)$ with $r\ne t$,
\item[\textbullet] $(X_r-X_t)$ and $(q_eX_r-X_t)$ with $r\ne t$,
\item[\textbullet] $(X_r-X_t)$ and $(q_eX_r-X_t)$ with $r\ne t$
\end{itemize}
respectively. Let $SH^{\rm loc}_{\overline\alpha,\overline\bfk}(\zeta_{e+1})$ and $SH^{\rm loc,\leqslant N}_{\overline\alpha,\overline K}(q_{e+1})$ be the localizations of $\widehat{SH}_{\overline\alpha,\overline\bfk}(\zeta_{e+1})$ and $\widehat{SH}^{\leqslant N}_{\overline\alpha,\overline K}(q_{e+1})$ such that the isomorphisms $\Phi_{\alpha,\bfk}$ and $\Phi_{\alpha,K}$ above induce isomorphisms
$$
\Phi_{\alpha,\bfk}\colon H^{\rm loc}_{\alpha,\bfk}(\zeta_e)\to{SH}^{\rm loc}_{\overline\alpha,\overline\bfk}(\zeta_{e+1})\\
$$
$$
\Phi_{\alpha,K}\colon H^{\rm loc,\leqslant N}_{\alpha,K}(q_e)\to{SH}^{\rm loc,\leqslant N}_{\overline\alpha,\overline K}(q_{e+1})\\.
$$
Let $SH^{\rm loc}_{\overline\alpha,\overline R}(q_{e+1})$ be the image in $SH^{\rm loc,\leqs N}_{\overline\alpha,\overline K}(q_{e+1})$ of the following composition of homomorphisms
$$
\bfe{H}^{\rm loc}_{\overline\alpha,\overline R}(q_{e+1})\bfe\to\bfe{H}^{\rm loc,\leqslant N}_{\overline\alpha,\overline K}(q_{e+1})\bfe\to{SH}^{\rm loc,\leqslant N}_{\overline\alpha,\overline K}(q_{e+1}).
$$
(We assume $N\geqslant 2d$. Then, similarly to Lemma \ref{ch3:lem-SH_R_indep_N}, the algebra $SH^{\rm loc}_{\overline\alpha,\overline R}$ is independent of $N$ under this assumption.)
%The relations between the idempotents and the other elements are as in Lemma \ref{ch3:lem_isom-Hecke-KLR-gen}.

%We define the algebras $SH_{\overline\alpha,\bfk}^{\rm loc}(\zeta_{e+1})$, $SH_{\overline\alpha,R}^{\rm loc}(q_{e+1})$ in a similar way by localizing with respect to $(X_r-X_t)^{-1}e(\ui)$ and $(qX_r-X_t)^{-1}e(\ui)$ only when $t\ne r+1$ or $t=r+1$ and $i_r\not\in \overline  I_1$. We define $SH_{\overline\alpha,K}^{\rm loc}(q_{e+1})$ similarly.

%The isomorphisms $\Phi_{\alpha,\bfk}$, $\Phi_{\alpha,K}$ in (\ref{ch3:eq_Phi-Hecke-k}) and (\ref{ch3:eq_Phi-Hecke-K}) give algebra isomorphisms
%$$
%\Phi_{\alpha,\bfk}\colon H_{\alpha,\bfk}^{\rm loc}(\zeta_e)\to SH_{\overline\alpha,\bfk}^{\rm loc}(\zeta_{e+1}),\qquad
%\Phi_{\alpha,K}\colon H_{\alpha,K}^{\rm loc}(q_e)\to SH_{\overline\alpha,K}^{\rm loc}(q_{e+1}).
%$$
%More precisely, each localized Hecke algebra introduced above is isomorphic to a ring of quotients of a KLR algebra. Futhermore, the algebra isomorphisms  $\Phi_{\alpha,\bfk}$, $\Phi_{\alpha,K}$ are induced from the algebra isomorphisms between KLR algebras given by localization of the isomorphism in Theorem \ref{ch3:thm_KLR-e-e+1}. See the proof of Lemma \ref{ch3:lem_isom-Hecke-KLR-gen} for more details.

Next, we want to prove that there exists an algebra homomorphism $\Phi_{\alpha,R}\colon H_{\alpha,R}^{\rm loc}(q_e)\to SH_{\overline\alpha,\overline R}^{\rm loc}(q_{e+1})$ such that the following diagram is commutative:

\begin{equation}
\label{ch3:eq_diag-localizations}
\begin{CD}
H^{\rm loc}_{\alpha,\bfk}(\zeta_e) @>{\Phi_{\alpha,\bfk}}>> SH^{\rm loc}_{\overline\alpha,\overline\bfk}(\zeta_{e+1})\\
@AAA                         @AAA\\
H^{\rm loc}_{\alpha,R}(q_e) @>{\Phi_{\alpha,R}}>> SH^{\rm loc}_{\overline\alpha,\overline R}(q_{e+1})\\
@VVV                         @VVV\\
H^{\rm loc,\leqslant N}_{\alpha,K}(q_e) @>{\Phi_{\alpha,K}}>> SH^{\rm loc,\leqslant N}_{\overline\alpha,\overline K}(q_{e+1}).\\
\end{CD}
\end{equation}

\vspace{1cm}

%We need to check that for each genetator $x$
We just need to check that the map $\Phi_{\alpha,K}$ takes an element of $H^{\rm loc}_{\alpha,R}(q_e)$ to an element of $SH^{\rm loc}_{\overline\alpha,\overline R}(q_{e+1})$ and that it specializes to the map $\Phi_{\alpha,\bfk}\colon H^{\rm loc}_{\alpha,\bfk}(\zeta_{e})\to SH^{\rm loc}_{\overline\alpha,\overline\bfk}(\zeta_{e+1})$. We will check this on the generators $e(\ui)$, $X_re(\ui)$ and $\Psi_re(\ui)$ of $H^{\rm loc}_{\alpha,R}(q_e)$.

This is obvious for the idempotents $e(\ui)$.

Let us check this for $X_re(\ui)$.
Assume that $\ui\in I^\alpha$ and $\uj\in \widetilde I^{|\alpha|}$ are such that we have $\pi_e(\uj)=\ui$. Write $\ui'=\phi(\ui)$ and $\uj'=\widetilde\phi(\uj)$. Set $r'=r'_\uj=r'_\ui$, see the notation in Section \ref{ch3:subs_comp-pol-reps}.
By Theorem \ref{ch3:thm_KLR-e-e+1} and Proposition \ref{ch3:prop-isom_Hekce-KLR-loc} we have
$$
\Phi_{\alpha,K}(X_{r}e(\uj))={\overline p}_{j'_{r'}}^{-1}p_{j_r}X_{r'}e(\uj').
$$
Since, $\overline p_{j'_{r'}}^{-1}p_{j_r}$ depends only on $\ui$ and $r$ and $e(\ui)=\sum_{\pi_e(\uj)=\ui}e(\uj)$, we deduce that
$$
%\begin{array}{rcl}
\Phi_{\alpha,K}(X_{r}e(\ui))=\overline p_{j'_{r'}}^{-1}p_{j_r}X_{r'}e(\ui').
%\end{array}
$$
Thus the element $\Phi_{\alpha,K}(X_{r}e(\ui))$ is in $SH^{\rm loc}_{\overline\alpha,R}$ and its image in $SH^{\rm loc}_{\overline\alpha,\bfk}$ is $\overline p_{i'_{r'}}^{-1}p_{i_r}X_{r'}e(\ui')=\Phi_{\alpha,\bfk}(X_{r}e(\ui))$.

Next, we consider the generators $\Psi_re(\ui)$.
We must prove that for each $\uj$ such that $\pi_e(\uj)=\ui$ and for each $r$ we have
\begin{itemize}
    \item[\textbullet] $\Phi_{\alpha,K}(\Psi_re(\uj))=\Xi e(\uj')$, for some element $\Xi\in H^{\rm loc}_{\alpha,R}(q_e)$ that depends only on $r$ and $\ui$,
    \item[\textbullet] the image of $\Xi e(\ui')$ in $SH^{\rm loc}_{\overline\alpha,\overline\bfk}(q_{e+1})$ under the specialization $R\to\bfk$ is $\Phi_{\alpha,\bfk}(\Psi_{r}e(\ui))$.
\end{itemize}
This follows from Lemma \ref{ch3:lem_morph-Phi-intert-op}.

Now we obtain the diagram from the claim of Proposition \ref{ch3:prop_morph-Phi-over-ring} as the restriction of the diagram (\ref{ch3:eq_diag-localizations}).
\end{proof}

\subsection{Alternalive definition of a categorical representation}

There is an alternative definition of a categorical representation, where the KLR algebra is replaced by the affine Hecke algebra. 

Let $R$ be a $\bbC$-algebra. Fix an invertible element $q\in R$, $q\ne 1$. Let $\calC$ be an $R$-linear exact category.

\smallskip
\begin{df}
A \emph{representation datum} in $\calC$ is a tuple $(E,F,X,T)$ where $(E,F)$ is a pair of exact biadjoint functors $\calC\to\calC$ and $X\in\End(F)^{\rm op}$ and $T\in\End(F^2)^{\rm op}$ are endomorphisms of functors such that
for each $d\in\bbN$, there is an $R$-algebra homomorphism $\psi_d\colon H_{d,R}(q)\to \End(F^d)^{\rm op}$ given by
$$
\begin{array}{ll}
X_r\mapsto F^{d-r}XF^{r-1} &\forall r\in[1,d],\\
T_r\mapsto F^{d-r-1}TF^{r-1} &\forall r\in[1,d-1].
\end{array}
$$
 \end{df}

\smallskip
Now, assume that $R=\bfk$ is a field. Assume that $\calC$ is a $\Hom$-finite $\bfk$-linear abelian category. Let $\scrF$ be a subset of $\bfk^\times$ (possibly infinite). %such that $q^\bbZ\scrF/q^\bbZ$ is finite. 
As in Section \ref{ch3:subs_isom-KLR-Hecke-gen}, we view $\scrF$ as the vertex set of a quiver with an arrow $i\to j$ if and only if $j=qi$.

\smallskip
\begin{df}
\label{ch3:def-categ_action-Hecke}
A $\mathfrak{g}_{\scrF}$-categorical representation in $\calC$ is the datum of a representation datum $(E,F,X,T)$ and a decomposition $\calC=\bigoplus_{\mu\in X_\scrF}\calC_\mu$ satisfying the conditions $(a)$ and $(b)$ below. For $i\in\scrF$, let $E_i$ and $F_i$ be endofunctors of $\calC$ such that for each $M\in\calC$ the objects $E_i(M)$ and $F_i(M)$ are the generalized $i$-eigenspaces of $X$ acting on $E(M)$ and $F(M)$ respectively, see also Remark \ref{ch3:rk_df-repl-F-by-E} $(a)$.
We assume
\begin{itemize}
    \item[$(a)$] $F=\bigoplus_{i\in\scrF}F_i$ and $E=\bigoplus_{i\in\scrF}E_i$,
%    \item[\textbullet] the action of $E_i$, $F_i$, $i\in\scrF$ on $[\calC]$ gives an integrable representation of $\mathfrak{sl}_\scrF$,
    \item[$(b)$] $E_i(\calC_\mu)\subset \calC_{\mu+\alpha_i}$ and  $F_i(\calC_\mu)\subset \calC_{\mu-\alpha_i}$.
\end{itemize}
If the set $\scrF$ is infinite, condition $(a)$ should be understood in the same way as in Remark \ref{ch3:rk_df-repl-F-by-E} $(b)$.
\end{df}

\begin{rk}
\label{ch3:rk_cat-res-loc-H}
$(a)$ 
%Fix $Q=(Q_1,\cdots,Q_l)\in (\bfk^\times)^l$. Assume that we have $\calF=\calF(q,Q)$ (see (\ref{ch3:eq_F(q,Q)})). In this case 
%Since the category $\calC$ is assumed to be $\Hom$-finite, 
By definition, for each object $M\in \calC$ and each $d\in\mathbb{Z}_{\geqslant 0}$, we have $F_{i_d}\cdots F_{i_1}(M)\ne 0$ only for a finite number of sequences $(i_1,\cdots,i_d)\in \calF^d$. (Else, the endomorphism algebra of $F^d(M)$ is infinite-dimensional.) Then the homomorphism $H_{d,\bfk}(q)\to \End(F^d(M))^{\rm op}$ extends to a homomorphism  $\widehat H_{d,\bfk}(q)\to \End(F^d(M))^{\rm op}$ such that only a finite number of idempotents $e(\uj)$ has a nonzero image. (We define the action of $e(\ui)$ as the projection from $F^d$ to $F_{i_d}\ldots F_{i_1}$. Note that the action of $(X_r-X_t)^{-1}e(\ui)$ such that $i_r\ne i_t$ is well-defined because $X_r$ and $X_t$ have different eigenvalues. Similarly, the action of $(qX_r-X_t)^{-1}e(\ui)$ such that $r\ne t$ and $qi_r\ne i_t$ is well-defined.) In particular, we obtain a homomorphism $\widehat H_{d,\bfk}(q)\to \End(F^d)^{\rm op}$.

%\todo{Truncation is not defined!}

%The homomorphism $\psi_d\colon H_{d,\bfk}(q)\to \End(F^d)^{\rm op}$ extends to a homomorphism $\widehat H_{d,\bfk}(q)\to \End(F^d)^{\rm op}$, where $\widehat H_{d,\bfk}(q)$ is as in Section \ref{ch3:subs_isom-KLR-Hecke-gen}.

$(b)$ As in part $(a)$, if we have a categorical representation in the sense of Definition \ref{ch3:def-categ_action-KLR}, then the homomorphism $R_{d,\bfk}\to \End(F^d)^{\rm op}$ extends to a homomorphism $\widehat R_{d,\bfk}\to \End(F^d)^{\rm op}$. Then Proposition \ref{ch3:prop-isom_Hekce-KLR-widehat} impies that the two definitions of a categorical representation of $\frakg_\calF$ (Definition \ref{ch3:def-categ_action-KLR} and Definition \ref{ch3:def-categ_action-Hecke}) are equivalent.
\end{rk}

\subsection{Categorical representations over $R$}

We assume that the ring $R$ is as in Section \ref{ch3:subs_param-Hecke}. We are going to obtain an analogue of Theorem \ref{ch3:thm_categ-e-e+1} over $R$.

Let $\calC_R$, $\calC_\bfk$ and $\calC_K$ be $R$-, $\bfk$- and $K$-linear categories, respectively. Assume that $\calC_\bfk$ and $\calC_K$ are $\Hom$-finite $\bfk$-linear and $K$-linear abelian categories, respectively. Assume that the category $\calC_R$ is exact. Fix $R$-linear functors $\Omega_\bfk\colon \calC_R\to\calC_\bfk$ and $\Omega_K\colon \calC_R\to\calC_K$.

\begin{rk}
The first example of a situation as above that we should imagine is the following. Let $A$ be an $R$-algebra that is finitely generated as an $R$-module. We set $\calC_R=\mod(A)$, $\calC_\bfk=\mod(\bfk\otimes_R A)$, $\calC_K=\mod(K\otimes_R A)$, $\Omega_\bfk=\bfk\otimes\bullet$ and $\Omega_K=K\otimes\bullet$.

Another interesting situation (that in fact motivated the result of this section) is when $\calC_B$, for $B\in\{R,\bfk,H\}$, is the category $\calO$ for $\widehat{\mathfrak{gl}}_N$ over $B$ at a negative level. We do not want to assume in this section that the category $\calC_R$ is abelian because \cite{RSVV} constructs a categorical representation only in the $\Delta$-filtered category $\calO$ over $R$ (and not in the whole abelian category $\calO$ over $R$). 
\end{rk}
 
\begin{df}
\label{ch3:def-categ_rep_R}
A categorical representation of $(\widetilde{\mathfrak{sl}}_{e},\mathfrak{sl}_{\infty}^{\oplus l})$  in $(\calC_R,\calC_\bfk,\calC_K)$ is the following data:

\begin{itemize}
%\item[$(2)$] decompositions $E=\bigoplus_{i=0}^{e-1}E_i$ and $F=\bigoplus_{i=0}^{e-1}F_i$,
\item[$(1)$] a categorical representation of $\frakg_I=\widetilde{\mathfrak{sl}}_e$ in $\calC_\bfk$,
\item[$(2)$] a categorical representation of $\frakg_{\widetilde I}=\mathfrak{sl}_{\infty}^{\oplus l}$ in $\calC_K$,
\item[$(3)$] a representation datum $(E,F)$ in $\calC_R$ (with respect to the Hecke algebra $H_{d,R}(q_e)$) such that the functors $E$ and $F$ commute with $\Omega_\bfk$ and $\Omega_K$,
\item[$(4)$] lifts (with respect to $\Omega_\bfk$) of decompositions $E=\bigoplus_{i\in I}E_i$, $F=\bigoplus_{i\in I}F_i$ and $\calC_\bfk=\bigoplus_{X_I} \calC_{\bfk,\mu}$ from $\calC_\bfk$ to $\calC_R$
\end{itemize}
such that the following compatibility conditions are satisfied.

\begin{itemize}
%\item The functors $E_i$, $F_i$ for $\calC_\bfk$ are obtained by base change from the functors $E_i$, $F_i$ for $\calC_R$ and actions of the Hecke algebra on $\End(F^d)^{\rm op}$ are compatible.

\item The decomposition $\calC_R=\bigoplus_{\mu\in X_e} \calC_{R,\mu}$ is compatible with the decomposition $\calC_K=\bigoplus_{\widetilde\mu\in X_{\widetilde I}} \calC_{K,\widetilde\mu}$ (i.e., we have $\Omega_K(\calC_{R,\mu})\subset\bigoplus_{\pi_e(\widetilde\mu)=\mu}\calC_{K,\widetilde\mu}$).

\item The decompositions $E=\bigoplus_{i\in I}E_i$ and $F=\bigoplus_{i\in I}F_i$ in $\calC_R$ are compatible with the decompositions $E=\bigoplus_{j\in \widetilde I}E_j$ and $F=\bigoplus_{j\in \widetilde I}F_j$ in $\calC_K$ with respect to $\Omega_K$ (i.e., the functors $E_i=\bigoplus_{j\in \widetilde I,\pi_e(j)=i}E_j$ and $F_i=\bigoplus_{j\in \widetilde I,\pi_e(j)=i}F_j$ for $\calC_K$ correspond to the functors $E_i$, $F_i$ for $\calC_R$).

\item The actions of the Hecke algebras $H_{d,R}(q_e)$, $H_{d,\bfk}(\zeta_e)$ and $H_{d,K}(q_e)$ on $\End(F^d)^{\rm op}$ for $\calC_R$, $\calC_\bfk$ and $\calC_K$ are compatible with $\Omega_\bfk$ and $\Omega_K$. %\todo{Have to truncate}

%\todo{Base change has sense for functors?}

\end{itemize}

\end{df}

Proposition \ref{ch3:prop_morph-Phi-over-ring} yields the following version of Theorem \ref{ch3:thm_categ-e-e+1} over $R$.

\bigskip
Let $(\overline\calC_R,\overline\calC_\bfk,\overline\calC_K)$ be a categorical representation of $(\widetilde{\mathfrak{sl}}_{e+1},\mathfrak{sl}_{\infty}^{\oplus l})$. Assume that for each $\mu\in X_{\overline I}\backslash X^+_{\overline I}$ we have $\overline\calC_{\bfk,\mu}=\overline\calC_{R,\mu}=0$ and the for each $\widetilde\mu\in X_{\widetilde I}\backslash X^+_{\widetilde I}$ we have $\overline\calC_{K,\widetilde\mu}=0$. Let $\calC_R$, $\calC_\bfk$ and $\calC_K$ be the subcategories of $\overline\calC_R$, $\overline\calC_\bfk$ and $\overline\calC_K$ defined in the same way as in Section \ref{ch3:subs_cat-lem}. Then we have the following.

\begin{thm}
%Assume that we have a categorical representation of $(\widetilde{\mathfrak{sl}}_{e+1},\mathfrak{sl}_{\infty}^{\oplus l})$ in $(\overline\calC_R,\overline\calC_\bfk,\overline\calC_K)$. 
There is a categorical representation of $(\widetilde{\mathfrak{sl}}_{e},\mathfrak{sl}_{\infty}^{\oplus l})$ in $(\calC_R,\calC_\bfk,\calC_K)$.
\end{thm}
\begin{proof}
We obtain a categorical representation of $\widetilde{\mathfrak{sl}}_e$ in $\calC_\bfk$ by Theorem \ref{ch3:thm_categ-e-e+1}. A similar argument as in the proof of Theorem \ref{ch3:thm_categ-e-e+1} yields a categorical representation of $\mathfrak{sl}_{\infty}^{\oplus l}$ in $\calC_K$ (we just have to replace the isomorphism $\Phi$ from section \ref{ch3:subs-isom_Phi} associated with the quivcer $\Gamma_e$ by a similar isomorphism associated with the quiver $\widetilde \Gamma$.) To construct a representation datum in $\calC_R$, we use the homomorphism $\Phi_{\alpha,R}$ from Proposition \ref{ch3:prop_morph-Phi-over-ring}. All axioms of a $(\widetilde{\mathfrak{sl}}_{e},\mathfrak{sl}_{\infty}^{\oplus l})$-categorical representation in $(\calC_R,\calC_\bfk,\calC_K)$ follow automatically from the axioms of a categorical representation of $(\widetilde{\mathfrak{sl}}_{e+1},\mathfrak{sl}_{\infty}^{\oplus l})$  in $(\overline\calC_R,\overline\calC_\bfk,\overline\calC_K)$.
\end{proof}

\section*{Acknowledgements}
I am grateful for the hospitality of the Max-Planck-Institut f\"ur Mathematik in Bonn, where a big part of this work is done.
I would like to thank \'Eric Vasserot for his guidance and helpful
discussions during my work on this paper. I would like to thank Alexander Kleshchev for useful discussions about KLR algebras. I would like to thank C\'edric Bonnaf\'e for his comments on an earlier version of this paper. I would also like to thank the anonymous reviewer for the careful reading and constructive comments.

\end{document}